\documentclass[11pt,reqno,a4paper]{amsart} 

\makeatletter
\def\@tocline#1#2#3#4#5#6#7{\relax
	\ifnum #1>\c@tocdepth % then omit
	\else
	\par \addpenalty\@secpenalty\addvspace{#2}%
	\begingroup \hyphenpenalty\@M
	\@ifempty{#4}{%
		\@tempdima\csname r@tocindent\number#1\endcsname\relax
	}{%
		\@tempdima#4\relax
	}%
	\parindent\z@ \leftskip#3\relax \advance\leftskip\@tempdima\relax
	\rightskip\@pnumwidth plus4em \parfillskip-\@pnumwidth
	#5\leavevmode\hskip-\@tempdima
	\ifcase #1
	\or\or \hskip 1em \or \hskip 2em \else \hskip 3em \fi%
	#6\nobreak\relax
	\dotfill\hbox to\@pnumwidth{\@tocpagenum{#7}}\par
	\nobreak
	\endgroup
	\fi}
\makeatother
\usepackage[utf8]{inputenc}
\usepackage{amsmath}
\usepackage{amsthm}
\usepackage{amssymb}
\usepackage{euscript}
\usepackage{mathrsfs}
\usepackage{wasysym}
\usepackage[all]{xy}
\usepackage{graphicx}
\usepackage{color}
\usepackage{multirow}
\usepackage{pifont}
\usepackage[table]{xcolor}
\usepackage{tikz-cd}
\usepackage{enumitem}
\usepackage{xcolor}
\usepackage{verbatim}
\usepackage{tikz-cd}
\usepackage[hidelinks]{hyperref}
\usepackage[babel]{csquotes}        %%%%%% References bibtex
\usepackage[maxnames=99,style=alphabetic,backend=bibtex]{biblatex}
\AtBeginBibliography{\scriptsize}
\usepackage{academicons}			%%%% ORCID
\def\orcid#1{\kern .08em\href{https://orcid.org/#1}{\includegraphics[keepaspectratio,width=0.7em]{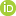}}}
\usepackage{chngcntr}
\counterwithout*{equation}{section}
\counterwithin{equation}{section}

\numberwithin{figure}{section}

\newtheorem*{mthm}{Main Theorem}
\newtheorem{theorem}{Theorem}[section]
\newtheorem{corollary}[theorem]{Corollary}
\newtheorem{lemma}[theorem]{Lemma}
\newtheorem{proposition}[theorem]{Proposition}

\newtheorem{notation}[theorem]{Notation}
\newtheorem{definition}[theorem]{Definition}

\theoremstyle{definition}

\newtheorem{example}[theorem]{Example}

\newtheorem{remark}[theorem]{Remark}

 \newcommand{\N}{{\mathbb N}}
\newcommand{\Z}{{\mathbb Z}} \newcommand{\R}{{\mathbb R}}
\newcommand{\Q}{{\mathbb Q}} \newcommand{\C}{{\mathbb C}}
\newcommand{\A}{{\mathbb A}} 
\newcommand{\G}{{\mathbb G}} 
\newcommand{\sph}{{\mathbb S}} 
\newcommand{\E}{{\mathbb E}} \newcommand{\Proy}{{\mathbb P}}

\newcommand{\reg}{{\mathcal R}} 
 
 \newcommand{\I}{{\mathcal I}}

\newcommand{\Ii}{\mc{I}}%{{\EuScript I}}

\newcommand{\Nn}{{\EuScript N}}

\newcommand{\im}{\operatorname{im}}
\newcommand{\Reg}{\operatorname{Reg}}

\newcommand{\Int}{\operatorname{Int}}

\newcommand{\id}{\operatorname{id}}

\newcommand{\Zcl}{\textnormal{Zcl}}

\newcommand{\w}{\mr{w}}%\newcommand{\w}{{\tt w}}

\newcommand{\mc}{\mathcal}
\newcommand{\mr}{\mathrm}

\newcommand{\mscr}{\mathscr}

\newlist{mydescription}{description}{1}
\setlist[mydescription]{font=\normalfont\normalcolor\textsc, before*={\normalfont}}

\begin{document}
	%\title[]{Embedded $\Q$-desingularization of real Schubert varieties and application to the relative $\Q$-algebraicity problem}
	\title[]{A relative Nash-Tognoli theorem over $\mathbb{Q}$ and application to the $\mathbb{Q}$-algebraicity problem}

\author{Enrico Savi }%\orcid{0009-0000-9608-2343}}
\address{Laboratoire J. A. Dieudonné, Université Côte d'Azur, Parc Valrose, 28 Avenue Valrose, 06108 Nice (FRANCE).}
\email{enrico.savi@univ-cotedazur.fr, enrico.savi@univ-angers.fr}
\thanks{The author is supported by GNSAGA of INDAM}

\begin{comment}
	An algebraic set $V\subset\mathbb{R}^{n}$ of dimension $d$ is a $\mathbb{Q}$-nonsingular $\mathbb{Q}$-algebraic set if it can be described, both globally and locally, by polynomial equations with rational coefficients. Set $m:=\max\{n,2d+1\}$, if $V$ is compact, or $m:=n+2d+3$, otherwise. We prove that every nonsingular real algebraic set $V\subset\mathbb{R}^n$ with nonsingular algebraic subsets $\{V_i\}_{i=1}^\ell$ in general position is Nash diffeomorphic to a $\mathbb{Q}$-nonsingular $\mathbb{Q}$-algebraic set $V'\subset\mathbb{R}^m$ with $\mathbb{Q}$-nonsingular $\mathbb{Q}$-algebraic subsets $\{V'_i\}_{i=1}^\ell$ in general position and the Nash diffeomorphism $h:V\to V'$ sends each $V_i$ to $V_i'$. As a byproduct we obtain a $\mathbb{Q}$-algebraic relative version of Nash-Tognoli theorem for compact manifolds. A key result in the proof is the description of $\mathbb{Z}/2\mathbb{Z}$-homological cycles of real Grassmannian manifolds by $\mathbb{Q}$-nonsingular $\mathbb{Q}$-algebraic representatives via the Bott-Samelson resolution of real embedded Schubert varieties.
\end{comment}

\begin{abstract}
	We prove a relative version over $\mathbb{Q}$ of Nash-Tognoli theorem, that is: Let $M$ be a compact smooth manifold with closed smooth submanifolds $M_1,\dots,M_\ell$ in general position, then there exists a nonsingular real algebraic set $M'\subset\mathbb{R}^n$ with nonsingular algebraic subsets $M_1',\dots,M_\ell'$ and a diffeomorphism $h:M\to M'$ such that $h(M_i)=M_i'$ for all $i=1,\dots,\ell$ such that $M',M_1',\dots,M_\ell'$ are described, both globally and locally, by polynomial equations with rational coefficients. In addition, if $M,M_1,\dots,M_\ell$ are nonsingular algebraic sets, then we prove the diffeomorphism $h:M\to M'$ can be chosen semialgebraic and the result can be extended to the noncompact case. In the proof we describe also the $\mathbb{Z}/2\mathbb{Z}$-homological cycles of real embedded Grassmannian manifolds by nonsingular algebraic representatives over $\mathbb{Q}$ via the Bott-Samelson resolution of Schubert varieties.
\end{abstract}

% providing a complete positive answer to a relative version of an open problem asked by Parusi\'{n}ski.

\date{\today}

\maketitle 

%%%%%%%%%%

%\setcounter{tocdepth}{1}
\tableofcontents

%%%

\section*{Introduction}

\vspace{0.5em}

%Algebraic varieties defined over rational numbers is a concept of central interest in mathematics. ...

% This task is not easy at all.

\subsection*{General introduction}
One of the main topics in real algebraic geometry is the ``Algebraicity problem", that is, the characterization of those topological spaces admitting an algebraic structure. By Whitney embedding theorem \cite{Whi36}, every smooth manifold $M$ of dimension $d$ can be smoothly embedded in $\R^{2d+1}$ and the image $M'\subset\R^{2d+1}$ of such embedding can be described, both globally and locally, as the solution set of finitely many global real analytic equations. More in detail, the above global and local description means that $M'$ can be described as the common solution set of finitely many global real analytic functions and, locally at any point $p\in M'$, $M'$ coincides with the common solution set of $d+1$ analytic functions vanishing on $M'$ whose gradients are linearly independent at $p$. Thus, the further task to address was whether the previous global analytic equations could be produced algebraic, both globally and locally. Clearly, there are examples of non-compact manifolds that are not homeomorphic to any real algebraic set, such as any manifold with infinitely many connected components. We remark that here real algebraic set means the common solution set of real polynomial equations in some real affine space. In contrast, the algebraicity problem for compact smooth manifolds was a challenging task to address.

In his groundbreaking paper \cite{Na52}, Nash proved that every compact smooth manifold $M$ of dimension $d$ is diffeomorphic to a real analytic submanifold $M'\subset\R^{2d+1}$ which is actually the union of some nonsingular connected components of a real algebraic subset of $\R^{2d+1}$. Hence, Nash conjectured that $M'$ could be chosen to be a whole nonsingular real algebraic subset of $\R^{2d+1}$, a so-called \emph{algebraic model} of $M$. In 1973, Tognoli \cite{Tog73} proved this conjecture to be true, thus the Nash-Tognoli theorem asserts that: \textit{Every compact manifold $M$  of dimension $d$ is diffeomorphic to a nonsingular real algebraic subset of $\R^{2d+1}$.} With respect to Nash's result two main original ideas appear in Tognoli's work: relative algebraic approximation of smooth functions by polynomial (and regular) functions and cobordism theory, in particular the algebraic representatives of cobordism classes produced by Milnor in \cite{Mil65}.

After Tognoli's solution of Nash's conjecture a systematic study of real algebraic sets and of the algebraicity problem started. There is a wide literature devoted to improvements and extensions of Nash-Tognoli theorem. We refer the interested reader to the books \cite[\S II]{AK92a}, \cite[\S 14]{BCR98}, \cite[\S 6]{Man14}, the survey \cite[\S 2]{Kol17}, the papers \cite{AK81b,AK81a,CS92,BK89,Kuc11}, the more recent ones \cite{Ben1,GT17} and references therein. %In particular, a complete characterization of real algebraic sets with isolated singularities is developed in \cite{AK81a}.

\subsection*{The $\Q$-algebraicity problem} %At that point it is natural to ask whether the algebraic models of compact smooth manifolds can be further simplified by requiring that the coefficients of the describing polynomial equations belong to a subfield $k$ of $\R$ as small as possible. Since the rationals are the smallest subfield of $\R$, the final goal is $k=\Q$. 

At this point, both for applications and for theoretical reasons, a very interesting problem to investigate is the existence of models of smooth compact manifolds, or even models of singular sets, that can be described by polynomial equations with rational coefficients. It is evident how relevant it would be to succeed in this project for algorithmic algebraic geometry: a systematic and efficient description on $\Q$ of algebraic sets, both non-singular and singular, would allow exact calculations. In particular, for the purposes of applications, the description over $\Q$ of topological models of algebraic sets in low dimensions would be of particular interest in this project.  On the other hand, the existence of normal forms over $\Q$ (i.e. over $\Z$) up to homeomorphism for real algebraic sets allows the reduction modulo $p$ with potential applications of techniques from arithmetic geometry.

Let us first focus on describing models of compact smooth manifolds and real algebraic sets over $k$, where $k$ denotes a subfield of $\R$. In particular, the case $k=\Q$ corresponds to the previous problem and constitutes the hardest task to achieve since $\Q$ is the smallest subfield of $\R$. Denote by $\R_{\mr{alg}}$ the field of real algebraic numbers, or in other words, the real closure of $\Q$.

For compact smooth manifolds, above problem for $k=\R_{\mr{alg}}$ has an affirmative answer. This follows combining Nash's theorem and the algebraicity result \cite[Corollary\,3.9]{CS92} for Nash manifolds over any real closed field by Coste and Shiota. For $k=\Q$ a positive answer is given by Ghiloni and the author in \cite[Theorem\,1.7]{GSa}. Previously, Ballico and Tognoli \cite[Theorem\,0.1]{BT92} stated a result which is even stronger but, unfortunately, the proof of this theorem is not complete, so the result is not guaranteed to be true. In Remark \ref{rem:Q-tognoli} we outline why some arguments of the classical Nash-Tognoli theorem do not directly extend to obtain polynomial equations over $\Q$. In particular, we stress that several subtleties arise when looking for equations over $\Q$ so a different notion of regularity over $\mathbb{Q}$ (see Definition \ref{def:R|Q}) with respect to the one proposed in \cite[Definition\,3,\,p.\,30]{Tog78} and used in \cite{BT92} is required to fill the gaps in the proof. In particular, the theory of subfield-algebraic geometry recently developed by Fernando and Ghiloni \cite{FG} plays an essential role to clarify the basic concepts involved and to produce useful tools for our purposes. However, we recognize to Ballico and Tognoli the original idea of extending the algebraic approximation techniques to obtain nonsingular equations with coefficients over $\Q$ for the algebraic models of compact smooth manifolds.

%In the algebraic setting, a similar problem is to simplify the description of homeomorphic models of real algebraic sets in terms of the coefficients of the describing equations. 
A similar topological description over $k=\R_{\mr{alg}}$ holds for singular algebraic sets as well. %, where Nash diffeomorphisms have to be substituted by arc-analytic semialgebraic homeomorphisms. 
Indeed, by means of Zariski equisingular techniques, Parusi\'{n}ski and Rond \cite{PR20} proved that every real algebraic subset $V\subset\R^{n}$ can be deformed by an strongly equisingular semialgebraic and arc-wise analytic deformation $h_t:\R^n\to \R^n$ to an algebraic subset $V'\subset\R^n$ defined over $\R_{\textnormal{alg}}$. Remarkably, no regularity assumptions on the algebraic set $V$ are necessary. In addition, the polynomial equations describing the above model $V'$ can be chosen to have coefficients in $\Q$ if the extension of $\Q$ obtained by adding the coefficients of the polynomial equations defining $V$ is purely transcendental, see \cite[Remark\,13]{PR20}. However, in general, the fact that $\Q$ is not a real closed field is a crucial difficulty. Indeed, the following open problem proposed by Parusi\'{n}ski \cite{Par21} and motivated by Teissier \cite{Tei90} is widely open:

\vspace{.7em}

\noindent\textsc{$\Q$-algebraicity problem.}(\cite[Open\,problem\,1, p.\,199]{Par21}) \textit{Is every real algebraic set $V$ homeomorphic to a real algebraic set $V'$ defined by polynomial equations with rational coefficients?}

\vspace{.7em}

%We outline that both the techniques by Coste-Shiota and Parusi\'{n}ski-Rond can not be extended over $\Q$, indeed, in both cases, a key point is the application of model completeness of the theory of real closed fields. Unfortunately, $\Q$ is not a real closed field and the theory of ordered fields is not model complete. Hence, the unique approach which seems to work at the moment in investigating the \textsc{$\Q$-algebraicity problem} is the one proposed by Ghiloni and the author in \cite{GSa} and which is further developed in the present paper.

%We outline that both the techniques by Coste-Shiota and Parusi\'{n}ski-Rond can not be extended over $\Q$, since $\Q$ is not a real closed field. Hence, the unique approach which seems to work at the moment in investigating the \textsc{$\Q$-algebraicity problem} is the one proposed by Ghiloni and the author in \cite{GSa} and which is further developed in the present paper.

We outline that both the approaches by Coste-Shiota and Parusi\'{n}ski-Rond make use of the model completeness of the theory of real closed fields. That is the reason why their techniques can not be extended to rational numbers, indeed $\Q$ is not a real closed field and the theory of ordered fields is not model complete. Hence, the unique approach which seems to work at the moment in investigating the \textsc{$\Q$-algebraicity problem} is the one proposed by Ghiloni and the author in \cite{GSa} and which is further developed in the present paper.

In order to investigate latter open problem, as already mentioned, classical notions in algebraic geometry must be very carefully adapted when looking for descriptions over $\Q$. In the brand new paper \cite{FG}, Fernando and Ghiloni introduced and studied $\Q$-algebraic subsets of $\R^n$, that is, real algebraic sets globally described by polynomial equations with rational coefficients. %By means of the complexification $V_\C$ of a real $\Q$-algebraic set $V$ and the action of the Galois group $Gal(E|\Q)$ on $V_\C$, the authors prove important properties of real $\Q$-algebraic sets and define a precise notion of $\R|\Q$-nonsingularity for a point of a real $\Q$-algebraic set (see Definition \ref{def:R|Q}). 
Let $V\subset\R^n$ be a $\Q$-algebraic set. Denote by $\overline{\Q}:=\R_{\mr{alg}}[i]$ the algebraic closure of $\Q$. Denote by $V_\C\subset\C^n$ the complexification of $V$ and choose some equations over $\R_{\mr{alg}}$ defining $V_\C$. Let $E|\Q$ be a finite Galois subextension of $\overline{\Q}|\Q$ containing all the coefficients of the chosen equations defining $V_\C$. By means of the action of the Galois group $\textnormal{Gal}(E|\Q)$ on $V_\C$, they proved fundamental properties of real $\Q$-algebraic sets and defined a precise notion of $\R|\Q$-regularity for a point of a real $\Q$-algebraic set (see Definition \ref{def:R|Q}). Latter notion of regularity over $\Q$ for a point on a real $\Q$-algebraic set turned out to be crucial to extend the algebraic approximation techniques developed in \cite{Na52}, \cite{Tog73} and \cite{AK81a}. Indeed, in \cite{GSa} Ghiloni and the author developed such $\Q$-algebraic approximation techniques and gave a complete solution of above \textsc{$\Q$-algebraicity problem} in the case of nonsingular real algebraic sets and real algebraic sets with isolated singularities. As already mentioned, in \cite[Theorem\,1.7]{GSa} Ghiloni and the author proved, as a special case, the following version over $\Q$ of Nash-Tognoli theorem: \textit{Every compact smooth manifold $M$  of dimension $d$ is diffeomorphic to a $\Q$-nonsingular real $\Q$-algebraic subset of $\R^{2d+1}$}.

In this paper we extend latter version over $\Q$ of Nash-Tognoli theorem to a relative setting (see Theorem \ref{thm:Q_tico_tognoli}) providing a general answer to a relative version of \cite[Open\,problem\,1, p.\,199]{Par21} in the nonsingular case, that is, to positively answer the following question.

\vspace{4em}

\noindent\textsc{Relative $\Q$-algebraicity problem.} \textit{Is every nonsingular real algebraic set $V$, with nonsingular algebraic subsets $\{V_i\}_{i=1}^\ell$, in general position, diffeomorphic to a nonsingular algebraic set $V'$, with nonsingular algebraic subsets $\{V'_i\}_{i=1}^\ell$, in general position, all defined by polynomial equations with rational coefficients such that the diffeomorphism sends each $V_i$ to $V'_i$?}

\vspace{.7em}

\subsection*{Further developments and open problems}

Our Main Theorem below, which positively answers above \textsc{Relative $\Q$-algebraicity problem}, is very useful in applications. Let us outline the two main directions for applications.

\begin{comment}
	\begin{mydescription}
		\item[The $\Q$-algebraicity problem in low dimensions] Our Main Theorem below is expected to be deeply used to provide a complete positive answer to the \textsc{$\Q$-algebraicity problem} in low dimensions. The 1-dimensional case is just a consequence of \cite[Theorem\,5.1\,\&\,Corollary\,5.2]{GSa}, in dimension $2$ and $3$ the problem is much more challenging. We refer to \cite{BD81,AK92a} for more details on the topological characterization of polyhedra of dimension up to $3$ admitting an algebraic structure and to \cite{CP00} for necessary conditions on local invariants in all dimensions.
		\item[Algebraic homology on manifolds]  $\Q$-algebraic representatives of $\Z/2\Z$-homological cycles on manifolds, see \cite{BT80,AK85,BK89,Kuc11}.
	\end{mydescription}
\end{comment}

\vspace{.5em}

\noindent\textsc{$\Q$-algebraicity problem in low dimensions:} Our Main Theorem below is expected to be deeply used to provide a complete positive answer to the \textsc{$\Q$-algebraicity problem} in low dimensions. The 1-dimensional case is just a consequence of \cite[Theorem\,1.12]{GSa}, in dimension $2$ and $3$ the problem is much more challenging. We refer to \cite{BD81,AK92a} for more details on the topological characterization of polyhedra of dimension up to $3$ admitting an algebraic structure and to \cite{CP00} for necessary conditions on local invariants in all dimensions.

\vspace{.5em}

\noindent\textsc{$\Q$-algebraic homology of manifolds:} It was shown by Thom \cite{Tho54} that not every $\Z/2\Z$-homology class of compact smooth manifolds can be represented by a closed smooth submanifold, but it can be represented by the fundamental class of a compact smooth manifold by a smooth mapping. As a consequence, Kucharz \cite{Kuc05} showed that any $\Z/2\Z$-homology class of a compact Nash manifold can be represented by a semialgebraic arc-symmetric subset. On the other hand, there are many homology classes that can not be represented by algebraic subsets (even singular), see \cite{BD84,AK85,BK89,Ben1}. It is interesting then to see whether those classes that can be represented by real algebraic subsets can be realized by $\Q$-algebraic subsets.

\subsection*{Structure of the paper \& Main Theorem} Let us outline the structure of the paper:

Section \ref{sec:1} is devoted to review the fundamental results of $\Q$-algebraic geometry developed in \cite{FG,GSa}. We recall the notions of $\Q$-algebraic set, the decomposition in $\Q$-irreducible components and the notion of $\R|\Q$-regularity of a point in a $\Q$-algebraic set $V\subset\R^n$ (see Definition \ref{def:R|Q_p}). This notion turned out to be crucial in separating the irreducible components of $\Q$-algebraic sets (see Proposition \ref{cor:Q_setminus}) and to develop $\Q$-algebraic approximation techniques in \cite[\S 3]{GSa}. In the last part of this first section we propose fundamental examples of $\Q$-nonsingular $\Q$-algebraic sets (see Definition \ref{def:R|Q}) that became crucial in Subsection \ref{subsec:3.1}.

Section \ref{sec:2} is devoted to the study of $\Z/2\Z$-homology groups of real Grassmannians. It is well known that each real Grassmannian $\G_{m,n}$ of affine $m$-planes in $\R^{m+n}$ can be embedded as an algebraic subset of $\R^{(m+n)^2}$, see \cite[Theorem\,3.4.4]{BCR98}. Let us identify $\G_{m,n}$ with the above algebraic subset of $\R^{(m+n)^2}$. It is also well known that incidence conditions with respect to a complete flag induce a finite cellular decomposition of $\G_{m,n}$ such that the Euclidean closure of each cell is an algebraic subset of $\G_{m,n}$, see \cite{MS74}. The closures of these cells are called Schubert varieties associated to the complete flag and they generate the $\Z/2\Z$-homology groups of each $\G_{m,n}$. The main result of this section is that, if we choose the complete flag $0\subset\R\subset\R^2\subset\dots\subset\R^{m+n}$ with respect to the standard coordinates of $\R^{m+n}$, each Schubert variety $\sigma_\lambda$ defined by incidence conditions, prescribed by a partition $\lambda$, with respect to the above complete flag admits a $\Q$-desingularization (see Definition \ref{def:Q_des}). Latter result ensures that each $\Z/2\Z$-homology class of each real Grassmannian $\G_{m,n}$ has a $\Q$-nonsingular $\Q$-algebraic representative. This is a key property for the construction of $\Q$-algebraic relative bordisms in Section \ref{sec:3}. %We point out that the desingularization procedure we apply in this section is inspired by Zelevinski's paper \cite{Zel83} on small resolution of singularities of Schubert varieties. 
We point out that very general algorithms in resolution of singularities in characteristic $0$ provide a resolution which is functorial and invariant under field extension, see as references \cite[Section\,5.7]{Wlo05}, \cite[Section\,3.34.2]{Kol07} and \cite[Theorem\,1.1]{BM08}. However, mentioned algorithms have a very high complexity, thus it is a very hard task in general to concretely apply them obtaining an explicit resolution of a given algebraic variety, see \cite{FKP}. By contrast, the application of the very explicit Bott-Samelson desingularization procedure for real embedded Schubert varieties, see \cite{Dem74,Zel83} and the books \cite{FP98,Man01,Bri05}, allows us to conbinatorically control the equations defining the center of each blow-up of the resolution. We emphasize that, to the best of the author's knowledge, the $\R|\Q$-regularity of real embedded Bott-Samelson varieties is not a direct consequence of results already known in the literature, although Demazure \cite{Dem74} develops the desingularization procedure in the very general setting of flag varieties over a field $k$. Indeed, for our purposes, the choice of an appropriate embedding in some real affine space of each Grassmannian manifold $\G_{m,n}$ and of the resulting desingularization of its embedded Schubert subvarieties plays a crucial role for the notion of $\R|\Q$-regularity, see Example \ref{ex:R|Q-reg} for a clear explanation. We refer the interested reader to Remark \ref{rem:Q-des} for more details on how to deduce the existence of a $\Q$-desingularization of $\Q$-algebraic sets from general results in resolution of singularities and why Bott-Samelson procedure is preferable in the case of Schubert varieties.

Section \ref{sec:3} is divided in two different subsections. The first one is devoted to adapt `over $\Q$' the topological construction of what we call ``relative bordisms" introduced by Akbulut and King in \cite{AK81b}. We stress that the topological ideas come from Akbulut and King paper but to ensure the equations to have rational coefficients, both globally and locally, our choice of an appropriate embedding is crucial. The second subsection is a review on $\Q$-stable $\Q$-algebraic sets introduced by Ghiloni and the author in \cite{GSa}.

Section \ref{sec:4} contains the main results of the paper, whose proofs require all previous constructions. Consider $\R^{k}$ equipped with the usual Euclidean topology, for every $k\in\N$. Let $V$ be a nonsingular real algebraic subset of $\R^{n}$ of dimension $d$ equipped with the relative topology of $\R^{n}$. Let $\mscr{C}^\infty(V,\R^k)$ be the set of all $\mscr{C}^\infty$ maps from $V$ to $\R^k$, and let $\Nn(V,\R^k)\subset\mscr{C}^\infty(V,\R^k)$ be the set of all Nash maps from $V$ to $\R^k$. Denote by $\mscr{C}^\infty_\w(V,\R^k)$ the set $\mscr{C}^\infty(V,\R^k)$ equipped with the usual weak $\mscr{C}^\infty$ topology, see \cite[\S 2.1]{Hir94}, and $\Nn_\w(V,\R^k)$ the set $\Nn(V,\R^k)$ equipped with the relative topology induced by $\mscr{C}^\infty_\w(V,\R^k)$. Here we summarize both statements of Theorems \ref{thm:Q_tico_tognoli} \& \ref{thm:non-compact} to completely address the above \textsc{Relative $\Q$-algebraicity problem}.

\vspace{0.7em}

\begin{mthm}
Let $V$ be a nonsingular algebraic subset of $\R^{n}$ of dimension $d$ and let $\{V_i\}_{i=1}^\ell$ be a finite family of nonsingular algebraic subsets of $V$ of codimension $\{c_i\}_{i=1}^\ell$ in general position. Set $m:=\max\{n,2d+1\}$, if $V$ is compact, or $m:=n+2d+3$, if $V$ is non-compact.

Then, for every neighborhood $\mathcal{U}$ of the inclusion map $\iota:V\hookrightarrow\R^{m}$ in $\Nn_{\w}(V,\R^{m})$ and for every neighborhood $\mathcal{U}_i$ of the inclusion map $\iota|_{V_i}:V_i\hookrightarrow\R^m$ in $\Nn_{\w}(V_i,\R^m)$, for every $i\in\{1,\dots,\ell\}$, there exist a $\Q$-nonsingular $\Q$-algebraic set $V'\subset\R^m$, a family $\{V'_i\}_{i=1}^\ell$ of $\Q$-nonsingular $\Q$-algebraic subsets of $V'$ in general position and a Nash diffeomorphism $h:V\to V'$, which simultaneously takes each $V_i$ to $V'_i$, such that, if $\jmath:V'\hookrightarrow\R^m$ denotes the inclusion map, then $\jmath\circ h\in\mathcal{U}$ and $(\jmath\circ h)|_{V_i}\in\mathcal{U}_i$, for every $i\in\{1,\dots,\ell\}$. Moreover, $h$ extends to a semialgebraic homeomorphism from $\R^m$ to $\R^m$.
\end{mthm}

\vspace{2em}

%%%%%%%%%%

\section{Review on (real) $\Q$-algebraic geometry}%{Real $\Q$-algebraic sets}% defined over $\Q$}
\label{sec:1}

\vspace{0.5em}

\subsection{$\Q$-Algebraic sets, $\R|\Q$-regularity and $\Q$-regular functions}

In this subsection we briefly recall some fundamental notions of real and complex $\Q$-algebraic geometry introduced in \cite{FG,GSa}.

\emph{Let $L|K$ be a extension of fields.} Fix $n\in\N\setminus\{0\}$. Consider $K[x]:=K[x_1,\dots,x_{n}]\subset L[x_1,\dots,x_{n}]=:L[x]$ and $K^{n}\subset L^{n}$. Given $F\subset L[x]$ and $S\subset L^{n}$ define
\begin{align*}
\mathcal{Z}_L(F)&:=\{x\in L^{n}:f(x)=0,\,\forall f\in F\},\\
\I_K(S)&:=\{f\in K[x]:f(x)=0,\,\forall x\in S\}.
\end{align*}

Clearly $\I_K(S)$ is an ideal of $K[x]$. If $F=\{f_1,\dots,f_s\}\subset L[x]$, for some $s\in\N$, then we set $\mathcal{Z}_L(f_1,\dots,f_s):=\mathcal{Z}_\R(F)$.

Let us generalize the notions of (real) algebraic and $\Q$-algebraic sets.

\vspace{0.5em}

\begin{definition}
Let $V\subset L^{n}$. We say that $V$ is a \emph{$K$-algebraic subset of $L^{n}$}, or $V\subset L^{n}$ is a \emph{$K$-algebraic set}, if there exists $F\subset K[x]$ such that $V=\mathcal{Z}_L(F)$.
\end{definition}

\vspace{0.5em}

Denote as the \emph{$K$-Zariski topology of $L^{n}$} the unique topology of $L^{n}$ having $K$-algebraic sets as closed sets. Such a topology is noetherian, since it is coarser than the usual Zariski topology of $L^{n}$. As usual, Noetherianity implies that every $K$-algebraic subset of $L^{n}$ is the common solution set of a finite number of polynomials with coefficients in $K$.

\emph{Let $V\subset L^{n}$ be a $K$-algebraic set.} We say that $V\subset L^{n}$ is \emph{$K$-irreducible} if it is irreducible with respect to the $K$-Zariski topology. By classical arguments, we have that $V\subset L^{n}$ is $K$-irreducible if and only if $\I_K(V)$ is a prime ideal of $K[x]$.  Noetherianity also implies that every $K$-algebraic set can be uniquely decomposed as a finite union of $K$-irreducible algebraic subsets of $V$. We call those $K$-irreducible algebraic subsets the \emph{$K$-irreducible components} of $V\subset L^{n}$. Observe that $L$-irreducible components of $V\subset L^{n}$ coincides with the usual irreducible components of $V$, viewed as an algebraic subset of $L^{n}$. As the usual ($L$-)Zariski topology is finer than the $K$-Zariski topology, if the $K$-algebraic set $V\subset L^{n}$ is irreducible, it is also $K$-irreducible. If both $L$ and $K$ are algebraically closed or real closed, the converse implication is also true. Otherwise, it may happen that $V\subset L^{n}$ is $K$-irreducible but not irreducible. An example of this behavior can be found choosing $L|K=\R|\Q$ and $V:=\{-\sqrt{2},\sqrt{2}\}\subset\R$.

The \emph{$K$-dimension $\dim_K(V)$} of $V$ is defined as the Krull dimension of the ring $K[x]/\I_K(V)$. As above, $\dim_L(V)$ coincides with the usual dimension $\dim(V)$ of $V$, viewed as an algebraic subset of $L^{n}$. A $K$-version of a classical result concerning irreducibility and dimension holds:  if $V\subset L^{n}$ is $K$-irreducible and $W\subset L^{n}$ is a $K$-algebraic set such that $W\subsetneq V$, then $\dim_K(W)<\dim_K(V)$. Moreover, if $L$ is algebraically closed or real closed, then $\dim_K(V)=\dim(V)$. For further information on these topics we refer to \cite[\S 2]{FG}.

Moreover, if $L$ is algebraically closed, an application of Hilbert's Nullstellensatz and Galois Theory in \cite[Theorem\,2.3.4(vi)]{FG} guarantees that $\I_L(V)=\I_K(V)L[x]$, thus the ideal $\I_K(V)$ of $K[x]$ gives complete information about $V\subset L^{n}$. Latter equality holds as well when both $L$ and $K$ are real closed fields by the Tarski-Seidenberg principle, see \cite[Corollary\,2.2.17(ii)]{FG}. However, this is false for a general field extension $L|K$. As an example, if $L|K=\R|\Q$ and $V:=\{x-\sqrt[3]{2}=0\}=\{x^3-2=0\}\subset\R$, then

\begin{equation}\label{eq:non_Q-det}
\I_\R(V)=(x-\sqrt[3]{2})\R[x]\supsetneqq(x^3-2)\R[x]=\I_\Q(V)\R[x].
\end{equation}

We say that an algebraic set $S\subset L^{n}$ is \emph{defined over $K$} if $\I_L(S)=\I_K(S)L[x]$. Clearly, if $S\subset L^{n}$ is defined over $K$, then $S\subset L^n$ is also $K$-algebraic. As said before, if $L$ is algebraically closed or both $L$ and $K$ are real closed, then the concepts of $K$-algebraic subset of $L^{n}$ and algebraic subset of $L^{n}$ defined over $K$ do coincide. As explained in \cite[Appendix\,C]{GSa}, in the real algebraic setting the situation is completely different and various notions of nonsingularity over $\Q$ arise. In what follows we recall the definition of $\R|\Q$-regular points of a $\Q$-algebraic set introduced and studied in \cite{FG} and whose characterizations are further investigated in \cite{GSa}. 

Pick a point $a=(a_1,\ldots,a_{n})\in\R^n$. We denote by $\mathfrak{n}_a$ the maximal ideal of polynomials of $\R[x]$ vanishing at $a$, that is
\[
\mathfrak{n}_a:=(x_1-a_1,\ldots,x_{n}-a_{n})\R[x]
\]
and by $\I_\Q(V)$ the vanishing ideal of $V$ in $\Q[x]$, as above.
The following notion of $\R|\Q$-regular point was introduced in \cite[Definition\,5.1.1]{FG}.

\vspace{0.5em}

\begin{definition}[$\R|\Q$-regular points]\label{def:R|Q_p}
Let $V\subset\R^{n}$ be a $\Q$-algebraic set and let $a\in V$. We define the \emph{$\R|\Q$-local ring $\mathcal{R}^{\R|\Q}_{V,a}$ of $V$ at $a$} as
\[
\mathcal{R}^{\R|\Q}_{V,a}:=\R[x]_{\mathfrak{n}_a}/(\mathcal{I}_{\Q}(V)\R[x]_{\mathfrak{n}_a}).
\]
We say that $a$ is a \emph{$\R|\Q$-regular point of $V$} if $\mathcal{R}^{\R|\Q}_{V,a}$ is a regular local ring of dimension $\dim(V)$. We denote by $\Reg^{\R|\Q}(V)$ the set of all $\R|\Q$-regular points of $V$.
\end{definition}

\vspace{0.5em}

In \cite[Theorem\,5.4.1]{FG}, Fernando and Ghiloni show that the set $\Reg^{\R|\Q}(V)$ is a non-empty Zariski open subset of $\Reg(V)$. Though, it may happen that $\Reg^{\R|\Q}(V)\subsetneq\Reg(V)$, see the example below.

\begin{example}\label{ex:R|Q-reg}
Consider the $\Q$-algebraic line $V:=\{x_1-\sqrt[3]{2}x_2=0\}=\{x_1^3-2x_2^3=0\}\subset\R^2$. Observe that $V$ is nonsingular (as an algebraic set), but $\overline{O}:=(0,0)$ is not $\R|\Q$-regular. Indeed, $\Ii_{\Q}(V)=(x_1^3-2x_2^3)$ and $\mathcal{R}^{\R|\Q}_{V,\overline{O}}:=\R[x_1,x_2]_{(x_1,x_2)}/(x_1^3-2x_2^3)\R[x_1,x_2]_{(x_1,x_2)}$, hence $\mathcal{R}^{\R|\Q}_{V,\overline{O}}$ is not an integral domain since $x_1^3-2x_2^3=(x_1-\sqrt[3]{2}x_2)(x_1^2+\sqrt[3]{2}x_1x_2+\sqrt[3]{4}x_2^2)$. In particular, this implies that $\mathcal{R}^{\R|\Q}_{V,\overline{O}}$ is not a regular local ring either.
\end{example}

This leads to the following definition originally introduced in \cite[Definition\,1.9]{GSa}. 

\vspace{0.5em}

\begin{definition}\label{def:R|Q}
Let $V\subset\R^{n}$ be a $\Q$-algebraic set. We say that $V$ is \emph{$\Q$-determined} if $\Reg^{\R|\Q}(V)=\Reg(V)$. If in addition $V$ is nonsingular, in other words, $V=\Reg(V)=\Reg^{\R|\Q}(V)$, we say that $V$ is \emph{$\Q$-nonsingular}.
\end{definition}

\vspace{0.5em}

\begin{notation}
In what follows, if $V\subset\R^{n}$ is a $\Q$-algebraic set and $a\in V$, we set $\reg^*_{V,a}:=\reg^{\R|\Q}_{V,a}$ and $\Reg^*(V):=\Reg^{\R|\Q}(V)$ for short.
\end{notation}

\vspace{0.5em}

In \cite[Theorem\,5.1.9]{FG} the authors characterized the notion of $\R|\Q$-regularity via a $\R|\Q$-Jacobian criterion. In particular, if $V\subset\R^{n}$ of dimension $d$ is a nonsingular $\Q$-algebraic set, then $V$ is $\Q$-nonsingular if and only if for every $a\in V$ there are $q_1,\dots,q_{n-d}\in\mathcal{I}_\Q(V)$ such that $\nabla q_1(a),\dots,\nabla q_{n-d}(a)$ are linearly independent and there exists a Zariski open subset $U$ of $\R^{n}$ such that $V\cap U=\mathcal{Z}(q_1,\dots,q_{n-d})\cap U$. A systematic treatment on equivalent definitions of $\Q$-nonsingular $\Q$-algebraic sets involving Galois Theory, the complex Zariski closure of a (real) $\Q$-algebraic set and the vanishing ideal in the ring of global rational functions can be found in \cite[\S 1.6]{Sav-T}.

%In \cite[Lemma\,2.10]{GSa}, Ghiloni and the author characterized the notion of $\R|\Q$-nonsingularity via a $\R|\Q$-jacobian criterion. In particular, if $V\subset\R^{n}$ of dimension $d$ is a nonsingular $\Q$-algebraic set, then $V$ is $\Q$-nonsingular if and only if for every $a\in V$ there are $q_1,\dots,q_{n-d}\in\mathcal{I}_\Q(V)$ such that $\nabla q_1(a),\dots,\nabla q_{n-d}(a)$ are linearly independent and there exists a Zariski open subset $U$ of $\R^{n}$ such that $V\cap U=\mathcal{Z}(q_1,\dots,q_{n-d})\cap U$.

Here we recall a crucial consequence of mentioned $\R|\Q$-Jacobian criterion \cite[Theorem\,5.1.9]{FG} and deeply applied in that paper. Its importance will be clear in the proof of our Main Theorem, more precisely, in the proof of Lemma \ref{lem:Q-Tognoli} below.

\vspace{0.5em}

\begin{proposition}[{\cite[Proposition\,2.14]{GSa}}]\label{cor:Q_setminus}
Let $V\subset\R^{n}$ and $Z\subset\R^{n}$ be two $\Q$-nonsingular $\Q$-algebraic sets of the same dimension $d$ such that $Z\subsetneq V$. Then $V\setminus Z\subset\R^{n}$ is a $\Q$-nonsingular $\Q$-algebraic set of dimension $d$ as well.
\end{proposition}

\vspace{0.5em}

%%%%%%%%%%%%%%%%%%%%%%%

Let us recall the notion of $\Q$-regular map introduced in \cite{GSa}. Let $S\subset\R^{n}$ be a set and let $f:S\to\R$ be a function. We say that $f$ is \emph{$\Q$-regular} if there are $p,q\in\Q[x]=\Q[x_1,\ldots,x_{n}]$ such that $\mathcal{Z}_\R(q)\cap S=\varnothing$ and $f(x)=\frac{p(x)}{q(x)}$ for every $x\in S$. We denote by $\reg^\Q(S)$ the set of $\Q$-regular functions on $S$. Observe that usual pointwise addition and multiplication induce a natural ring structure on $\reg^\Q(S)$. Let $T\subset\R^h$ be a set ad let $g:S\to T$ be a map. We say that $g$ is \emph{$\Q$-regular} if there exist $g_1,\ldots,g_h\in\reg^\Q(S)$ such that $g(x)=(g_1(x),\ldots,g_h(x))$ for all $x\in S$. We denote by $\reg^\Q(S,T)$ the set of $\Q$-regular maps from $S$ to $T$. We say that the map $g:S\to T$ is a \emph{$\Q$-biregular isomorphism} if $g$ is bijective and both $g$ and $g^{-1}$ are $\Q$-regular. If there exists such a $\Q$-biregular isomorphism $g:S\to T$, we say that $S$ is \emph{$\Q$-biregularly isomorphic} to $T$. Observe that, as usual in real algebraic geometry, previous global definitions of $\Q$-regular functions and maps do coincide with the local ones, that is: for each $a\in S$, there exist $p_a,q_a\in\Q[x]$ such that $q_a(a)\neq0$ and $f(x)=\frac{p_a(x)}{q_a(x)}$ for all $x\in S\setminus \mathcal{Z}_\R(q_a)$ or there exist $p_{a,1}\dots,p_{a,h},q_{a}\in\Q[x]$ such that $q_{a}(a)\neq0$ and $g(x)=(\frac{p_{a,1}(x)}{q_{a}(x)},\dots,\frac{p_{a,h}(x)}{q_{a}(x)})$ for all $x\in S\setminus \mathcal{Z}_\R(q_a)$.

Ghiloni and the author proved basic properties of $\Q$-determined $\Q$-algebraic sets and $\Q$-regular maps, for more details we refer to \cite[Lemma\,2.18\,\&\,Proposition\,2.19]{GSa}. Those properties will be strongly applied in this paper.

Let us recall the definitions of overt polynomial and projectively $\Q$-closed $\Q$-algebraic set, introduced in \cite[p.\ 427]{AK81a} and specified over $\Q$ in \cite{GSa}. Let $p\in\R[x]$ be a polynomial. Write $p$ as follows: $p=\sum_{i=0}^dp_i$, where $d$ is the degree of $p$ and each polynomial $p_i$ is homogeneous of degree $i$ (as a convention we assume the degree of the zero polynomial to be $0$). We say that the polynomial $p\in\R[x]$ is \emph{overt} if $\mathcal{Z}_\R(p_d)$ is either empty or $\{0\}$. Observe that a non-constant overt polynomial function $p:\R^{n}\to\R$ is proper. A $\Q$-algebraic set $V\subset\R^{n}$ is called \emph{projectively $\Q$-closed} if there exists an overt polynomial $p\in\Q[x]\subset\R[x]$ such that $V=\mathcal{Z}_\R(p)$. This notion coincides to require that $\theta(V)$ is Zariski closed in $\Proy^{n}(\R)$ and it is described by polynomial equations with rational coefficients, where $\theta:\R^{n}\to\Proy^{n}(\R)$ denotes the affine chart $\theta(x_1,\ldots,x_{n}):=[1,x_1,\ldots,x_{n}]$. As a consequence, if $V$ is projectively closed, then it is compact in $\R^{n}$. We refer to \cite[Lemma\,2.24]{GSa} for fundamental properties of projectively $\Q$-closed $\Q$-algebraic subsets of $\R^{n}$.

\subsection{Fundamental examples}

Let $\G_{m,n}$ denote the Grassmannian manifold of $m$ dimensional subspaces of $\R^{m+n}$. Identify $\R^{(m+n)^2}$ with the set of $(m+n)\times (m+n)$ real matrices. It is well known, see \cite[\S 3]{BCR98}, that every Grassmannian $\G_{m,n}$ is biregular isomorphic to the following algebraic subset of $\R^{(m+n)^2}$:
\begin{equation}\label{eq:GG}
\G_{m,n}=\big\{X\in\R^{(m+n)^2}:X^T=X,X^2=X,\mr{tr}(X)=m\big\}.
\end{equation}
The biregular map assigns to each point $p$ of the Grassmannian, corresponding to an $m$ subspace $V_p$ of $\R^{m+n}$, the matrix $X_p\in\R^{(m+n)^2}$ of the orthogonal projection of $\R^{m+n}$ onto $V_p$ with respect to the canonical basis of $\R^{m+n}$.

Let $\E_{m,n}$ denote the (total space of the) universal vector bundle over $\G_{m,n}$ as the following algebraic subset of $\R^{(m+n)^2+m+n}$: 
\[
\E_{m,n}:=\{(X,y)\in\G_{m,n}\times\R^{m+n}\,:\,Xy=y\}.
\]
It is well-known that $\E_{m,n}$ is a connected $\mscr{C}^{\infty}$ submanifold of $\R^{(m+n)^2+m+n}$ of dimension $m(n+1)$.

In \cite[Lemma\,2.31]{GSa}, Ghiloni and the author proved that both $\G_{m,n}\subset\R^{(m+n)^2}$ and $\E_{m,n}\subset\R^{(m+n)^2+m+n}$ are projectively $\Q$-closed $\Q$-nonsingular $\Q$-algebraic sets. These algebraic sets are fundamental examples to represent cobordism classes of compact smooth manifolds and to develop algebraic approximation techniques over $\Q$ as in \cite[\S 3]{GSa}. However, for the purposes of this paper we need further examples to develop relative techniques `over $\Q$'  in the spirit of \cite{AK81b}.

Let $\E_{m,n}$ denote the (total space of the) universal sphere bundle over $\G_{m,n}$ as the following algebraic subset of $\R^{(m+n)^2+m+n+1}$:
\[
\E^*_{m,n}=\{(X,y,t)\in\G_{m,n}\times\R^{m+n}\times\R\,|\,Xy=y,\,|y|_n^2+t^2=t \}.
\]
It is well-known that $\E^*_{m,n}$ is a connected $\mscr{C}^{\infty}$ submanifold of $\R^{(m+n)^2}\times\R^{m+n}\times\R$ of dimension $m(n+1)$.

\vspace{0.5em}

\begin{lemma}\label{lem:Q_sphere_bundle}
Each universal sphere bundle $\E^*_{m,n}\subset\R^{(m+n)^2+m+n+1}$ is a projectively $\Q$-closed $\Q$-nonsingular $\Q$-algebraic set.
\end{lemma}	

\begin{proof}
Let $\phi:\R^{(m+n)^2}\times\R^{m+n}\times\R\rightarrow\R^{(m+n)^2}\times\R^{(m+n)^2}\times\R^{m+n}\times\R$ be the polynomial map defined by
\[
\phi(X,y,t):=(X^T-X,X^2-X,Xy-y,|y|^2_{m+n}+t^2-t).
\]
We prove that the polynomial $\mr{tr}(X)-m$ and the polynomial components of $\phi$ do suffice to describe nonsingular points of $\E^*_{m,n}\subset\R^{(m+n)^2}\times\R^{m+n}\times\R$ via the $\R|\Q$-Jacobian criterion \cite[Theorem\,5.1.9]{FG}. As in the proof of \cite[Lemma\,2.31]{GSa}, it suffices to show that $\mr{rnk}\,J_\phi(A,b,c)\geq (m+n)^2+m+n+1-m(n+1)-1=(m+n)^2-mn+n$ for all $(A,b,c)\in \E^*_{m,n}$.

First, we prove that $\mr{rnk}\,J_\phi(D_m,v,c)\geq (m+n)^2-mn+n$ if $D$ is the diagonal matrix in $\R^{(m+n)^2}$ having $1$ in the first $m$ diagonal positions and $0$ otherwise, $v=(v_1,\ldots,v_{m+n})^T$ is a vector of $\R^{m+n}$ and $c\in\R$ such that $(D_m,v,c)\in \E^*_{m,n}$. For each $\ell\in\{1,\dots,m+n+1\}$, define the polynomial functions $h_\ell:\R^{(m+n)^2}\times\R^{(m+n)}\times\R\to\R$ by
\begin{align*}
h_\ell(X,y,t)&:=\Big(\sum_{j=1}^{m+n} x_{\ell j}y_j\Big)-y_\ell\quad\text{if $\ell\neq m+n+1$},\\
h_{m+n+1}(X,y,t)&:=|y|^2_{m+n}+t^2-t,
\end{align*}
for all $X=(x_{ij})_{i,j}\in\R^{(m+n)^2}$, $y=(y_1,\dots,y_{m+n})\in \R^{m+n}$ and $t\in\R$. Thus, with the same notation of the proof of \cite[Lemma\,2.31]{GSa}, it follows that $$\phi(X,y,t)=((f_{ij}(X))_{i,j},(g_{ij}(X))_{i,j},(h_\ell(X,y,t))_\ell).$$ Thanks to the proof of mentioned \cite[Lemma\,2.31]{GSa}, we already know that the rank of the Jacobian matrix at $(D_m,v,c)$ of the map $$(X,y,t)\mapsto((f_{ij}(X))_{i,j},(g_{ij}(X))_{i,j},(h_\ell(X,y,t))_\ell)$$ is $\geq (m+n)^2-mn+n$. Thus, we only have to look at the components $(h_\ell(X,y))_\ell$ in order to prove that $h_{m+n+1}$ always produces an additional independent gradient with respect to the gradients of $(h_\ell(X,y))_{\ell\neq m+n+1}$. By a direct computation we see that
\begin{align*}
\nabla h_\ell(D_m,v,c)&=\Big(\sum_{j=1}^{m+n} v_j E_{\ell j},-e_\ell,0\Big)\quad\text{if $\ell\in\{m+1,\dots,m+n\}$},\\
\nabla h_{m+n+1}(D_m,v,c)&=\big(0,2v,2c-1\big),
\end{align*}
where $E_{\ell j}$ is the matrix in $\R^{(m+n)^2}$ whose $(\ell,j)$-coefficient is equal to $1$ and $0$ otherwise, and $\{e_1,\ldots,e_{m+n}\}$ is the canonical vector basis of $\R^{m+n}$. Observe that $\nabla h_{m+n+1}(D_m,v,c)$ is linearly independent with respect to $(\nabla h_\ell(X,y))_{\ell\neq m+n+1}$ when $c\neq 1/2$, otherwise, if $c=1/2$, then $$\nabla h_{m+n+1}(D_m,v,c)=(0,2v,0),$$ so it is contained in the $m$-plane satisfying $D_m v=v$, hence it is linearly independent with respect to $(\nabla h_\ell(X,y))_{\ell\neq m+n+1}$ as well. Consequently, we obtain that $\mr{rnk}\,J_\phi(D_m,v,c)\geq (m+n)^2-mn+n$ for every $v\in\R^n$ and $c\in\R$ such that $(D_m,v,c)\in\E^*_{m,n}$.

Let us complete the proof. Let $(A,b,c)\in\E^*_{m,n}$, let $G\in O(m+n)$ be such that $D_m=G^TAG$ and let $v:=G^Tb$. By the choice of $G$ we see that $|v|^2_{m+n}=|G^T v|^2_{m+n}=|b|^2_{m+n}$, hence $c$ satisfies $|v|^2_{m+n}+c^2-c=0$ as well. Note that $D_m v=G^TAGG^Tb=G^TAb=G^Tb=v$, i.e., $(D_m,v,c)\in\E^*_{m,n}$. Define the linear automorphisms $\psi:\R^{(m+n)^2}\to\R^{(m+n)^2}$ and $\tau:\R^{m+n}\rightarrow\R^{m+n}$ by $\psi(X):=G^TXG$ and $\tau(y)=G^T y$. Since $(\psi\times\tau\times\id_{\R})(A,b,c)=(D_m,v,c)$ and $(\psi\times\psi\times\tau\times\id_{\R})\circ\phi=\phi\circ(\psi\times\tau\times\id_{\R})$, we have that $J_{\psi\times\psi\times\tau\times\id_{\R}}(\phi(A,b,c))J_\phi(A,b)=J_\phi(D_m,v,c)J_{\psi\times\tau\times\id_{\R}}(A,b,c)$. Bearing in mind that both matrices $J_{\psi\times\psi\times\tau\times\id_{\R}}(\phi(A,b,c))$ and $J_{\psi\times\tau\times\id_{\R}}(A,b,c)$ are invertible, it follows that $\mr{rnk}\,J_\phi(A,b,c)=\mr{rnk}\,J_\phi(D_m,v,c)\geq (m+n)^2-mn+n+1$, as desired. 

Let $\sph^{m+n}\subset\R^{m+n+1}$ be the standard unit sphere. Since $\E^*_{m,n}\subset\R^{(m+n)^2}\times\R^{m+n}\times\R$ is the zero set of $|\phi(X,y,t)|_{(m+n)^2}^2+(\mr{tr}(X)-m)^2\in\Q[(x_{ij})_{i,j},(y_k)_k,t]$ and $\E^*_{m,n}\subset\G_{m,n}\times \sph^{m+n}$, which is a projectively $\Q$-closed $\Q$-algebraic subset of $\R^{(m+n)^2}\times\R^{m+n}\times\R$ by \cite[Lemmas\,2.31\,\&\,2.24(iv)]{GSa}, we have that $\E^*_{m,n}$ is projectively $\Q$-closed in $\R^{k}\times\R^{m+n}\times\R$ as well by \cite[Lemma\,2.24(ii)]{GSa}.
\end{proof}

Let $W\subset\R^k$ be a nonsingular algebraic set of dimension $d$. Let $\G:=\prod_{i=1}^{\ell}\G_{m_i,n_i}$, let $\E^*:=\prod_{i=1}^{\ell}\E_{m_i,n_i}^*$ and le $\mu:W\rightarrow \G$ be a regular map. Let $\pi_i:=\G\to\G_{m_i,n_i}$ be the projection onto the $i$-th factor and let $\mu_i:W\to\G_{m_i,n_i}$ be defined as $\mu_i:=\pi_i\circ \mu$ for every $i\in\{1,\dots,\ell\}$. We define the pull-back sphere bundle $\mu^*(\E^*)$ over $W$ of $\E^*$ via $\mu$ as the following algebraic subset of $\R^k\times\prod_{i=1}^{\ell}(\R^{m_i+n_i}\times\R)$:
\begin{align*}
\mu^*(\E^*):=\{(x,y^1,t_1,&\dots,y^{\ell},t_{\ell})\in W\times\prod_{i=1}^{\ell}(\R^{m_i+n_i}\times\R)\,|\,\\&\mu_i(x)y^i=y^i,\,
|y^i|_{m+n}^2+t_i^2=t_i \text{ for }i=1,\dots,\ell\}.
\end{align*}
It is well-known that $\mu^*(\E^*)$ is a compact $\mscr{C}^{\infty}$ submanifold of $\R^{k}\times(\R^{m+n}\times\R)^\ell$ of dimension $d+\sum_{i=1}^{\ell} m_i$.

\vspace{0.5em}

\begin{lemma}\label{lem:Q_sphere_pullback}
Let $W\subset\R^k$ be a projectively $\Q$-closed $\Q$-nonsingular $\Q$-algebraic set of dimension $d$. Let $\mu:W\rightarrow \G$ be a $\Q$-regular map. Then $\mu^*(\E^*)\subset \R^k\times\prod_{i=1}^{\ell}(\R^{m_i+n_i}\times\R)$ is a projectively $\Q$-closed $\Q$-nonsingular $\Q$-algebraic set.
\end{lemma}

\begin{proof}
For simplifying the notation we only prove the case $\ell=1$, in the general case the proof works in the same way. Let $\ell=1$, $\G=\G_{m,n}$ and $\E=\E^*_{m,n}$. %By \cite[Lemma\,lem:Q-regular]{GSa}, 
There are $s\in\N^*$ and $p_1,\dots,p_s\in\Q[x]=\Q[x_1,\dots,x_k]$ such that $\Ii_{\R^k}(W)=(p_1,\dots,p_s)$. Let $\phi:\R^{k}\times\R^{m+n}\times\R\rightarrow\R^s\times\R^{m+n}\times\R$ be the regular map defined by
\begin{align*}
\phi(x,y,t):=(p_1(x),\dots,p_s(x),&\mu(x)y-y,|y|^2_{m+n}+t^2-t),
\end{align*}
where $\mu(x)\in\G\subset\R^{(m+n)^2}$ is in matrix form. We prove that the polynomial components of $\phi$ do suffice to describe nonsingular points of $\mu^*(\E^*)\subset\R^{k}\times\R^{m+n}\times\R$ via the $\R|\Q$-Jacobian criterion \cite[Theorem\,5.1.9]{FG}. As before, it suffices to show that $\mr{rnk}\,J_\phi(a,b,c)\geq k-d+n+1$ for all $(a,b,c)\in \mu^*(\E^*)$.

As in the proof of Lemma \ref{lem:Q_sphere_bundle}, for every $r\in\{1,\dots,m+n+1\}$, define the polynomial functions $h_r:\R^{(m+n)^2}\times\R^{m+n}\times\R\to\R$ by
\begin{align*}
h_r(X,y,t)&:=\Big(\sum_{j=1}^{m+n} x_{rj}y_j\Big)-y_s\quad\text{if $r\neq m+n+1$}\\
h_{m+n+1}(X,y,t)&:=|y|^2_{m+n}+t^2-t,
\end{align*}
for all $X=(x_{ij})_{ij} \in\R^{(m+n)^2}$, $y=(y_1,\dots,y_{m+n})\in \R^{m+n}$ and $t\in\R$. Thus, with the same notation of the proof of Lemma \ref{lem:Q_sphere_bundle}, it follows that
\[
\phi(x,y,t)=(p_1(x),\dots,p_s(x),h_1(\mu(x),y,t),\dots,h_{m+n+1}(\mu(x),y,t)). 
\]
Let $\nu:\R^k\to\G$, defined as $\nu(x):=(\nu(x)_{ij})_{ij}$, be any regular function such that $D_m\in\nu(W)$. Define $h'_r:\R^k\times\R^{m+n}\times\R\to \R$ as $h'_r:=h_r\circ(\nu\times\id_{\R^{m+n}}\times\id_{\R})$ for every $r\in\{1,\dots,m+n+1\}$. Thanks to the proof of \cite[Lemma\,2.31]{GSa} and being $W$ nonsingular of dimension $d$, we get that the rank of the Jacobian matrix of the map $$(x,y,t)\mapsto(p_1(x),\dots,p_s(x),h'_1(x,y,t),\dots,h'_{m+n+1}(x,y,t))$$ at every $(a,v,c)\in\nu^*(\E^*)$ such that $\nu(a)=D_m$ is $\geq k-d+n+1$, hence equal to $k-d+n+1$. Indeed, denote by $\nu_r:\R^{m+n}\to\R$ be the regular map defined as $\nu_r(x):=(\nu_{r1}(x),\dots,\nu_{r\,m+n}(x))$, that is, the map associated to the $r$-th row of $\nu$. Then we have:
\begin{align*}
\nabla p_i(a,v,c)&=(\nabla p_i(a),0,0)\quad\text{for every $i=1,\dots,s$};\\
\nabla h'_r(a,v,c)&=\Big(\nabla \nu_r(a)\cdot v^T,-e_r,0\Big)\quad\text{if $r\in\{m+1,\dots,m+n\}$},\\
\nabla h_{m+n+1}(a,v,c)&=\big(0,2v,2c-1\big),
\end{align*}
where $x=(x_1,\dots,x_k)$ and $\{e_1,\ldots,e_{m+n}\}$ denotes the canonical vector basis of $\R^{m+n}$.

Let us complete the proof. Let $(a,b,c)\in\mu^*(\E^*)$, let $G\in O(m+n)$ be such that $D_m=G^T\mu(a)G$ and let $v:=G^Tb$. By the choice of $G$ we see that $|v|^2_{m+n}=|G^T v|^2_{m+n}=|b|^2_{m+n}$, hence $c$ satisfies $|v|^2_{m+n}+c^2-c=0$ as well. Note that $D_m v=G^TAGG^Tb=G^TAb=G^Tb=v$, i.e., $(D_m,v,c)\in\nu^*(E^*)$ with $\nu:W\to\G$ defined as $\nu(a):=G^T\mu(a)G$. Define the regular function $\psi:\R^{k}\times\R^{m+n}\times\R\rightarrow\R^s\times\R^{m+n}\times\R$ by
\[
\psi(x,y,t):=(p_1(x),\dots,p_s(x),\nu(x)y-y,|y|^2_{m+n}+t^2-t),
\]
and the linear automorphism $\tau:\R^{m+n}\rightarrow\R^{m+n}$ by $\tau(y)=G^T y$. Since $(\id_{\R^k}\times\tau\times\id_\R)(a,b,c)=(a,v,c)$ and  $(\id_{\R^s}\times\tau\times\id_{\R})\circ\phi=\phi\circ(\id_{\R^k}\times\tau\times\id_\R)$ we have that
$J_{\id_{\R^s}\times\tau\times\id_{\R}}(\phi(a,b,c))J_\phi(a,b,c)=J_\psi(a,v,c)J_{\id_{\R^k}\times\tau\times\id_{\R}}(a,b,c)$. Since both matrices $J_{\id_{\R^s}\times\tau\times\id_{\R}}(\phi(a,v,c))$ and $J_{\id_{\R^k}\times\tau\times\id_{\R}}(a,b,c)$ are invertible and $\nu(a)=D_m$, it follows that $\mr{rnk}\,J_\phi(a,b,c)=\mr{rnk}\,J_\psi(a,v,c)= k-d+n+1$, as desired. 

Since $\mu^*(\E^*)\subset\R^{k}\times\R^{m+n}\times\R$ is the zero set of $|\phi(x,y,t)|_{s+m+n+1}^2\in\Q[x,y,t]$ and $\mu^*(\E^*)$ is contained in $W\times S^{m+n}$, which is a projectively $\Q$-closed algebraic subset of $\R^{k}\times\R^{m+n}\times\R$ by \cite[Lemmas\,2.31\,\&\,2.24(iv)]{GSa}, we have that $\mu^*(\E^*)$ is projectively $\Q$-closed in $\R^{k}\times\R^{m+n}\times\R$ as well by \cite[Lemma\,2.24(ii)]{GSa}.
\end{proof}

%%%%%%%%%%%%%%%%%%%%%%%%
\vspace{1.5em}

\section{Homology of real embedded Grassmannians}\label{sec:2}

\vspace{0.5em}

\subsection{$\Q$-Desingularization of real embedded Schubert varieties}\label{subsec:2.1}

Let $\G_{m,n}\subset\R^{(m+n)^2}$ be the embedded Grassmannian manifold of $m$-dimensional vector subspaces of $\R^{m+n}$. Let us construct an embedded version of Schubert varieties inducing a cellular decomposition of $\G_{m,n}$. Consider the complete flag of $\R^{m+n}$ consisting of the strictly increasing sequence of each $\R^k$, with $k\leq m+n$, spanned by the first $k$ elements of the canonical basis of $\R^{m+n}$. That is:
\[
0\subset \R\subset\dots\subset\R^k\subset \dots\subset \R^{m+n}.
\]
We will refer to the previous complete flag as the \emph{canonical complete flag of $\R^{m+n}$}. Let us define the \emph{Schubert varieties} of $\G_{m,n}$ with respect to above complete flag by following the convention in \cite[\S 3]{Man01}. Define a \emph{partition} $\lambda=(\lambda_1,\dots,\lambda_m)$ as a decreasing sequence of integers such that $n\geq \lambda_1\geq\dots\geq\lambda_m\geq 0$. Hence, $\lambda$ corresponds uniquely to a Young diagram in a $(m\times n)$-rectangle. Denote by $D_\ell$ the $(m+n)^2$ matrix associated to the orthogonal projection $\R^{m+n}\to\R^\ell$ sending $(x_1,\dots,x_{m+n})\mapsto(x_1,\dots,x_\ell)$ with respect to the canonical basis of $\R^{m+n}$ for every $\ell\in\{1,\dots,m+n\}$. Hence, $D_\ell$ is the diagonal matrix in $\R^{(m+n)^2}$ having $1$ in the first $\ell$ diagonal positions and $0$ otherwise. Define the \emph{Schubert open cell} of $\G_{m,n}$ associated to $\lambda$ with respect to the canonical complete flag as
\[
\Omega_\lambda:=\big\{X\in\G_{m,n}\,|\,\mr{rnk}(XD_\ell)=k\quad\text{if $n+k-\lambda_k\leq \ell\leq n+k-\lambda_{k+1}$}\big\}.
\]
Define the \emph{Schubert variety} of $\G_{m,n}$ associated to the partition $\lambda$ with respect to the canonical complete flag as
\begin{equation}\label{eq:sch}
\sigma_\lambda:=\big\{X\in\G_{m,n}\,|\,\mr{rnk}(XD_{n+k-\lambda_k})\geq k, \text{ for } k=1,\dots,m\big\}.
\end{equation}

The partition $\lambda$ is uniquely determined and uniquely determines a sequence of incidence conditions  with respect to the above canonical complete flag of $\R^{m+n}$. In addition, the matrix $XD_\ell=(x'_{ij})_{i,j}\in\R^{(m+n)^2}$ satisfies the following relations with respect to $X=(x_{ij})_{i,j}\in\R^{(m+n)^2}$:
\[
x'_{ij}=x_{ij}\text{ if $j\leq \ell$ and }x'_{ij}=0\text{ otherwise.}
\]

Here we summarize some general properties of Schubert varieties translated in our embedded construction. For more details we refer to \cite[\S 6]{MS74}\,\&\,\cite[\S 3.2]{Man01}.

\vspace{0.5em}

\begin{lemma}\label{lem:schubert}
Let $\G_{m,n}\subset\R^{(m+n)^2}$ be an embedded Grassmannian manifold and let $\lambda$ be a partition of the  $(m\times n)$-rectangle. Let $\sigma_{\lambda}$ be the Schubert variety in $\G_{m,n}$ defined by the incidence conditions prescribed by $\lambda$ with respect to the canonical complete flag of $\R^{m+n}$. Then:
\begin{enumerate}[label=\emph{(\roman*)}, ref=(\roman*)]
\item\label{lem:schubert_1} $\sigma_{\lambda}$ is an algebraic subset of $\R^{(m+n)^2}$ and $\Omega_\lambda\subset\Reg(\sigma_{\lambda})$.
\item $\Omega_\lambda$ is biregular isomorphic to $\R^{mn-|\lambda|}$, where $|\lambda|:=\sum_{k=1}^{m} \lambda_k$.
\item $\sigma_\lambda$ coincides with the Euclidean closure of $\Omega_\lambda$.
\item $\sigma_\lambda=\bigcup_{\mu\geq\lambda} \Omega_\mu$, where $\mu\geq\lambda$ if and only if $\mu_k\geq \lambda_k$ for every $k\in\{1,\dots,m\}$.
\item $\sigma_\lambda\supset\sigma_\mu$ if and only if $\lambda\leq\mu$, where $\lambda\leq\mu$ means $\lambda_i\leq \mu_i$ for every $i\in\{1,\dots,m\}$.
\end{enumerate}
\end{lemma}

\vspace{0.5em}

The choice of the canonical complete flag allows as to obtain $\Q$-algebraic equations of Schubert varieties, as explained by next result.

\vspace{0.5em}

\begin{lemma}\label{lem:Q-schubert}
Let $\G_{m,n}\subset\R^{(m+n)^2}$ be a Grassmannian manifold and let $\lambda$ be a partition of the rectangle $m\times n$. Then the Schubert variety $\sigma_\lambda$ defined by the incidence conditions prescribed by $\lambda$ with respect to the canonical flag of $\R^{m+n}$ is a projectively $\Q$-closed $\Q$-algebraic subset of $\R^{(m+n)^2}$. 
\end{lemma}

\begin{proof}
We want to prove that $\sigma_\lambda$ is $\Q$-algebraic, namely we prove that conditions in (\ref{eq:sch}) are $\Q$-algebraic. Recall that $X\in\G_{m,n}$ is the matrix of the orthogonal projection of $\R^{m+n}$ onto an $m$-dimensional subspace $W$ of $\R^{m+n}$, hence $\ker(X-\id_{\R^{m+n}})=W$. This means that upper conditions on $\mr{rnk}(XD_\ell)$ correspond to lower conditions on $\mr{rnk}((X-\id_{\R^{m+n}})D_\ell)$, in particular for every $k\in\{1,\dots,m\}$ it holds:
\[
\mr{rnk}(XD_{n+k-\lambda_k})\geq k\quad\text{if and only if}\quad\mr{rnk}((X-\id_{\R^{m+n}})D_{n+k-\lambda_k})\leq n-\lambda_k.
\]
Latter condition is algebraic since it corresponds to the vanishing of the determinant of all $(n-\lambda_k+1)\times(n-\lambda_k+1)$-minors of the matrix $(X-\id_{\R^{m+n}})D_{n+k-\lambda_k}$. In particular, since $\G_{m,n}\subset\R^{(m+n)^2}$ is $\Q$-algebraic, $\id_{\R^{m+n}}$ and $D_{n+k-\lambda_k}$ are matrices with rational coefficients and the determinant is a polynomial with rational coefficients with respect to the entries of the matrix $X$, the algebraic set $\sigma_\lambda$ is $\Q$-algebraic. In addition, since $\G_{m,n}$ is projectively $\Q$-closed, $\sigma_\lambda$ is projectively $\Q$-closed as well by \cite[Lemma\,2.24(ii)]{GSa}.
\end{proof}

Let us introduce the notion of $\Q$-desingularization of a $\Q$-algebraic set $V\subset \R^n$.

\vspace{0.5em}

\begin{definition}\label{def:Q_des}
Let $V\subset \R^m$ be a $\Q$-algebraic set of dimension $d$. We say that $V'\subset \R^m\times\R^n$, for some $n\in\N$, is a \emph{desingularization} of $V$ if $V'\subset\R^{m+n}$  is a nonsingular algebraic set of dimension $d$ and $\pi|_{V'}:V'\to V$ is a birational map, where $\pi:\R^m\times\R^n\to\R^m$ is the projection onto the first factor. If, in addition, $V'\subset\R^{m+n}$ is a $\Q$-nonsingular projectively $\Q$-closed $\Q$-algebraic set we say that $V'$ is a \emph{$\Q$-desingularization} of $V$. %\bs
\end{definition}

\vspace{0.5em}

It is very well-known since Demazure \cite{Dem74} that Bott-Samelson method produces a desingularization of Schubert varieties. This algorithm has been deeply studied in literature, for instance Zelevinsky \cite{Zel83} proved that a precise choice on the order of the subvarieties to blow-up produces a small resolution. The goal of this section is to show that the Bott-Samelson method produces $\Q$-desingularizations of embedded Schubert varieties defined by incidence conditions with respect to the canonical complete flag of $\R^{m+n}$, that is, to prove the next result.

\vspace{0.5em}

\begin{theorem}\label{thm:Q-desingularization}
Let $\G_{m,n}\subset \R^{(m+n)^2}$ be a Grassmannian manifold and let $\sigma_{\lambda}$ be any Schubert variety of $\G_{m,n}$ defined by incidence conditions, prescribed by $\lambda$, with respect to the canonical complete flag $0\subset\R\subset\R^2\subset\dots\subset\R^{m+n}$ of $\R^{m+n}$. 
Then, $\sigma_{\lambda}$ admits a $\Q$-desingularization.
\end{theorem}

\vspace{0.5em}

\begin{comment}
, that is
\[
0\subset\R\subset\R^2\subset\dots\subset\R^{m+n}.
\]
\end{comment}

Previous desingularization theorem will play a crucial role in Lemma \ref{lem:Q_tico_bordism} that constitutes a preliminary construction used in the proof of the relative Nash-Tognoli theorem `over $\Q$', namely Theorem \ref{thm:Q_tico_tognoli} below. Let us provide a complete proof of Theorem \ref{thm:Q-desingularization}.

Let $m,n\in\N^*:=\N\setminus\{0\}$.  Let $\lambda=(\lambda_1,\dots,\lambda_m)$ be a partition together with its associated Young diagram in a $(m\times n)$-rectangle. Then, there are $c\in \N^*$,  $a_1,\dots,a_{c-1}\in\N^*$, $a_c,b_0\in\N$ and $b_1,\dots,b_{c-1}\in\N^*$, uniquely determined by $\lambda$, such that:
\begin{enumerate}[label=(\alph*),ref=(\alph*)]
\item\label{en:a} $a_1+\dots+a_c=m\quad\text{and}\quad b_0+\dots+b_{c-1}=n$,
\item\label{en:b} $\lambda_j=\sum_{k=i}^{c} b_k$ for every $j\leq a_i$ and for every $i=1,\dots,c$.
\end{enumerate}
The interpretation of the previous integers with respect to the Young diagram associated to the partition $\lambda$ is explained in Figure \ref{fig:partition}.

\begin{figure}[h!]
\centering

\tikzset{every picture/.style={line width=0.75pt}} %set default line width to 0.75pt        

\begin{tikzpicture}[x=0.75pt,y=0.75pt,yscale=-1,xscale=1]
%uncomment if require: \path (0,300); %set diagram left start at 0, and has height of 300

%Shape: Rectangle [id:dp8186503310018783] 
\draw   (192,41.81) -- (470,41.81) -- (470,237.81) -- (192,237.81) -- cycle ;
%Straight Lines [id:da08001770911426853] 
\draw [color={rgb, 255:red, 208; green, 2; blue, 27 }  ,draw opacity=1 ][line width=0.75]    (470,41.81) -- (470,59.81) ;
%Straight Lines [id:da7241588089531471] 
\draw [color={rgb, 255:red, 208; green, 2; blue, 27 }  ,draw opacity=1 ][line width=0.75]    (470,59.81) -- (430,59.81) ;
%Straight Lines [id:da4721759978729396] 
\draw [color={rgb, 255:red, 208; green, 2; blue, 27 }  ,draw opacity=1 ][line width=0.75]  [dash pattern={on 0.84pt off 2.51pt}]  (370,100.81) -- (350,100.81) ;
%Straight Lines [id:da7241873944949725] 
\draw [color={rgb, 255:red, 208; green, 2; blue, 27 }  ,draw opacity=1 ][line width=0.75]    (430,59.81) -- (430,67.81) -- (430,80.81) ;
%Straight Lines [id:da4555276687266103] 
\draw [color={rgb, 255:red, 208; green, 2; blue, 27 }  ,draw opacity=1 ][line width=0.75]  [dash pattern={on 0.84pt off 2.51pt}]  (290,150.81) -- (290,159.81) -- (290,171.81) ;
%Straight Lines [id:da09131608880011777] 
\draw [color={rgb, 255:red, 208; green, 2; blue, 27 }  ,draw opacity=1 ][line width=0.75]    (192,220.91) -- (230,220.72) ;
%Straight Lines [id:da16426790133593971] 
\draw [color={rgb, 255:red, 208; green, 2; blue, 27 }  ,draw opacity=1 ][line width=0.75]    (290,120.81) -- (290,150.81) ;
%Straight Lines [id:da6283944965117378] 
\draw [color={rgb, 255:red, 208; green, 2; blue, 27 }  ,draw opacity=1 ][line width=0.75]  [dash pattern={on 0.84pt off 2.51pt}]  (321,120.81) -- (341,120.81) ;
%Straight Lines [id:da9645239145098018] 
\draw [color={rgb, 255:red, 208; green, 2; blue, 27 }  ,draw opacity=1 ][line width=0.75]    (290,120.81) -- (321,120.81) ;
%Straight Lines [id:da11818112612788434] 
\draw [color={rgb, 255:red, 208; green, 2; blue, 27 }  ,draw opacity=1 ][line width=0.75]    (192,237.81) -- (192,220.91) ;
%Straight Lines [id:da6946823413401447] 
\draw [color={rgb, 255:red, 208; green, 2; blue, 27 }  ,draw opacity=1 ][line width=0.75]  [dash pattern={on 0.84pt off 2.51pt}]  (250,201.72) -- (258,201.72) -- (270,201.72) ;
%Straight Lines [id:da8114616689548707] 
\draw [color={rgb, 255:red, 208; green, 2; blue, 27 }  ,draw opacity=1 ][line width=0.75]    (430,80.81) -- (410,80.81) ;
%Straight Lines [id:da26166472488506853] 
\draw [color={rgb, 255:red, 208; green, 2; blue, 27 }  ,draw opacity=1 ][line width=0.75]    (410,80.81) -- (410,99.81) ;
%Straight Lines [id:da2493859121574309] 
\draw [color={rgb, 255:red, 208; green, 2; blue, 27 }  ,draw opacity=1 ][line width=0.75]    (410,99.81) -- (370,100.81) ;
%Straight Lines [id:da9838787247347176] 
\draw [color={rgb, 255:red, 208; green, 2; blue, 27 }  ,draw opacity=1 ][line width=0.75]    (230,201.72) -- (230,220.72) ;
%Straight Lines [id:da10347099955367722] 
\draw [color={rgb, 255:red, 208; green, 2; blue, 27 }  ,draw opacity=1 ][line width=0.75]    (230,201.72) -- (250,201.72) ;

% Text Node
\draw (463,25) node [anchor=north west][inner sep=0.75pt]    {$b_{0}$};
% Text Node
\draw (472,47) node [anchor=north west][inner sep=0.75pt]    {$a_{1}$};
% Text Node
\draw (442,43) node [anchor=north west][inner sep=0.75pt]    {$b_{1}$};
% Text Node
\draw (432,65) node [anchor=north west][inner sep=0.75pt]    {$a_{2}$};
% Text Node
\draw (300,103) node [anchor=north west][inner sep=0.75pt]    {$b_{i-1}$};
% Text Node
\draw (293,131) node [anchor=north west][inner sep=0.75pt]    {$a_{i}$};
% Text Node
\draw (195,224) node [anchor=north west][inner sep=0.75pt]    {$a_{c}$};
% Text Node
\draw (198,203) node [anchor=north west][inner sep=0.75pt]    {$b_{c-1}$};
% Text Node
\draw (482,140) node [anchor=north west][inner sep=0.75pt]    {$m$};
% Text Node
\draw (325,250) node [anchor=north west][inner sep=0.75pt]    {$n$};

\end{tikzpicture}
\caption{Disposition of the $a_i$'s and $b_i$'s with respect to the partition $\lambda$.}\label{fig:partition}
\end{figure}

\begin{remark}\label{rem:Q-grass}
Let $m',n'\in\N$ such that $m'\leq m$ and $n'\leq n$. Consider the Schubert variety of $\G_{m,n}$ associated to the partition $\lambda=(\lambda_1,\dots,\lambda_m)$ defined as:
\[
\lambda_k=
\begin{cases}
n \quad&\text{if } k\leq m-m',\\
n-n' &\text{if } k> m-m'.
\end{cases}
\]
If $m=m'$ and $n=n'$ the Schubert variety $\sigma_\lambda$ corresponds to the whole $\G_{m,n}$, otherwise $\sigma_\lambda$ is given by the equations:
\[
\sigma_\lambda=\{X\in\G_{m,n}\,|\,\mr{rnk}(XD_{m-m'})=m-m',\,\mr{rnk}(XD_{m+n'})=m\}.
\]
Observe that $\sigma_\lambda$ is biregular isomorphic to $\G_{m',n'}$. Actually, by our choice of the canonical complete flag of $\R^{m+n}$, $\sigma_\lambda$ is $\Q$-biregular isomorphic to $\G_{m',n'}$. Indeed, define the $\Q$-biregular isomorphism  $\varphi:\G_{m',n'}\to\sigma_\lambda\subset\G_{m,n}$ as follows: let $X':=(x'_{ij})_{i,j=1,\dots, m'+n'}$, then $\varphi(X')=(x_{i,j})_{i,j=1,\dots, m+n}$ with
\[
x_{ij}=
\begin{cases}
1\quad\quad\text{if $i=j$ and $i\leq m-m'$;}\\
x'_{st}\quad\text{if $m-m'<i,j<m+n'$ and $s=i-m+m'$, $t=j-m+m'$;}\\
0\quad\quad\text{otherwise.}
\end{cases}
\]
Recall that $\G_{m',n'}\subset\R^{(m'+n')^2}$ is a projectively $\Q$-closed $\Q$-nonsingular $\Q$-algebraic set. Let $\textnormal{graph}(\varphi)\subset\G_{m',n'}\times\G_{m,n}$ be the graph of $\varphi$. Then, $\textnormal{graph}(\varphi)\subset\R^{(m'+n')^2+(m+n)^2}$ is a $\Q$-algebraic set contained in $\G_{m',n'}\times\G_{m,n}$, hence projectively $\Q$-closed by \cite[Lemma\,2.24(ii)]{GSa} and $\Q$-nonsingular since $\varphi$ is a $\Q$-biregular isomorphism. Thus, $\textnormal{graph}(\varphi)\subset\R^{(m'+n')^2+(m+n)^2}$ is a $\Q$-desingularization of $\sigma_{\lambda}$. %\bs
\end{remark}

\vspace{0.5em}

By the above Remark \ref{rem:Q-grass} we are left to find $\Q$-desingularizations of Schubert varieties $\sigma_\lambda$ of $\G_{m,n}$ defined by incidence conditions with respect to the canonical complete flag such that $a_c$ and $b_0$ are non-null. Indeed, if $\lambda$ is a partition with $a_c,b_0=0$, then $\sigma_\lambda$ is $\Q$-biregularly isomorphic to a Schubert variety $\sigma_{\mu}$ of $\G_{m-a_1,n-b_{c-1}}$ with $\mu_i:=\lambda_{i+a_1}-b_{c-1}$ for every $i\in\{1,\dots,m-a_1\}$. Hence, $\mu_{1}:=\lambda_{1+a_1}-b_{c-1}=n-b_1-b_{c-1}<n-b_{c-1}$ and $\mu_{m-a_1}:=b_{c-1}-b_{c-1}=0$, as desired.

We define the \emph{depressions} of the partition $\lambda$, with $a_c,b_0>0$, as the elements of the Young diagram whose coordinates, with respect to the upper corner on the left, are:
\[
(a_1+\dots+a_{i}+1,b_i+\dots+b_{c-1}+1),\quad i=1,\dots,c-1.
\]
Here we provide an inductive desingularization of the Schubert variety $\sigma_{\lambda}$ with respect to the number $c-1\in\N$ of depressions of the partition $\lambda$.

In next result we translate to our real embedded setting the Bott-Samelson desingularization technique in \cite{Dem74,Zel83}.

\vspace{0.5em}

\begin{lemma}\label{lem:des}
Let $\lambda$ be a partition of the $(m\times n)$-rectangle such that $a_c$ and $b_0$ are non-null. Let $\sigma_{\lambda}$ be the Schubert variety of $\G_{m,n}$ defined by incidence conditions, prescribed by $\lambda$, with respect to the canonical complete flag of $\R^{(m+n)^2}$. Let $m_k:=\sum_{i=1}^{k} a_i$, $n_k:=m+n-m_k$ and $d_k:=m_k+\sum_{i=1}^{k}b_{i-1}$ for every $k=1,\dots,c$.

Then the embedded Bott-Samelson algebraic set:
\begin{align*}
Z_\lambda:=\{(X,Y_{c-1},\dots,Y_1)&\in \G_{m,n}\times\G_{m_{c-1},n_{c-1}}\times\dots\times\G_{m_1,n_1}\,|\\
Y_i D_{d_i}&=Y_i,\quad\text{for every $i=1,\dots,c-1$},\\
Y_{i+1} Y_{i}&=Y_{i},\quad\text{for every $i=1,\dots,c-2$},\\
XY_{c-1}&=Y_{c-1}\}.
\end{align*}
is a desingularization of $\sigma_{\lambda}$.
\end{lemma}

\begin{proof}
Let us prove by induction on $c\in\N^*:=\N\setminus\{0\}$. Let $c=1$, that is $a_1,b_0>0$ and $\lambda$ has no depressions, so $\lambda$ is the null partition. Thus, $\sigma_{\lambda}=\G_{m,n}\subset\R^{(m+n)^2}$, which is a nonsingular algebraic set, thus there is nothing to prove.

Let $c>1$ and $\lambda$ be a partition with $c-1$ depressions such that $a_c,b_0>0$. Recall that the Schubert variety $\sigma_\lambda$ defined by the incidence conditions, prescribed by $\lambda$, with respect to the canonical complete flag of $\R^{m+n}$ is defined as:
\[
\sigma_\lambda=\{X\in\G_{m,n}\,|\,\mr{rnk}(XD_{d_k})\geq m_k\quad\text{for $k=1,\dots,c$.}\}
\]
Consider the algebraic set $Z_\lambda\subset\G_{m,n}\times\G_{m_{c-1},n_{c-1}}\times\dots \times\G_{m_1,n_1}$ as in the statement of Lemma \ref{lem:des}. Define $\pi_i:Z_\lambda\to \G_{m_{c-i},n_{c-i}}$ for $i\in\{1,\dots,c\}$ be the restriction over $Z_\lambda$ of the projection from $\G_{m,n}\times\G_{m_{c-1},n_{c-1}}\times\dots \times\G_{m_1,n_1}$ onto the $(c-i+1)$-component.

Observe $\pi_1(Z_\lambda)=\{Y_1\in\G_{m_1,n_1}\,|\,Y_1 D_{d_1}=Y_1\}$ is biregular isomorphic to $\G_{a_1,b_0}=\G_{m_1,d_1-m_1}$. Let $\mu$ be a partition of the $\big((m-m_1)\times n\big)$-rectangle defined as: $\mu=(\mu_1,\dots,\mu_{m-a_1})$ with
\[
\mu_k= \lambda_{k+a_1} \text{ for every $k=1,\dots,m-a_1$}.
\]
Then, for every $B_1\in\pi_1(Z_\lambda)$, we observe that $\pi_1^{-1}(B_1)$ is biregular isomorphic to the set $Z_\mu$. Indeed, define the biregular isomorphism $\phi:Z_\mu \to(\pi_1)^{-1}(D_{m_1})$ as follows: let $(A,B_{c-1},\dots,B_2)\in Z_\mu$, then define
\[
\phi(A,B_{c-1},\dots,B_2):=(\varphi(A),\varphi(B_{c-1}),\dots,\varphi(B_2),D_{m_1}),
\]
where $\varphi:\R^{(m-m_1+n)^2}\to \R^{(m+n)^2}$ is defined as $\varphi((x_{st})_{s,t})=(x'_{ij})_{i,j}$, with
\begin{align*}
x'_{ij}:=
\begin{cases}
1\quad \text{if $i=j$ and $i\leq m_1$},\\
x_{st}\,\,\text{if $m-m_1<i,j$, with $s=i-m+m_1$ and $t=j-m+m-1$},\\
0\quad\text{otherwise}.
\end{cases}
\end{align*}
Moreover, for every $B_1\in\pi_1(Z_\lambda)$, then $(\pi_1)^{-1}(B_1)$ is biregularly isomorphic to $(\pi_1)^{-1}(D_{m_1})$, indeed it suffices to chose $G\in O(m+n)$ such that $D_{m_1}=G^{T} B_1 G$ and apply $G$ to every factor of $(\pi_1)^{-1}(D_{m_1})$ to produce the wondered isomorphism. Observe that the partition $\mu$ has exactly $(c-2)$-depressions, indeed it is constructed by erasing the first depression $(a_1+1,n-b_0+1)$ of $\lambda$, thus by inductive assumption the algebraic set $Z_\mu$ is a desingularization of $\sigma_\mu$. In particular:
\[
\dim(Z_\mu)=\dim(\sigma_\mu)=\dim(\sigma_\lambda)-a_1b_0.
\]
Hence, $\pi_1:Z_\lambda\to\G_{a_1,b_0}$ is an algebraic fibre bundle of dimension $\dim(\sigma_\lambda)$, thus $Z_\lambda$ is a nonsingular algebraic subset of $\R^{(m+n)^2 c}$ of dimension $\dim(\sigma_\lambda)$. Moreover, $Z_\lambda$ is a desingularization of $\sigma_\lambda$ indeed, if $A\in\Omega_\lambda$, then $(A,B_{c-1},\dots,B_1)\in Z_\lambda$ if and only if $B_i=A D_{d_i}$ for every $i\in\{1,\dots,c-1\}$. Hence, the map $\pi_c:Z_\lambda\to \sigma_\lambda$ is birational by Lemma \ref{lem:schubert}\ref{lem:schubert_1}.
\end{proof}

By Remark \ref{rem:Q-grass}, in order to prove Theorem 	\ref{thm:Q-desingularization} we are only left to prove that the explicit equations of each embedded Bott-Samelson algebraic set $Z_\lambda$ of Lemma \ref{lem:des} actually describe the local smooth behavior of $Z_\lambda$ at any point as well. %We outline that the equations of $Z_\lambda$ explicitly obtained in Lemma \ref{lem:des} 

%produce a $\Q$-desingularization of every Schubert variety $\sigma_\lambda$ defined by incidence conditions, prescribed by $\lambda$, with respect to the canonical complete flag of $\R^{m+n}$,  we are only left to prove the following result.

\vspace{0.5em}

\begin{lemma} \label{lem:Q-Z_{lambda}}
Each embedded Bott-Samelson algebraic set $Z_{\lambda}\subset\R^{(m+n)^2 c}$ as in Lemma \ref{lem:des} is a projectively $\Q$-closed $\Q$-nonsingular $\Q$-algebraic set.
\end{lemma}

\begin{proof}
By definition, $Z_\lambda$ is a $\Q$-algebraic subset of $\R^{(m+n)^2 c}$ defined by the following equations in the variables $X:=(x_{ij})_{i,j=1,\dots,m+n}$ and $Y_k:=(y^{(k)}_{ij})_{i,j=1,\dots,m+n}$, for $k=1,\dots,c-1$:
\begin{align*}
X&=X^T,\quad X^2=X,\quad \mr{tr}(X)=m;\\
Y_k&=Y_k^T,\quad Y_k^2=Y_k,\quad \mr{tr}(Y_k)=m_k\quad\text{for every $k=1,\dots,c-1$};\\
Y_k D_{d_k}&=Y_k \quad\text{for every $k=1,\dots,c-1$};\\
Y_{k+1}Y_{k}&=Y_{k}\quad\text{for every $k=1,\dots,c-2$};\\
XY_{c-1}&=Y_{c-1}.
\end{align*}
Let $\varphi_k:\R^{(m+n)^2 c} \to\R^{(m+n)^2}\times\R^{(m+n)^2}\times\R^{(m+n)^2}\times\R^{(m+n)^2}$, for every $k=1,\dots,c-2$, $\varphi_{c-1}:\R^{(m+n)^2 c}\to\R^{(m+n)^2}\times\R^{(m+n)^2}\times\R^{(m+n)^2}\times\R^{(m+n)^2}$ and $\varphi_{c}:\R^{(m+n)^2 c}\to\R^{(m+n)^2}\times\R^{(m+n)^2}$ be defined as:
\begin{align*}
\varphi_k(X,Y_{c-1},\dots,Y_1)&:=(Y_k-Y_k^T,Y_k^2-Y_k,Y_kD_{d_k}-Y_k,Y_{k+1}Y_{k}-Y_{k}),\\
\varphi_{c-1}(X,Y_{c-1},\dots,Y_1)&:=(Y_{c-1}-Y_{c-1}^T,Y_{c-1}^2-Y_{c-1},Y_{c-1}D_{d_{c-1}}-Y_{c-1}, XY_{c-1}-Y_{c-1}),\\
\varphi_{c}(X,Y_{c-1},\dots,Y_1)&:=(X-X^T,X^2-X).
\end{align*}
Define $\phi:\R^{(m+n)^2 c}\to(\R^{(m+n)^2}\times\R^{(m+n)^2}\times\R^{(m+n)^2}\times\R^{(m+n)^2}) ^{c-1}\times\R^{(m+n)^2}\times\R^{(m+n)^2}$ be the polynomial map:
\begin{align*}
\phi(X,Y_{c-1},\dots,Y_1):=(&\varphi_1(X,Y_{c-1},\dots,Y_1),\dots,\varphi_{c-1}(X,Y_{c-1},\dots,Y_1),\varphi_c(X,Y_{c-1},\dots,Y_1)).
\end{align*}
We prove that the polynomials $\mr{tr}(X)-m$, $\mr{tr}(Y^k)-m_k$, for every $k=1,\dots,c-1$, and the polynomial components of $\phi$ do suffice to describe the local structure of nonsingular points of $Z_{\lambda}$ in $\R^{(m+n)^2 c}$. Since these polynomials have coefficients in $\Q$ and their common zero set is $Z_{\lambda}$, bearing in mind that 
\[
\dim(Z_{\lambda})=\sum_{k=1}^{c} \dim(\G_{a_k,n-\sum_{i=k}^{c-1}b_k})=\sum_{k=1}^{c} a_k \left(n-\sum_{i=k}^{c-1}b_k\right)=\dim(\sigma_{\lambda}),
\]
it suffices to show that, for each $(A,B_{c-1},\dots,B_1)\in Z_{\lambda}$, the rank of the Jacobian matrix $J_\phi(A,B_{c-1},\dots,B_1)$ of $\phi$ at $(A,B_{c-1},\dots,B_1)$ is greater than or equal to (and hence equal to) 
\begin{align*}
c(m+n)^2-\dim(\sigma_{\lambda})&=\sum_{k=1}^{c}(m+n)^2-\dim(\G_{a_k,n-\sum_{i=k}^{c-1}b_k})\\
&=\sum_{k=1}^{c}(m+n)^2-a_k (d_k-m_k),
\end{align*}
i.e. $\mr{rnk}\,J_\phi(A,B_{c-1},\dots,B_1)\geq c(m+n)^2-\sum_{k=1}^{c}a_k (d_k-m_k)$ for all $(A,B_{c-1},\dots,B_1)\in Z_{\lambda}$.

First, we prove that $\mr{rnk}\,J_\phi(D_m,D_{m_{c-1}},\dots,D_{m_1})\geq c(m+n)^2-\dim(\sigma_\lambda)$ if $D_m=D_{m_c}$ and $D_{m_k}$ are the diagonal matrices in $\R^{(m+n)^2}$ having $1$ in the first $m_k$ diagonal positions and $0$ otherwise, for every $k=1,\dots,c$. Observe that $(D_m,D_{m_{c-1}},\dots,D_{m_1})\in Z_{\lambda}$ since $D_{m_{k+1}}D_{m_k}=D_{m_k}$, for every $k=1,\dots,c-1$.

For each $i,j\in\{1,\ldots,m+n\}$ and $k\in\{1,\dots,c\}$, define the polynomial functions $f^{(k)}_{ij},g^{(k)}_{ij},p^{(k)}_{ij},q^{(k)}_{ij}:\R^{(m+n)^2}\times\R^{(m+n)^2}\to\R$ by
\begin{align*}
f^{(c)}_{ij}(X,Y_{c-1},\dots,Y_1)&:=x_{ij}-x_{ji},\\ g^{(c)}_{ij}(X,Y_{c-1},\dots,Y_1)&:=\textstyle\big(\sum_{\ell=1}^n x_{i\ell}x_{\ell j}\big)-x_{ij},\\
f^{(k)}_{ij}(X,Y_{c-1},\dots,Y_1)&:=y^{(k)}_{ij}-y^{(k)}_{ji},\\ g^{(k)}_{ij}(X,Y_{c-1},\dots,Y_1)&:=\textstyle\big(\sum_{\ell=1}^n y^{(k)}_{i\ell}y^{(k)}_{\ell j}\big)-y^{k}_{ij},\quad\quad\quad\quad\quad\quad\quad\quad
\end{align*}
\begin{align*}
p^{(k)}_{ij}(X,Y_{c-1},\dots,Y_1)&:=
\begin{cases}
0\quad&\text{if $i,j\leq d_k=\sum_{\ell=1}^{k} (a_\ell+b_{\ell-1})$};\\
-y^{(k)}_{ij}&\text{otherwise},
\end{cases}\\
q^{(c)}_{ij}(X,Y_{c-1},\dots,Y_1)&:=y^{(c-1)}_{ij}-\sum_{\ell=1}^{m+n} x_{i\ell}y^{(c-1)}_{\ell j},\\
q^{(k)}_{ij}(X,Y_{c-1},\dots,Y_1)&:=y^{(k)}_{ij}-\sum_{\ell=1}^{m+n} y^{(k+1)}_{i\ell}y^{(k)}_{\ell j}\quad\text{with $k\neq 1,c$.}
\end{align*}
for all $(X,Y_{c-1},\dots,Y_1)=((x_{ij})_{i,j},(y^{(c-1)}_{ij})_{i,j},\dots,(y^{(1)}_{ij})_{i,j})\in\R^{(m+n)^2 c}$. It follows that
\begin{align*}
\phi(X,Y_{c-1},\dots,Y_1)=\Big(&(f^{(1)}_{ij}(X,Y_{c-1},\dots,Y_1))_{i,j},(g^{(1)}_{ij}(X,Y_{c-1},\dots,Y_1))_{i,j},\\
&(p^{(1)}_{ij}(X,Y_{c-1},\dots,Y_1))_{i,j},(q^{(2)}_{ij}(X,Y_{c-1},\dots,Y_1))_{i,j},\\
&\dots,\\
&(f^{(c-1)}_{ij}(X,Y_{c-1},\dots,Y_1))_{i,j},(g^{(c-1)}_{ij}(X,Y_{c-1},\dots,Y_1))_{i,j},\\
&(p^{(c-1)}_{ij}(X,Y_{c-1},\dots,Y_1))_{i,j},(q^{(c)}_{ij}(X,Y_{c-1},\dots,Y_1))_{i,j},\\
&(f^{(c)}_{ij}(X,Y_{c-1},\dots,Y_1))_{i,j},(g^{(c)}_{ij}(X,Y_{c-1},\dots,Y_1))_{i,j}\Big).
\end{align*}
Define, for every $k\in\{1,\dots,c\}$:
\begin{align*}
S^{(k)}_1&:=\{(i,j)\in\{1,\ldots,m+n\}^2\,|\,i<j\leq d_k\},\\
S^{(k)}_2&:=\{(i,j)\in\{1,\ldots,m+n\}^2\,|\,i\leq j\leq m_k\},\\
S^{(k)}_3&:=\{(i,j)\in\{1,\ldots,m+n\}^2\,|\,m_k<i\leq j\leq d_k\},\\
S^{(k)}_4&:=\{(i,j)\in\{1,\ldots,m+n\}^2\,|\,d_k< i \text{ or } d_k< j\},\\
T^{(1)}&:=\varnothing,\\
T^{(k)}&:=\{(i,j)\in\{1,\ldots,m+n\}^2\,|\,m_k<i\leq d_k,\, j\leq m_{k-1}\}.%\quad\text{with $k\neq 1$}.\\
\end{align*}
Notice that the sum of the cardinalities of $S^{(k)}_1$, $S^{(k)}_2$, $S^{(k)}_3$ and $S^{(k)}_4$ equals
\begin{align*}
\frac{(d_k-1)d_k}{2}+\frac{m_k(m_k+1)}{2}+&\frac{(d_k-m_k)(d_k-m_k+1)}{2}+(m+n)^2-d_k^2\\
&=(m+n)^2-m_k(d_k-m_k),
\end{align*}
for every $k\in\{1,\dots,c\}$. In particular, the sum of the cardinalities of $S^{(1)}_1$, $S^{(1)}_2$, $S^{(1)}_3$ and $S^{(1)}_4$ is equal to $a_1b_0$. In addition, the cardinality of $T^{(k)}$ is equal to $m_{k-1}(d_k-m_k)$, for every $k\in\{2,\dots,c\}$. Hence the sum of the cardinalities of $S^{(k)}_1$, $S^{(k)}_2$, $S^{(k)}_3$, $S^{(k)}_4$ and $T^{(k)}$ equals $(m+n)^2-a_k(d_k-m_k)$, for every $k\in\{2,\dots,c\}$.

By a direct computation, we see that
\[
\begin{array}{ll}
\nabla f^{(1)}_{ij}(D_m,D_{m_{c-1}},\dots,D_{m_1})=(0,\dots,0,E^{(1)}_{ij}-E^{(1)}_{ji}) & \text{ if $(i,j)\in S^{(1)}_1$,}\\
\nabla g^{(1)}_{ij}(D_m,D_{m_{c-1}},\dots,D_{m_1})=(0,\dots,0,E^{(1)}_{ij}) & \text{ if $(i,j)\in S^{(1)}_2$,}\\
\nabla g^{(1)}_{ij}(D_m,D_{m_{c-1}},\dots,D_{m_1})=(0,\dots,0,-E^{(1)}_{ij}) & \text{ if $(i,j)\in S^{(1)}_3$,}\\
\nabla p^{(1)}_{ij}(D_m,D_{m_{c-1}},\dots,D_{m_1})=(0,\dots,0,E^{(1)}_{ij}) & \text{ if $(i,j)\in S^{(1)}_4$,}\\
\end{array}
\]
and, for every $k\in\{2,\dots,c\}$
\[
\begin{array}{ll}
\begin{comment}
\nabla f^{(k)}_{ij}(D_m,D_{m_{c-1}},\dots,D_{m_1})=(0,\dots,0,E^{(k)}_{ij},0,\dots,0) &\\
\quad\quad\quad\quad\quad\quad\quad\quad\quad\quad\quad\,\,\,\,\,\,\,\,\, +\,(0,\dots,0,-E^{(k)}_{ji},0,\dots,0) & \text{ if $(i,j)\in S^{(k)}_1$,}\\
\end{comment}
\nabla f^{(k)}_{ij}(D_m,D_{m_{c-1}},\dots,D_{m_1})=(0,\dots,0,E^{(k)}_{ij}-E^{(k)}_{ji},0,\dots,0) & \text{ if $(i,j)\in S^{(k)}_1$,}\\
\nabla g^{(k)}_{ij}(D_m,D_{m_{c-1}},\dots,D_{m_1})=(0,\dots,0,E^{(k)}_{ij},0,\dots,0) & \text{ if $(i,j)\in S^{(k)}_2$,}\\
\nabla g^{(k)}_{ij}(D_m,D_{m_{c-1}},\dots,D_{m_1})=(0,\dots,0,-E^{(k)}_{ij},0,\dots,0) & \text{ if $(i,j)\in S^{(k)}_3$,}\\
\end{array}
\]
\[
\begin{array}{ll}
\nabla p^{(k)}_{ij}(D_m,D_{m_{c-1}},\dots,D_{m_1})=(0,\dots,0,E^{(k)}_{ij},0,\dots,0) & \text{ if $(i,j)\in S^{(k)}_4$,}\\
\nabla q^{(k)}_{ij}(D_m,D_{m_{c-1}},\dots,D_{m_1})=(0,\dots,0,-E^{(k)}_{ij},0,\dots,0) & \text{ if $(i,j)\in T^{(k)}$,}\\
\end{array}
\]
where $E^{(k)}_{ij}$ is the matrix in $\R^{(m+n)^2}$ whose $(i,j)$-coefficient equals $1$ and the other coefficients are $0$ holding the $(c-k+1)$-position in the vector $(X,Y_{c-1},\dots,Y_1)\in\R^{(m+n)^2 c}$, for every $k\in\{1,\dots,c\}$. Consequently, we have that 
\begin{align*}
\mr{rnk}\,J_\phi(D_m,D_{m_{c-1}},\dots,D_{m_1})&\geq \sum_{k=1}^{c}((m+n)^2-a_k(d_k-m_k))\\
&=c(m+n)^2-\dim(\sigma_{\lambda}).
\end{align*}
Let $(A,B_{c-1},\dots,B_1)\in Z_{\lambda}$ and let $G\in O(m+n)$ be such that $D_m=G^TAG$ and $D_{m_k}=G^TB_kG$, for every $k\in\{1,\dots,c-1\}$. Define the linear automorphisms $\psi:\R^{(m+n)^2}\to\R^{(m+n)^2}$ by $\psi(X):=G^TXG$ and $\psi^{\times k}:\R^{(m+n)^2 k}\to \R^{(m+n)^2 k}$ by $\psi^{\times k}(X_1,\dots,X_k):=(\psi(X_1),\dots,\psi(X_k))$, for $k\in \N^*$. Since $\psi(A)=D_m$ and $(\psi^{\times (4c-2)})\circ\phi=\phi\circ(\psi^{\times c})$, we have that
\begin{align*}
J_{\psi^{\times (4c-2)}}(\phi(A,B_{c-1}&,\dots,B_1))J_\phi(A,B_{c-1},\dots,B_1)=\\
&J_\phi(D_m,D_{m_{c-1}},\dots,D_{m_1})J_{\psi^{\times c}}(A,B_{c-1},\dots,B_1).
\end{align*}
Bearing in mind that both matrices $J_{\psi^{\times (4c-2)}}(\phi(A,B_{c-1},\dots,B_1))$ and $J_{\psi^{\times c}}(\\A,B_{c-1},\dots,B_1)$ are invertible, it follows that
\begin{align*}
\mr{rnk}\,J_\phi(A,B_{c-1},\dots,B_1)&=\mr{rnk}\,J_\phi(D_m,D_{m_{c-1}},\dots,D_{m_1})\\
&\geq c(m+n)^2-\dim(\sigma_{\lambda}),
\end{align*}
as desired. Since $Z_{\lambda}\subset\R^{(m+n)^2 c}$ is $\Q$-algebraic and is contained in the projectively $\Q$-closed $\Q$-algebraic set $\G_{m,n}\times\G_{m_{c-1},n_{c-1}}\times\dots\times\G_{m_1,n_1}\subset\R^{(m+n)^2 c}$, \cite[Lemma\,2.24(ii)]{GSa} ensures that $Z_{\lambda}\subset\R^{(m+n)^2 c}$ is a  projectively $\Q$-closed $\Q$-algebraic set as well. This proves that $Z_\lambda\subset\R^{(m+n)^2 c}$ is a projectively $\Q$-closed $\Q$-nonsingular $\Q$-algebraic set, as desired.
\end{proof}

A combination of Remark \ref{rem:Q-grass} and Lemmas \ref{lem:des} \& \ref{lem:Q-Z_{lambda}} provides a complete proof of Theorem \ref{thm:Q-desingularization}.

\vspace{0.5em}

\begin{remark}\label{rem:Q-des}
We point out that Theorem \ref{thm:Q-desingularization} follows as well by general algorithms for resolution of singularities, indeed the invariance under field extension is explicitly mentioned in \cite[Section\,5.7]{Wlo05}, \cite[Section\,3.34.2]{Kol07} and \cite[Theorem\,1.1]{BM08}. However, there are some points to clarify. These algorithms in characteristic $0$ actually work by resolving the marked ideal associated to an algebraic variety so, in our setting, we actually resolve the vanishing ideal $\Ii_\Q(V)$ of an affine $\Q$-algebraic set $V\subset\R^n$. We recall that in general $\Ii_\Q(V)\subset \Ii(V)$, so in principle the resolution will differ with respect to the classical Hironaka resolution for the algebraic set $V\subset \R^n$. The algorithm is claimed, in all mentioned references, to be stable under field extension. As explained in \cite{Wlo05}, a possible strategy is to deduce the invariance under field extension in three steps: first W\l odarczyk deals with extensions of algebraically closed fields in Section 5.4, essentially by using Chevalley's Theorem as a transfer principle, then the author proves the statement for the extension  $\overline{K}|K$ of a field $K$ of characteristic $0$, see Section 5.7, and finally in the same section W\l odarczyk deduces the general result by a scheme theoretical argument. The invariance under algebraic closure in step two is not explicitly proved but it can be deduced by proving that, if a marked ideal $\underline{\mathcal{I}}$ over $\overline{K}$ is $K$-equivariant, then the derivative ideal $\mathcal{D}(\underline{\mathcal{I}})$ of $\underline{\mathcal{I}}$ over $\overline{K}$ is $K$-equivariant as well. Mentioned result is not obvious, indeed the derivative ideal of a marked ideal is locally defined so the local étale description of derivatives naturally involves coefficients over $\overline{K}$. However, the $K$-equivariance of $\mathcal{D}(\underline{\mathcal{I}})$, when $\underline{\mathcal{I}}$ is $K$-equivariant, can be proven by a computation applying Kähler differentials. The latter property allows us to run the algorithm locally preserving the coefficients of the involved equations over $K$. Then, the local centers at each stage of the resolution have to be globalized in such a way the center of each blow-up is globally a $K$-nonsingular $K$-algebraic set, obtaining finally the invariance under the field extension $\overline{K}|K$ by induction. The proof in \cite{Kol07} uses a similar argument based on the theory of schemes and in \cite{BM08} the result is only claimed.  However, as previously mentioned, the application of scheme theoretical arguments is not really satisfying for a concrete point of view in the resolution procedure. We remark as well that our notion of $\R|\Q$-regularity turns out to be satisfied by the resulting real nonsingular algebraic set produced by the considered algorithms for resolution of singularities, see \cite[Proposition\,2.5.1]{Wlo05} and \cite[Appendix\,C]{GSa} for more details. 

As already mentioned, the complexity of the algorithms proposed in \cite{Wlo05,Kol07,BM08} is very high since the procedure is local and it has to be globalized by comparing all the local resolutions so it becomes very hard in practice to apply the algorithm to a given algebraic set, see \cite{FKP} for more details. On the contrary, Bott-Samelson resolution of singular Schubert varieties is much more elementary, very explicit and the complexity is controlled. Indeed, the number of steps to resolve the singularities of a Schubert variety turns out to be exactly the number of depressions of the associated Young diagram, which is at most $c=\min\{m,n\}$ for each Schubert variety of $\G_{m,n}$. Hence, such a simple and explicit $\Q$-desingularization of real Schubert varieties has interest on its own.
\end{remark}

\subsection{Real embedded Grassmannians have projectively $\Q$-algebraic homology}\label{subsec:2.2}

In \cite[Definition\,2.38]{GSa}, Ghiloni and the author introduced a variant `over $\Q$' of the classical algebraic representatives of homology classes of a topological space. Here we briefly recall the notion of `projectively $\Q$-closed $\Q$-algebraic homology'. Let $W\subset\R^k$ be a set endowed with the induced Euclidean topology. Given $p\in\N$ and $\alpha\in H_p(W,\Z/2\Z)$, we say that $\alpha$ is \emph{projectively $\Q$-algebraic} if there exist a $p$-dimensional projectively $\Q$-closed $\Q$-nonsingular $\Q$-algebraic set $Z\subset\R^h$ and a $\Q$-regular map $g:Z\rightarrow W$ such that $g_*([Z])=\alpha$, where $[Z]$ is the fundamental class of $Z$ in $H_p(Z,\Z/2\Z)$.

\vspace{0.5em}

\begin{definition}\label{def:Q-homology}
Given $d\in\N$, we say that $W$ has \emph{projectively $\Q$-algebraic homology} if, for all $p\in\{0,\ldots,d\}$ and for all $\alpha\in H_p(W,\Z/2\Z)$, the homology class $\alpha$ is projectively $\Q$-algebraic. %\bs
\end{definition}

\vspace{0.5em}

The aim of this subsection is to prove that real embedded Grassmannians have projectively $\Q$-algebraic homology. Let us fix some notation about CW complexes. Let $X$ be a topological space endowed by a finite CW complex structure $\mathcal{S}$ of dimension $d$. We denote by $\mathcal{S}^{(k)}$ the set of open $k$-cells of $\mathcal{S}$, for every $k\in\{0,\dots,d\}$. Denote by $X_k:=\bigcup_{\Omega\in\mathcal{S}^{(k)}} \overline{\Omega}$ the $k$-skeleton of $X$ for every $k\in\{0,\dots,d\}$, and $X_{-1}:=\varnothing$. Define $C_k(\mathcal{S},\Z/2\Z):=H_k(X_k,X_{k-1})$ the group of unoriented cellular $k$-chains of $\mathcal{S}$ for every $k\in\{1,\dots,d\}$. Let $\partial_k^\mathcal{S}:C_k(\mathcal{S},\Z/2\Z)\to C_{k-1}(\mathcal{S},\Z/2\Z)$ denote the boundary operator in cellular homology for every $k\in\{1,\dots,d\}$. Define the \emph{$k$-cellular homology group of $X$ (with coefficients in $\Z/2\Z$)} as $H_k(\mathcal{S},\Z/2\Z):=\ker(\partial^\mathcal{S}_k)/\im(\partial^\mathcal{S}_{k+1})$. For more details about CW complexes and their homological theory we refer to \cite{LW69}.

\vspace{0.5em}

\begin{lemma}\label{lem:CW}
Let $W\subset\R^{n}$ be a compact algebraic subset of dimension $d$. Suppose that $W$ admits a finite CW complex structure $\mathcal{S}$ such that the closure of each open cell $\Omega\in\mathcal{S}^{(k)}$ is algebraic for every $k\in\{0,\dots,d\}$. Then,
\[
H_k(W,\Z/2\Z)=\textnormal{Span}\big(\{[\overline{\Omega}]\in H_k(W,\Z/2\Z)\,|\,\Omega\in\mathcal{S}^{(k)}\}\big).
\]
and $\{[\overline{\Omega}]\in H_k(W,\Z/2\Z)\,|\,\Omega\in\mathcal{S}^{(k)}\}$ is a basis of $H_k(W,\Z/2\Z)$ as a $\Z/2\Z$-vector space for every $k\in\{0,\dots,d\}$.
\end{lemma}

\begin{proof}
By classical arguments about cellular and simplicial homology, $\big\{[\overline{\Omega}]\in H_k(\mathcal{S},\Z/2\Z)\,|\,\Omega\in\mathcal{S}^{(k)}\big\}$ constitutes a system of generators of $H_k(\mathcal{S},\Z/2\Z)$ for every $k=0,\dots,d$. We are only left to prove that $\big\{[\overline{\Omega}]\in H_k(\mathcal{S},\Z/2\Z)\,|\,\Omega\in\mathcal{S}^{(k)}\big\}$ is linearly independent over $\Z/2\Z$. Since $\overline{\Omega}$ is algebraic for every open cell $\Omega\in\mathcal{S}^{(k)}$, for every $k\in\{0,\dots,d\}$, the fundamental class $[\overline{\Omega}]$ of $\overline{\Omega}$ is a well defined homology class in $H_k(W,\Z/2\Z)$, see \cite[\S\,11.3]{BCR98}. Suppose $\Omega\in\mathcal{S}^{(k)}$, then for every $\Omega'\in\mathcal{S}^{(k+1)}$ we have $\partial_{k+1}^{\mathcal{S}}(\overline{\Omega}')=0$, since $\overline{\Omega}'$ is algebraic as well. Hence, we get that $[\overline{\Omega}]\in H_k(\mathcal{S},\Z/2\Z)$ is non-null and linearly independent with respect to $\big\{[\overline{\Omega}']\in H_k(\mathcal{S},\Z/2\Z)\,|\,\Omega'\in\mathcal{S}^{(k)}\,\text{and}\,\Omega'\neq\Omega\big\}$ for every choice of $\Omega\in\mathcal{S}^{(k)}$ and $k\in\{0,\dots,d\}$. This proves that $\big\{[\overline{\Omega}]\in H_k(\mathcal{S},\Z/2\Z)\,|\,\Omega\in\mathcal{S}^{(k)}\big\}$ is a basis of $H_k(\mathcal{S},\Z/2\Z)$, then $\big\{[\overline{\Omega}]\in H_k(W,\Z/2\Z)\,|\,\Omega\in\mathcal{S}^{(k)}\big\}$ it is also a basis of $H_k(W,\Z/2\Z)$, as desired.
\end{proof}

Following the notation of Section \ref{subsec:2.1}, we refer to embedded Schubert varieties $\sigma_{\lambda}$ of $\G_{m,n}\subset\R^{(m+n)^2}$ defined by incidence conditions, prescribed by $\lambda$, with respect to the canonical complete flag of $\R^{m+n}$. Denote by $|\lambda|:=\sum_{i=1}^m \lambda_i$.

\vspace{0.5em}

\begin{corollary}\label{cor:homology}
Let $\G_{m,n}\subset\R^{(m+n)^2}$. Then:
\[
H_k(\G_{m,n},\Z/2\Z)=\textnormal{Span}\big(\big\{[\sigma_{\lambda}]\in H_k(\G_{m,n},\Z/2\Z)\,|\,|\lambda|=mn-k\big\}\big)
\]
for every $k\in\{0,\dots,mn\}$, where $\lambda$ is a partition of the $(m\times n)$-rectangle, $\sigma_{\lambda}$ is the Schubert variety of $\G_{m,n}$ defined by the incidence conditions, prescribed by $\lambda$, with respect to the canonical complete flag.

In particular, $\big\{[\sigma_{\lambda}]\in H_k(\G_{m,n},\Z/2\Z)\,|\,|\lambda|=mn-k\big\}$ as above is a basis of $H_k(\G_{m,n},\Z/2\Z)$ for every $k\in\{1,\dots,mn\}$.
\end{corollary}

\begin{proof}
By Lemma \ref{lem:schubert} the family of $\Omega_{\lambda}$ such that $\lambda$ is a partition of the $(m\times n)$-rectangle constitutes the cells of a finite CW-complex whose underlying topological space is $\G_{m,n}$ such that $\sigma_{\lambda}=\overline{\Omega}_{\lambda}$ is algebraic for every partition $\lambda$ of the $(m\times n)$-rectangle. Hence, the thesis follows by Lemma \ref{lem:CW}.
\end{proof}

\begin{theorem}\label{thm:Q-algebraic-homology}
Each $\G_{m,n}\subset\R^{(m+n)^2}$ is a projectively $\Q$-closed $\Q$-nonsingular $\Q$-algebraic set having projectively $\Q$-algebraic homology.
\end{theorem}

\begin{proof}
By Corollary \ref{cor:homology}, for every $k\in\{0,\dots,mn\}$:
\[
H_k(\G_{m,n},\Z/2\Z)=\textnormal{Span}\big(\big\{[\sigma_{\lambda}]\in H_k(\G_{m,n},\Z/2\Z)\,|\,|\lambda|=mn-k\big\}\big),
\]
where each $\sigma_{\lambda}$ is a Schubert variety of $\G_{m,n}$ defined by the incidence conditions, prescribed by $\lambda$, with respect to the canonical complete flag of $\R^{m+n}$. By Theorem \ref{thm:Q-desingularization}, each Schubert variety $\sigma_\lambda$ as above admits a $\Q$-desingularization, that is: there exists a projectively $\Q$-closed $\Q$-nonsingular $\Q$-algebraic set $Z_\lambda\subset\R^{(m+n)^2}\times \R^p$ of dimension $\dim(\sigma_{\lambda})$, for some $p\in\N$, such that $\pi_1:Z_\lambda\to\sigma_{\lambda}$ is a birational map. Observe that, since $\pi_1:Z_\lambda\to\sigma_{\lambda}$ is surjective, injective onto the Zariski open subset $\Omega_\lambda$ such that $\overline{\Omega}_{\lambda}=\sigma_{\lambda}$ and $\dim(Z_\lambda)=\dim(\sigma_{\lambda})$, we get that $\pi_{1*}([Z_\lambda])=[\sigma_{\lambda}]$ by \cite[Lemma\,5.3]{AK85} or \cite[Lemma\,1.1]{BT80}, as desired.
\end{proof}

\vspace{1.5em}

%%%

\section{Relative $\Q$-algebraic constructions}\label{sec:3}

\vspace{0.5em}

\subsection{$\Q$-Algebraic bordism classes and unoriented relative bordisms}\label{subsec:3.1}

In \cite[Definition\,2.39]{GSa}, Ghiloni and the author introduced a variant `over $\Q$' of the classical algebraic unoriented bordism and investigated its relation with projectively $\Q$-algebraic homology. Here we briefly recall those notions and useful results. Let $W\subset\R^k$ be a real algebraic set. Given a compact $\mscr{C}^\infty$ manifold $P$ and a $\mscr{C}^\infty$ map $f:P\to W$, we say that the unoriented bordism class of $f$ is \emph{projectively $\Q$-algebraic} if there exist a compact $\mscr{C}^\infty$ manifold $T$ with boundary, a projectively $\Q$-closed $\Q$-nonsingular $\Q$-algebraic set $Y\subset\R^h$, a $\mscr{C}^\infty$ diffeomorphism $\psi:P\sqcup Y\to\partial T$ and a $\mscr{C}^\infty$ map $F:T\rightarrow W$ such that $F\circ\jmath\circ(\psi|_P)=f$ and $F\circ \jmath\circ(\psi|_Y)$ is a $\Q$-regular map, where $\jmath:\partial T\hookrightarrow T$ is the inclusion map.

\vspace{0.5em}

\begin{definition}\label{def:Q-bordism}
We say that $W$ has \emph{projectively $\Q$-algebraic unoriented bordism} if for all $p\in\{0,\ldots,d\}$, for all compact $\mscr{C}^\infty$ manifold $P$ of dimension $p$ and for all $\mscr{C}^\infty$ map $f:P\to W$, the unoriented bordism class of $f$ is projectively $\Q$-algebraic.
\end{definition}

\vspace{0.5em}

Here we recall the fundamental result of \cite{GSa} about the equivalence between Definitions \ref{def:Q-homology} \& \ref{def:Q-bordism}. That is:

\vspace{0.5em}

\begin{lemma}[{\cite[Lemma\,2.41]{GSa}}\label{thm:Q_homology}]
Let $W\subset\R^k$ be a $\Q$-nonsingular $\Q$-algebraic set. The following assertions are equivalent.
\begin{enumerate}[label=\emph{(\roman*)},ref=(\roman*)]
\item\label{en:Q_homology1} $W$ has projectively $\Q$-algebraic unoriented bordism.
\item\label{en:Q_homology2} $W$ has projectively $\Q$-algebraic homology.
\end{enumerate}
\end{lemma}

\vspace{0.5em}

Let us specify `over $\Q$' the construction of the algebraic unoriented relative cobordism by Akbulut and King in \cite[Lemma\,4.1]{AK81b}.

\vspace{0.5em}

\begin{lemma}\label{lem:Q_tico_bordism}
Let $M$ be a compact $\mscr{C}^\infty$ submanifold of $\R^{n}$ of dimension $d$ and let $M_i$, for $i=1,\dots,\ell$, be closed $\mscr{C}^\infty$ submanifolds of $M$ of codimension $c_i$ in general position. Then there are a compact $\mscr{C}^\infty$ manifold with boundary $T$ and proper $\mscr{C}^\infty$ submanifolds with boundary $T_i$, for $i=1,\dots,\ell$, in general position, a projectively $\Q$-closed $\Q$-nonsingular $\Q$-algebraic subset $Y$ of $\R^h$ for some $h\in\N$, and a $\mscr{C}^\infty$ diffeomorphism $\psi:M\sqcup Y\to \partial T$ such that:
\begin{enumerate}[label=\emph{(\roman*)}, ref=(\roman*)]
\item\label{Q_tico_bordism_a} $Y$ is the disjoint union of projectively $\Q$-closed $\Q$-nonsingular $\Q$-algebraic sets $Y^\alpha\subset\R^{h}$ for every $\alpha\subset\{1,\dots,\ell\}$ such that $\bigcap_{i\in\alpha}M_i\neq\varnothing$.
\item\label{Q_tico_bordism_b} $\partial T\cap T_i=\partial T_i$, $\psi(M)\cap T_i=\psi(M_i)$ and $\psi(Y^\alpha)\cap T_i=\psi(Y_{i}^\alpha)$ where $Y_i^{\alpha}$, for $i=1,\dots,\ell$, are projectively $\Q$-closed $\Q$-nonsingular $\Q$-algebraic subsets of $Y^\alpha$ in general position with $Y_i^\alpha=\varnothing$ whenever $i\notin\alpha$.
\item\label{Q_tico_bordism_d} For every $\alpha\subset\{1,\dots,\ell\}$ and $i\in\alpha$, there is a $\Q$-regular function $\mu_i ^\alpha:Y_i^\alpha\rightarrow \G_{c_i,n-c_i}$ such that
\[
Y^\alpha=(\mu_i^{\alpha})^*(\E^*_{c_i,n-c_i}).
\]
In particular, $\mu_i^{\alpha}$ is the Gauss mapping of $Y_i^\alpha$ in $Y^\alpha$.
\end{enumerate}
\end{lemma}

\begin{proof}
For every $\alpha\subset\{1,\dots,\ell\}$ we denote by $M_\alpha:=\bigcap_{i\in\alpha} M_i$, if $\alpha\neq\varnothing$, and $M_{\varnothing}:=M$. We argue by induction on the subsets $\alpha$ of $\{1,\dots,\ell\}$ so that $M_\alpha\neq\varnothing$. The case in which all $M_\alpha=\varnothing$, for every $\alpha\subset\{1,\dots,\ell\}$, means that $M=M_\varnothing=\varnothing$, thus the theorem follows by taking $T=\varnothing$. Suppose the set of $\alpha\subset\{1,\dots,\ell\}$ so that $M_\alpha\neq \varnothing$ is non-empty. Let $\alpha$ be such that $M_\alpha\neq\varnothing$ and $M_{\alpha'}=\varnothing$ for every $\alpha'\subset\{1,\dots,\ell\}$ so that $\alpha\varsubsetneq\alpha'$. Let $\beta_i:M_i\rightarrow\G_{c_i,n-c_i}$ be the Gauss mapping of $M_i$ in $M$ for every $i\in\alpha$. Let $\G_\alpha:=\prod_{i\in\alpha} \G_{c_i,n-c_i}$. By Theorem \ref{thm:Q-algebraic-homology} and by the K\"unneth formula, $\G_\alpha\subset \R^{n^2 |\alpha|}$ is a projectively $\Q$-closed $\Q$-nonsingular $\Q$-algebraic set having projectively $\Q$-algebraic homology. Let $\beta_\alpha:M_\alpha\rightarrow \G_\alpha$ be the $\mscr{C}^\infty$ function defined as $\beta_\alpha:=\prod_{i\in\alpha}\beta_i$. Thus, Theorem \ref{thm:Q_homology} ensures the existence of $k_\alpha\in\N$, a compact $\mscr{C}^\infty$ manifold with boundary $T_\alpha$, a projectively $\Q$-closed $\Q$-nonsingular $\Q$-algebraic set $Y_\alpha\subset\R^{k_\alpha}$, a $\mscr{C}^\infty$ diffeomorphism $\psi_\alpha:M_\alpha\sqcup Y_\alpha\to \partial T_\alpha$ and a $\mscr{C}^\infty$ map $\mu^\alpha:T_\alpha\rightarrow\G_\alpha$ such that $\mu^\alpha\circ\jmath_\alpha\circ(\psi_\alpha|_{M_\alpha})=\beta_\alpha$ and $g_\alpha:=\mu^\alpha\circ \jmath_\alpha\circ(\psi_\alpha|_Y)\in\reg^\Q(Y,\G_{\alpha})$, that is, $g_\alpha$ is $\Q$-regular, where $\jmath_\alpha:\partial T_\alpha\hookrightarrow T_\alpha$ denotes the inclusion map.

Let $\E_\alpha^*:=\prod_{i\in\alpha}\E_{c_i,n-c_i}^*$. Define the pullback bundle of $\E_\alpha^*$ via $\mu^\alpha$ as $S^\alpha:=(\mu^\alpha)^{\ast}(\E_\alpha^*)$ and the $\mscr{C}^\infty$ submanifolds $S_i^\alpha$ of $S^\alpha$ as follows
\begin{align*}
S^\alpha&:=\{(x,y_1,t_1,\dots,y_{|\alpha|},t_{|\alpha|}) \in T_\alpha\times(\R^{n}\times\R)^{|\alpha|}\,|\,(\mu^\alpha(x),y_1,t_1,\dots,y_{|\alpha|},t_{|\alpha|})\in \E_\alpha^*\}\\
S_i^\alpha&:=\{(x,y_1,t_1,\dots,y_{|\alpha|},t_{|\alpha|})\in S^\alpha\,|\,y_i=0,\,t_i=0\},
\end{align*}
for every $i\in\alpha$. By definition, the $S_i^\alpha$, for $i\in\alpha$, are in general position and $\bigcap_{i\in\alpha}S_i^\alpha=T_\alpha\times\{0\}\subset T_\alpha\times(\R^{n}\times\R)^{|\alpha|}$. In addition, considering the projections $\pi_{0}^{i}:S_i^\alpha\rightarrow T_\alpha$ and $\pi_i:\G_\alpha\rightarrow \G_{c_i,n-c_i}$, we define $\mu_i^\alpha:S_i^\alpha\rightarrow \G_{c_i,n-c_i}$ as $\mu_i^\alpha=\pi_i\circ\mu^\alpha\circ\pi_{0}^{i}$. Thus, we deduce that $S^\alpha$ is the pullback sphere bundle of $\E_{c_i,n-c_i}^*$ by $\mu_i^\alpha$, i.e. $S^\alpha=(\mu_i^\alpha)^*(\E_{c_i,n-c_i}^*)$, where 
\begin{align*}
(\mu_i^\alpha)^*(\E_{c_i,n-c_i}^*):=\{(x,y_1,t_1,\dots,y_\ell,t_\ell,y_{\ell+1},t_{\ell+1})&\in S_i^\alpha\times \R^{n}\times\R\,|\\ \,(\mu_i^\alpha(x),y_{|\alpha|+1},t_{|\alpha|+1})&\in\E_{c_i,n-c_i}^*\}.
\end{align*}
Thus, $S^\alpha$ and the $S_i^\alpha$, for every $i\in\alpha$, are $\mscr{C}^\infty$ manifolds with boundary satisfying $\partial S_i^\alpha\subset\partial S^\alpha$.  Define:
\begin{align*}
M^\alpha&:=\beta_\alpha^*(\E^*_\alpha)=(\mu^\alpha\circ\jmath_\alpha\circ\psi_\alpha)|_{M_\alpha}^*(\E^*_\alpha)\subset M_\alpha\times\R^{(n+1) |\alpha|},\\
Y^\alpha&:=g_\alpha^*(\E^*_\alpha)=(\mu^\alpha\circ\jmath_\alpha\circ\psi_\alpha)|_{Y_\alpha}^*(\E^*_\alpha)\subset\R^{k_\alpha}\times\R^{(n+1) |\alpha|}.
\end{align*} Observe that, by Lemma \ref{lem:Q_sphere_pullback}, we deduce that $Y^\alpha\subset\R^{k_\alpha}\times\R^{(n+1) |\alpha|}$ is a projectively $\Q$-closed $\Q$-nonsingular $\Q$-algebraic set. Since $\psi_\alpha:M_\alpha\sqcup Y_\alpha\to\partial T_\alpha$ is a diffeomorphism, we deduce that $\Psi_\alpha:M^\alpha\sqcup Y^\alpha\to\partial S^\alpha$ defined as $\Psi_\alpha(x,y_1,t_1,\dots,y_{|\alpha|},t_{|\alpha|})=(\psi_\alpha(x),y_1,t_1,\dots,y_{|\alpha|},t_{|\alpha|})$ is a diffeomorphism as well. Hence, define
\[
Y^\alpha_i:=Y^\alpha\cap\Psi_\alpha^{-1}(\partial S^\alpha_i)
\]
for every $i\in\alpha$. Observe that $Y^\alpha_i=((\mu^\alpha_{\alpha\setminus \{i\}}\circ \Psi_\alpha)|_{Y_\alpha})^*(\E^*_{\alpha\setminus \{i\}})$, where $\mu^\alpha_{\alpha\setminus \{i\}}:T_\alpha\to\G_{c_1,n-c_1}\times\dots\times\G_{c_{i-1},n-c_{i-1}}\times\{0\}\times\G_{c_{i+1},n-c_{i+1}}\times\dots\times\G_{c_{|\alpha|},n-c_{|\alpha|}}$ defined as
\[
\mu^\alpha_{\alpha\setminus\{i\}}(x):=(\mu^\alpha_1(x),\dots,\mu^\alpha_{i-1}(x),0,\mu^\alpha_{i+1}(x),\dots,\mu^\alpha_{|\alpha|}(x))
\]
and
\begin{equation}\label{eq:subbundle}
\E^*_{\alpha\setminus \{i\}}:=\{(y_1,t_1,\dots,y_{|\alpha|},t_{|\alpha|})\in\E^*_\alpha\,|\,y_i=0,\,t_i=0\},
\end{equation}
which is a projectively $\Q$-closed $\Q$-nonsingular $\Q$-algebraic sphere bundle by Lemma \ref{lem:Q_sphere_bundle}. Observe that $(\mu^\alpha_{\alpha\setminus \{i\}}\circ \Psi_\alpha)|_{Y_\alpha}$ is $\Q$-regular since $(\mu^\alpha\circ\psi_\alpha)|_{Y_\alpha}$ is so. Thus, $Y^\alpha_i\subset\R^{k_\alpha}\times\R^{(n+1)|\alpha|}$ is a projectively $\Q$-closed $\Q$-nonsingular $\Q$-algebraic set by Lemma \ref{lem:Q_sphere_pullback}, for every $i\in\alpha$.

Since $\mu^\alpha|_{M_\alpha}$ is the Gauss mapping of $M_\alpha$ in each $M_i$ with $i\in \alpha$, we can select two sufficiently small closed tubular neighborhoods $U_\alpha$ and $V_\alpha$ of $M_\alpha$ in $M^\alpha$ and in $M$, respectively, which are diffeomorphic via a diffeomorphism $h_\alpha:U_\alpha\rightarrow V_\alpha$ satisfying $h_\alpha(U_\alpha\cap S_i^\alpha)=V_\alpha\cap M_i$, for every $i\in\alpha$. Consider the $\mscr{C}^\infty$ manifold with boundary $S$ defined as $S^\alpha\cup (M\times[0,1])$ identifying $U_\alpha$ and $V_\alpha\times\{1\}$ via $h_\alpha\times\{1\}:U_\alpha\to V_\alpha\times\{1\}$ defined as $(h_\alpha\times\{1\})(a)=(h_\alpha(a),1)$, after smoothing corners. In the same way define the $\mscr{C}^\infty$ submanifolds with boundary $S_i$ as $S_i^\alpha\cup (M_i\times[0,1])$ identifying $U_\alpha\cap S_i^\alpha$ with $(V_\alpha\cap M_i)\times\{1\}$ via $h_\alpha\times\{1\}$. Observe that the $\mscr{C}^\infty$ submanifolds $S_i$ of $S$, with $i\in\alpha$, are in general position.

\begin{figure}[h!]
\centering
\tikzset{every picture/.style={line width=0.75pt}} %set default line width to 0.75pt        

\begin{tikzpicture}[x=0.65pt,y=0.65pt,yscale=-1,xscale=1]
%uncomment if require: \path (0,300); %set diagram left start at 0, and has height of 300

%Curve Lines [id:da30516897927777575] 
\draw [line width=1.5]    (101,177) .. controls (66,176) and (68,70) .. (100,59) ;
%Curve Lines [id:da057643225497074346] 
\draw [line width=1.5]    (101,177) .. controls (133,175) and (137,58) .. (100,59) ;
%Curve Lines [id:da5250930468136582] 
\draw [line width=1.5]    (101,177) .. controls (159,168) and (216,198) .. (251,186) ;
%Curve Lines [id:da7386075516200611] 
\draw [line width=1.5]    (100,59) .. controls (149,49) and (214,79) .. (250,68) ;
%Curve Lines [id:da9948892079598215] 
\draw [line width=1.5]    (127,127) .. controls (172,116) and (242,148) .. (277,136) ;
%Curve Lines [id:da6655025019289876] 
\draw [line width=1.5]    (394,166) .. controls (393,176) and (393,185) .. (381,187) ;
%Curve Lines [id:da6567130603761856] 
\draw [line width=1.5]  [dash pattern={on 1.69pt off 2.76pt}]  (250,68) .. controls (218,78) and (216,182) .. (251,186) ;
%Curve Lines [id:da1313302747067734] 
\draw [line width=1.5]    (571,138) .. controls (559,155) and (568,184) .. (550,187) ;
%Curve Lines [id:da3884972168403562] 
\draw [line width=1.5]    (549,69) .. controls (578,62) and (583,120) .. (571,138) ;
%Curve Lines [id:da8768337910610604] 
\draw [line width=1.5]    (380,69) .. controls (407,67) and (409,100) .. (408,114) ;
%Curve Lines [id:da43872538557918483] 
\draw [line width=1.5]  [dash pattern={on 1.69pt off 2.76pt}]  (549,69) .. controls (516,84) and (518,186) .. (550,187) ;
%Straight Lines [id:da05676174269871581] 
\draw [line width=1.5]    (402,138) -- (571,138) ;
%Straight Lines [id:da31784356630260924] 
\draw [line width=1.5]    (381,187) -- (550,187) ;
%Straight Lines [id:da5019617979429811] 
\draw [line width=1.5]    (380,69) -- (549,69) ;
%Curve Lines [id:da03929187842812498] 
\draw [line width=1.5]    (300,124) .. controls (324,124) and (352,123) .. (378,125) ;
%Curve Lines [id:da6645596846851405] 
\draw [line width=1.5]    (361,151) .. controls (380,152) and (394,154) .. (394,166) ;
%Curve Lines [id:da9577762299997271] 
\draw [line width=1.5]    (378,125) .. controls (389,125) and (405,128) .. (408,114) ;
%Curve Lines [id:da9186330113822176] 
\draw [line width=1.5]    (271,163) .. controls (274,148) and (292,150) .. (303,150) ;
%Curve Lines [id:da6405676739808747] 
\draw [line width=1.5]    (276,111) .. controls (277,123) and (281,125) .. (300,124) ;
%Curve Lines [id:da9682552652323206] 
\draw [line width=1.5]    (303,150) .. controls (313,150) and (337,151) .. (361,151) ;
%Curve Lines [id:da7417442219610677] 
\draw [line width=1.5]    (272,137) .. controls (285,133) and (362,138) .. (402,138) ;
%Curve Lines [id:da42261876474590176] 
\draw [line width=1.5]    (355,124) .. controls (357,93) and (367,73) .. (380,69) ;
%Curve Lines [id:da9572244138034933] 
\draw [line width=1.5]  [dash pattern={on 1.69pt off 2.76pt}]  (356,150) .. controls (354,134) and (354,136) .. (355,124) ;
%Curve Lines [id:da7761962078873376] 
\draw [line width=1.5]    (381,187) .. controls (367,185) and (362,176) .. (356,150) ;
%Curve Lines [id:da8277581574789521] 
\draw    (250,68) .. controls (296,56) and (362,67) .. (380,69) ;
%Curve Lines [id:da08025674175468356] 
\draw    (251,186) .. controls (284,178) and (293,218) .. (291,242) ;
%Curve Lines [id:da15031742293863615] 
\draw    (342,242) .. controls (342,218) and (341,187) .. (381,187) ;
%Curve Lines [id:da9331610788556416] 
\draw    (291,242) .. controls (293,224) and (340,224) .. (342,242) ;
%Curve Lines [id:da5804975189531417] 
\draw    (291,242) .. controls (291,262) and (341,262) .. (342,242) ;
%Curve Lines [id:da4115835585759827] 
\draw    (394,166) .. controls (397,140) and (408,134) .. (408,114) ;
%Curve Lines [id:da5051058395336077] 
\draw [line width=0.75]    (276,111) .. controls (279,128) and (275,147) .. (271,163) ;
%Curve Lines [id:da42031179939657726] 
\draw [line width=1.5]    (250,68) .. controls (273,71) and (274,103) .. (276,111) ;
%Curve Lines [id:da9951963277033522] 
\draw [line width=1.5]    (251,186) .. controls (263,183) and (269,171) .. (271,163) ;

% Text Node
\draw (574,159.4) node [anchor=north west][inner sep=0.75pt]    {$M$};
% Text Node
\draw (582,119.4) node [anchor=north west][inner sep=0.75pt]    {$M_{\alpha }$};
% Text Node
\draw (433,193.4) node [anchor=north west][inner sep=0.75pt]    {$M\times [ 0,1]$};
% Text Node
\draw (433,142.4) node [anchor=north west][inner sep=0.75pt]    {$M_{\alpha } \times [ 0,1]$};
% Text Node
\draw (103,106.4) node [anchor=north west][inner sep=0.75pt]    {$Y_{\alpha }$};
% Text Node
\draw (95,145.4) node [anchor=north west][inner sep=0.75pt]    {$Y^{\alpha }$};
% Text Node
\draw (247,115.4) node [anchor=north west][inner sep=0.75pt]    {$M_{\alpha }$};
% Text Node
\draw (242,156.4) node [anchor=north west][inner sep=0.75pt]    {$M^{\alpha }$};
% Text Node
\draw (175,131.4) node [anchor=north west][inner sep=0.75pt]    {$T_{\alpha }$};
% Text Node
\draw (177,189.4) node [anchor=north west][inner sep=0.75pt]    {$S^{\alpha }$};
% Text Node
\draw (318,103.4) node [anchor=north west][inner sep=0.75pt]    {$N$};
% Text Node
\draw (348,240.4) node [anchor=north west][inner sep=0.75pt]    {$Y'$};
% Text Node
\draw (296,32.4) node [anchor=north west][inner sep=0.75pt]    {$S\cup T$};
% Text Node
\draw (309,199.4) node [anchor=north west][inner sep=0.75pt]    {$T$};

\end{tikzpicture}
\caption{Inductive step constructing a relative bordism.}\label{fig:relative-bordism}
\end{figure}
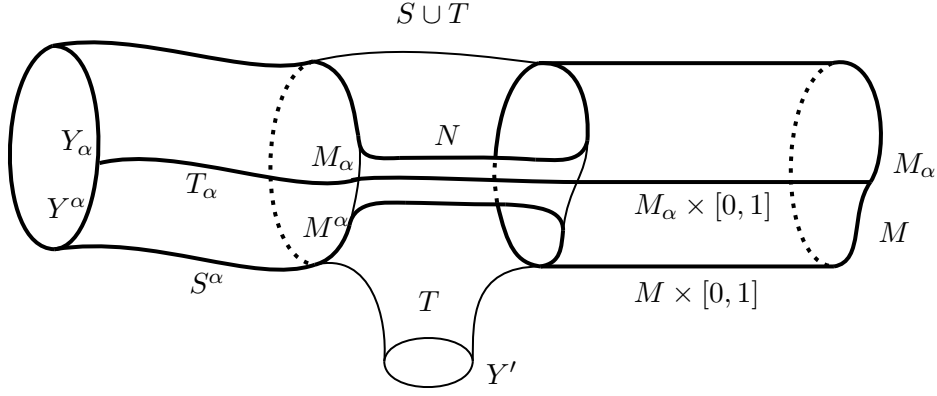

Define the $\mscr{C}^\infty$ manifold $N$ with $\mscr{C}^\infty$ submanifolds  in general position  $N_i$, for every $i\in\{1,\dots,\ell\}$, as follows:
\begin{align*}
N&:=(M^\alpha\setminus\Int(U_\alpha))\cup_{h_\alpha}(M\setminus\Int(V_\alpha)),\\
N_i&:=
\begin{cases}
N\cap S_i\quad&\text{if $i\in\alpha$,}\\
M_i\times\{1\}&\text{otherwise.}
\end{cases}
\end{align*}
Observe that, by construction, $\partial(S^\alpha\cup_{h_\alpha} (M\times[0,1]))=N\sqcup Y^\alpha\sqcup M$, with $M$ identified with $M\times\{0\}$, $\partial(S^\alpha_i\cup_{h_\alpha} (M_i\times[0,1]))=N_i\sqcup Y^\alpha_i\sqcup M_i$ for every $i\in\alpha$, and $\partial (M_i\times[0,1])=N_i\sqcup M_i$ for every $i\notin\alpha$. In particular, it holds that $N_\alpha:=\bigcap_{i\in\alpha}N_i=\varnothing$. By Whitney $\mscr{C}^\infty$ embedding theorem, there is a $\mscr{C}^\infty$ manifold $M'\subset\R^{2d+1}$ with $\mscr{C}^\infty$ submanifolds $M'_i$ of codimension $c_i$ in general position for $i\in\{1,\dots, \ell\}$, which is diffeomorphic to $N$ via a diffeomorphism $\varphi:M'\to N$ such that $\varphi(M'_i)=N_i$ for every $i\in\{1,\dots,\ell\}$. Thus, by inductive assumption on $M'\subset\R^{2d+1}$, there exist $k'\in\N$, a $\mscr{C}^\infty$ manifold with boundary $T'$ and $\mscr{C}^\infty$ submanifolds with boundary $T'_i$ for every $i\in\{1,\dots,\ell\}$, with transverse intersection, a projectively $\Q$-closed $\Q$-nonsingular $\Q$-algebraic subset $Y'$ of $\R^{k'}$ for some $k'\in\N$, a $\mscr{C}^\infty$ diffeomorphism $\psi':M'\sqcup Y'\to \partial T'$ (without lost of generality we can assume $\psi'(M')=N$ and $\psi'(M'_i)=N_i$) such that:
\begin{enumerate}[label={(\roman*$'$)}, ref=(\roman*$'$)]
\item $Y'$ is the disjoint union of a projectively $\Q$-closed $\Q$-nonsingular $\Q$-algebraic  sets $Y^{'\alpha}\subset\R^{n}\times\R^{k'}$, for every $\alpha\subset\{1,\dots,\ell\}$ such that $\bigcap_{i\in\alpha}M'_i\neq\varnothing$.
\item $\partial T'\cap T'_i=\partial T'_i$, $N\cap T'_i=\psi'(M')\cap T'_i=\psi(M'_i)=N_i$ and $\psi'(Y^{'\alpha})\cap T'_i=\psi'(Y_{i}^{'\alpha})$ where $Y_i^{'\alpha}$, for $i\in\{1,\dots,\ell\}$, are projectively $\Q$-closed $\Q$-nonsingular $\Q$-algebraic subsets of $Y^{'\alpha}$ transverse to each other with $Y_i^{'\alpha}=\varnothing$ whenever $i\notin\alpha$.
\item For every $\alpha\subset\{1,\dots,\ell\}$ and $i\in\alpha$, there is a $\Q$-regular function $\mu_i ^{'\alpha}:Y_i^{'\alpha}\rightarrow \G_{c_i,2d+1-c_i}$ such that
\[
Y^{'\alpha}=(\mu_i^{'\alpha})^*(\E^*_{c_i,2d+1-c_i}).
\]
In particular, $\mu_i^{'\alpha}$ is the Gauss mapping of $Y^i_\alpha$ in $Y_i$.
\end{enumerate}

Define $T:=S\cup T'$ and $T_i:=S_i\cup T'_i$, after smoothing corners. Let $k:=\max\{k_\alpha,k'\}$ and consider $\iota_\alpha:\R^{k_\alpha}\to\R^k$ and $\iota':\R^{k'}\to\R^k$ be the inclusion mappings. Then, after a translation of a rational factor $v\in\Q^{k}$ if necessary, we may assume that $(\iota'(Y')+v)\cap \iota_\alpha(Y^\alpha)=\varnothing$, thus $Y:=\iota_\alpha (Y^\alpha)\sqcup (\iota'(Y')+v)\subset\R^{k}$ is a projectively $\Q$-closed $\Q$-nonsingular $\Q$-algebraic set by \cite[Proposition\,2.16(i)\,\&\,Lemma\,2.24]{GSa}. Let $\psi:M\sqcup Y\to \partial T$ defined as follows $\psi|_M:=\psi_\alpha|_M$, $\psi|_{\iota_\alpha(Y_{\alpha})}(x):=\psi_\alpha(\iota_\alpha^{-1}(x))$ and $\psi|_{\iota'(Y')+v}(x):=\psi'((\iota')^{-1}(x-v))$.
\end{proof}

Here we provide an embedded version of Lemma \ref{lem:Q_tico_bordism} and we `double the relative bordism over $\Q$' following the strategy used by Tognoli in \cite[\S\,{\it b}),\,pp.\,176-177]{Tog73}.

\vspace{0.5em}

\begin{theorem} \label{thm:Q-spine-cobordism}
Let $M$ be a compact $\mscr{C}^\infty$ submanifold of $\R^{n}$ of dimension $d$, let $M_i$ for $i=1,\dots,\ell$, be $\mscr{C}^\infty$ submanifolds of $M$ of codimension $c_i$ in general position. Then there exist $s\in\N$ with $s\geq n$, a projectively $\Q$-closed $\Q$-nonsingular $\Q$-algebraic set $Y\subset\R^s=\R^{n}\times\R^{s-n}$ of dimension $d$, $\Q$-nonsingular $\Q$-algebraic subsets $Y_i$, for $i\in\{1,\dots,\ell\}$, of $Y$ in general position, a compact $\mscr{C}^\infty$ submanifold $S$ of $\R^{s+1}=\R^s\times\R$ of dimension $d+1$ and compact $\mscr{C}^\infty$ submanifolds $S_i$ of $S$ of codimension $c_i$, for $i=1,\dots,\ell$, in general position with the following properties:
\begin{enumerate}[label=\emph{(\roman*)}, ref=(\roman*)]
\item\label{thm:Q-spine-cobordism-1} $M\cap Y=\varnothing$.
\item\label{thm:Q-spine-cobordism-2} $S\cap(\R^s\times(-1,1))=(M\sqcup Y)\times(-1,1)$ and $S_i\cap(\R^s\times(-1,1))=(M_i\sqcup Y_i)\times(-1,1)$, for every $i\in\{1,\dots,\ell\}$.
\item\label{thm:Q-spine-cobordism-3} $Y$ is the finite disjoint union $\bigsqcup_{\alpha\in A}(Y^\alpha+v_\alpha)$ of projectively $\Q$-closed $\Q$-nonsingular $\Q$-algebraic sets of the form $Y^\alpha+v_\alpha\subset\R^s$, where $v_\alpha$ belongs to $\Q^s$, $Y^\alpha$ is inductively defined as in the proof of \emph{Lemma \ref{lem:Q_tico_bordism}} and
\[
A:=\Big\{\alpha\subset\{1,\dots,\ell\}\,|\,\bigcap_{j\in\alpha} M_j\neq\varnothing\Big\}.
\]
In addition, there are projectively $\Q$-closed $\Q$-nonsingular $\Q$-alge\-braic subset $Y_\alpha\subset\R^s$ and $\Q$-regular functions $\mu_\alpha:Y_\alpha\to \G^*_\alpha$ such that $Y^\alpha:=\mu_\alpha^*(\E^*_{\alpha})$, with $\G_{\alpha}^*:=\prod_{i\in\alpha}\G_{c_i,n-c_i}^*$ and $\E_{\alpha}^*:=\prod_{i\in\alpha}\E_{c_i,n-c_i}^*$.
\item\label{thm:Q-spine-cobordism-4} Let $i\in\{1,\dots,\ell\}$. Then, $Y_i$ is the finite disjoint union $\bigsqcup_{\alpha\in A_i}(Y^\alpha_i+v_\alpha)$ of projectively $\Q$-closed $\Q$-nonsingular $\Q$-algebraic sets of the form $Y^\alpha_i+v_\alpha\subset\R^s$, where $v_\alpha$ belongs to $\Q^s$ as above,  $Y^\alpha_i$ is inductively defined as in the proof of \emph{Lemma \ref{lem:Q_tico_bordism}} and
\[
A_i:=\{\alpha\in A\,|\, i\in\alpha\}.
\]
In addition, there is a $\Q$-regular map $\mu_i^\alpha:Y^\alpha_i\to \G_{c_i,n-c_i}$ such that, if $\beta_i:S_i\to\G_{c_i,n-c_i}$ denotes the Gauss mapping of $S_i$ in $S$, then $\beta_i|_{Y_i}=\bigsqcup_{\alpha\in A_i} \mu_i^{\alpha}$ is a $\Q$-regular map.
%In addition, if $Y_\alpha$ is as in point \emph{\ref{thm:Q-spine-cobordism-3}}, there is a $\Q$-regular function $\mu_i^\alpha:Y_\alpha\to \G_{c_i,n-c_i}$ such that $Y^\alpha_i:=(\mu_i^\alpha)^*(\E^*_{\alpha\setminus\{i\}})$, where $\E^*_{\alpha\setminus\{i\}}$ is defined as in \emph{(\ref{eq:subbundle})}. In particular, if $\beta_i:S_i\to\G_{c_i,n-c_i}$ denotes the Gauss mapping of $S_i$ in $S$, then $\beta_i|_{Y_i}=\bigsqcup_{\alpha\in A_i} \mu_i^{\alpha}$ is a $\Q$-regular map.
%\textcolor{red}{In addition, if $Y_\alpha$ is as in point \emph{\ref{en:Q-bordism-improved-d}}, there is a $\Q$-regular function $\mu_i^\alpha:Y_\alpha\to \G_{c_i,n-c_i}$ such that $Y^\alpha_i:=(\mu_i^\alpha)^*(\E^*_{c_i,n-c_i})$, where $\E^*_{c_i,n-c_i}$ is defined as in \emph{(\ref{eq:subbundle})}. In particular, if $\beta_i:S_i\to\G_{c_i,m+n}$ denotes the Gauss mapping of $S_i$ in $S$, then $\beta_i|_{Y_i}=\bigsqcup_{\alpha\in A_i} \mu_i^{\alpha}$ is a $\Q$-regular map.}
\end{enumerate}
\end{theorem}

\begin{proof}
Thanks to the proof of Lemma \ref{lem:Q_tico_bordism}, for $s\geq n$ sufficiently large, we know that there exist a projectively $\Q$-closed $\Q$-nonsingular $\Q$-algebraic set $Y=\bigsqcup_{\alpha\in A}(Y^\alpha+v_\alpha)\subset \R^s$ and $\Q$-nonsingular $\Q$-algebraic subsets $Y_i=\bigsqcup_{\alpha\in A_i}(Y^\alpha_i+v_\alpha)\subset\R^s$, with $i\in\{1,\dots,\ell\}$, in general position with above properties \ref{thm:Q-spine-cobordism-1}\,(changing the vectors $v_\alpha\in\Q^s$ if necessary)\&\,\ref{thm:Q-spine-cobordism-3}, compact $\mscr{C}^\infty$ manifolds $T$ and $T_i$ with boundary $\partial T$ and $\partial T_i$ so that $T_i\subset T$ and $\partial T_i\subset\partial T$, for every $i=1,\dots,\ell$, and a $\mscr{C}^\infty$ diffeomorphism $\psi:M\sqcup Y\to \partial T$ satisfying Lemma \ref{lem:Q_tico_bordism}\ref{Q_tico_bordism_a}-\ref{Q_tico_bordism_d}.

Let us construct the desired compact $\mscr{C}^\infty$ submanifold $S$ of $\R^{s+1}=\R^s\times\R$, following the strategy used by Tognoli in \cite[\S\,{\it b}),\,pp.\,176-177]{Tog73}. By the collaring theorem (see \cite[{Theorem 6.1, p.\,113}]{Hir94}), there exist an open neighborhood $U$ of $\partial T$ in $T$ and a $\mscr{C}^\infty$ diffeomorphism $\phi':U\to \partial T\times[0,1)$ such that $\phi'(t)=(t,0)$ for all $t\in\partial T$ and $\phi'|_{T_i\cap U}:T_i\cap U\to \partial T_i\times[0,1)$ is a diffeomorphism as well, for every $i=1,\dots,\ell$. Let $\phi:U\to(M\sqcup Y)\times[0,1)$ be the $\mscr{C}^\infty$ diffeomorphism $\phi:=(\psi^{-1}\times\mr{id}_{[0,1)})\circ\phi'$. Note that %$\phi(\partial T)=(M\sqcup Y)\times\{0\}$ and
$\phi(t)=(\psi^{-1}(t),0)$ for all $t\in\partial T$. Set $A:=T\setminus\partial T$, $B:=\phi^{-1}((M\sqcup Y)\times(0,\frac{1}{2}])\subset A$, $N:=\R^s\times(0,+\infty)$ and define the map $\theta:B\to N$ by $\theta(x,x_{s+1}):=\phi(x,x_{s+1})$. %Note that $f$ is of class $\mscr{C}^\infty$, i.e., it admits a $\mscr{C}^\infty$ extension from an open neighborhood of $B$ in $A$ to $N$.
Since we can safely assume $s+1\geq2(d+1)+1$, Tietze's theorem ensures the existence of a continuous extension of $\theta$ from the whole $A$ to $N$, we can apply to $\theta$ the extension theorem \cite[Theorem 5$(\mr{f})$]{Whi36} of Whitney, obtaining a $\mscr{C}^\infty$ embedding $\Theta:A\to N$ extending $\theta$. Let $R:\R^{s+1}=\R^s\times\R\to\R^{s+1}$ be the reflection $R(x,x_{s+1}):=(x,-x_{s+1})$ and let $S'$ and $S'_i$ be the compact $\mscr{C}^\infty$ submanifolds $\Theta(A)\sqcup((M\sqcup Y)\times\{0\})\sqcup R(\Theta(A))$ and $\Theta(A\cap T_i)\sqcup((M_i\sqcup Y_i)\times\{0\})\sqcup R(\Theta(A\cap T_i))$ of $\R^{s+1}$, for every $i=1,\dots,\ell$, respectively. Thanks to the compactness of $T$ and of each $T_i$, there exists $\epsilon>0$ such that $S'\cap(\R^s\times(-\epsilon,\epsilon))=(M\sqcup Y)\times(-\epsilon,\epsilon)$ and $S'_i\cap(\R^s\times(-\epsilon,\epsilon))=(M_i\sqcup Y_i)\times(-\epsilon,\epsilon)$. Let $L:\R^{s+1}\to\R^{s+1}$ be the linear isomorphism $L(x,x_{s+1}):=(x,\epsilon^{-1}x_{s+1})$. The compact $\mscr{C}^\infty$ submanifold $S:=L(S')$ with $\mscr{C}^\infty$ submanifolds $S_i:=L(S'_i)$, for every $i=1,\dots,\ell$, in general position of $\R^{s+1}$ have the desired properties \ref{thm:Q-spine-cobordism-2}\,\&\,\ref{thm:Q-spine-cobordism-4}.
\end{proof}

\subsection{A review on $\Q$-stable $\Q$-algebraic sets}\label{subsec:3.2}

Here we briefly recall the notion of $\Q$-stable $\Q$-algebraic set and we develop an explicit example of $\Q$-stable $\Q$-algebraic set useful for applying approximation techniques `over $\Q$' in Section \ref{sec:4}.

\vspace{0.5em}

\begin{definition}[{\cite[Definition\,3.1]{GSa}}] \label{def:Q-pair}
Let $L\subset\R^n$ be a $\Q$-algebraic set. We say that the pair $L$ is a \emph{$\Q$-stable} if for each $a\in L$, there exists an open neighborhood $U_a$ of $a$ in $\R^n$ such that
\begin{equation}\label{eq:Q-approx-pair}
\Ii^\infty_{\R^n}(L)\mscr{C}^\infty(U_a)=\Ii_\Q(L)\mscr{C}^\infty(U_a),
\end{equation}
i.e., for each $f\in\Ii^\infty_{\R^n}(L)$, we have $f|_{U_a}=\sum_{i=1}^\ell u_i\cdot p_i|_{U_a}$ for some $u_1,\ldots,u_\ell\in\mscr{C}^\infty(U_a)$ and $p_1,\dots,p_\ell$ generators of $\Ii_\Q(L)$. %\bs
\end{definition}

\vspace{0.5em}

%The reader observes that in condition \eqref{eq:Q-approx-pair} the converse inclusion always holds and it remains valid if we replace $U_a$ with a smaller open neighborhood of $a$ in $\R^{n}$. In addition, i
Evidently, by definition, the disjoint union of finitely many $\Q$-stable algebraic subsets of $\R^{n}$ is again a $\Q$-stable algebraic subset of $\R^{n}$. Moreover, by \cite[Lemma\,3.3(iii)(iv)]{GSa}, every disjoint union of finitely many $\Q$-nonsingular $\Q$-algebraic sets is $\Q$-stable and, if $L\subset\R^{n}$ is a $\Q$-stable $\Q$-algebraic set then $L\times\{0\}\subset\R^{n}\times\R^k$ is $\Q$-stable for every $k\in\N$. Next Lemma \ref{lem:Q-nice-hyp} will be very useful in the setting of Theorem \ref{thm:Q_tico_tognoli}.

\vspace{0.5em}

\begin{lemma}\label{lem:Q-nice-hyp}
Let $M\subset \R^n$ be a compact $\mscr{C}^\infty$ manifold of dimension $d$.  Let $X\subset M$ be a $\Q$-nonsingular $\Q$-algebraic subset of $\R^n$ of codimension $c$ and let $Y\subset M$ be a $\Q$-nonsingular $\Q$-algebraic hypersurface of $M$. If the germ $(M,X\cup Y)$ of $M$ at $X\cup Y$ is the germ of a $\Q$-nonsingular $\Q$-algebraic set, then $X\cup Y$ is $\Q$-stable.
\end{lemma}

\begin{proof}
Without lost of generality we may assume that none of the irreducible components of $X$ is contained in $Y$. Let $a\in (X\cup Y)\setminus (X\cap Y)=(X\setminus Y)\sqcup(Y\setminus X)$. Since both $X\subset\R^n$ and $Y\subset\R^n$ are $\Q$-nonsingular $\Q$-algebraic sets, up to shrink the neighborhood $U_a$, we deduce property (\ref{eq:Q-approx-pair}) by \cite[Appendix\,C]{GSa}.
Let $a\in X\cap Y$ and let $f\in\Ii^\infty_{\R^{n}}(X\cup Y)$. Let $V\subset\R^n$ be a $\Q$-nonsingular $\Q$-algebraic set such that the germ $(M,X\cup Y)$ of $M$ at $X\cup Y$ coincides to the germ $(V,X\cup Y)$ of $V$ at $X\cup Y$. Since $Y$ is a $\Q$-nonsingular $\Q$-algebraic hypersurface of $V\subset\R^n$ there are $p_1,\dots,p_{n-d}\in\Ii_\Q(V)$ and $p\in\Ii_\Q(Y)$ whose gradients at $a$ are linearly independent over $\R$ and there is a neighborhood $U_a$ of $a$ in $\R^{n}$ such that $Y\cap U_a=V\cap \mathcal{Z}_{\R^{n}}(p)\cap U_a=\mathcal{Z}_{\R^{n}}(p,p_1,\dots,p_{n-d})\cap U_a$. Hence, by \cite[Lemma 2.5.4]{AK92a}, there are $u,u_1,\dots,u_{n-d}\in\mscr{C}^\infty(U_a)$ such that $f|_{U_a}=u\cdot p|_{U_a}+\sum_{i=1}^{n-d} u_i \cdot p_i|_{U_a}$, up to shrink the neighborhood $U_a$ of $a$ in $\R^{n}$ if necessary. Since none of the irreducible components of $X$ is contained in $Y$, we deduce that $Y\cap U_a=V\cap\mathcal{Z}_{\R^{n}}(p)\cap U_a\varsubsetneq (X\cup Y)\cap U_a$. Thus $\mathcal{Z}_{\R^{n}}(p)\cap U_a\cap X\subset Y$, up to shrink the neighborhood $U_a$ of $a$ in $\R^{n}$ if necessary. In addition, since $f|_{U_a}=u\cdot p|_{U_a}+\sum_{i=1}^{n-d} u_i \cdot p_i|_{U_a}$, $p_1,\dots,p_{n-d}\in\Ii_\Q(V)$ and $\mathcal{Z}_{\R^{n}}(p)\cap U_a\cap X\subset Y$, we deduce that $X\cap U_a\subset\mathcal{Z}_{\R^{n}}(u)$. Now, let $U'_a\subset U_a$ be a neighborhood of $a$ in $\R^{n}$ such that $\overline{U'_a}\varsubsetneq U_a$. An explicit construction via partitions of unity subordinated to the open cover $\{U_a, X\setminus \overline{U'_a}\}$ of $\R^{n}$ ensures the existence of $g\in\mscr{C}^\infty(\R^{n})$ such that $g|_{U'_a}=u|_{U'_a}$ and $g\in\Ii^{\infty}_{\R^{n}}(X)$. Since $X\subset\R^n$ is a $\Q$-nonsingular $\Q$-algebraic set of codimension $c$ in $V\subset\R^n$, which is a $\Q$-nonsingular $\Q$-algebraic set as well, there are $q_1,\dots,q_c\in\Ii_\Q(X)$ such that $\nabla p_1(a),\dots,\nabla p_{n-d}(a),\nabla q_1(a),\dots,\nabla q_c(a)$ are linearly independent over $\R$ and there exists a neighborhood $V_a$ of $a$ in $\R^n$ such that $X\cap V_a=\mathcal{Z}_{\R^{n}}(p_1,\dots,p_{n-d},q_1,\dots,q_c)\cap V_a$. Thus, by \cite[Lemma 2.5.4]{AK92a}, there are $u'_1,\dots,u'_{n-d},v_1,\dots,v_c\in\mscr{C}^\infty(V_a)$ such that $g|_{V_a}=\sum_{i=1}^c v_i \cdot q_i|_{V_a}+\sum_{i=1}^{n-d} u'_i \cdot p_i|_{V_a}$, up to shrink the neighborhood $V_a$ of $a$ in $\R^{n}$ if necessary. Thus, fixing $V'_a:=U'_a\cap V_a$, we have:
\begin{align*}
f|_{V'_a}&=g|_{V'_a}\cdot p|_{V'_a}+\sum_{i=1}^{n-d} u_i|_{V'_a}\cdot p_i|_{V'_a}\\
&=\left(\sum_{i=1}^c v_i|_{V'_a}\cdot q_i|_{V'_a}+\sum_{i=1}^{n-d} u'_i|_{V'_a}\cdot  p_i|_{V'_a}\right) \cdot p|_{V'_a} +\sum_{i=1}^{n-d} u_i|_{V'_a}\cdot p_i|_{V'_a}\\
&=\sum_{i=1}^c v_i|_{V'_a}\cdot (p\cdot q_i )|_{V'_a} + \sum_{i=1}^{n-d} \left(u_i|_{V'_a}+u'_i|_{V'_a}\cdot p|_{V'_a}\right)\cdot  p_i|_{V'_a},
\end{align*}
where $p_1,\dots,p_{n-d},p\cdot q_1,\dots,p\cdot q_c\in\Ii_\Q(X\cup Y)$, as desired.
\end{proof}

\vspace{1.5em}

\section{Nash-Tognoli theorem over $\Q$ \& the Relative $\Q$-algebraicity problem}\label{sec:4}

Here we are in position to prove a version `over $\Q$' of the relative Nash-Tognoli theorem in \cite{AK81b}. With respect to \cite[Theorem\,3.9\,\&\,Corollary\,3.12]{GSa} our construction via relative cobordisms allows us to deal with submanifolds of arbitrary codimensions. %In particular, in case the starting compact manifold $M$ and each submanifold $M_i$ with $i\in\{1,\dots,\ell\}$ are actually Nash manifolds we get a Nash diffeomorphism in the statement. Hence, when the starting compact manifold $M$ and each submanifold $M_i$ with $i\in\{1,\dots,\ell\}$ are actually compact nonsingular algebraic sets, our Theorem \ref{thm:Q_tico_tognoli} provides a positive answer of the \textsc{Relative $\Q$-algebrization problem for nonsingular algebraic sets} in the compact case with Nash regularity.

Let us briefly present part of the construction of \cite[Theorem\,3.9]{GSa} that we will use in Theorem \ref{thm:Q_tico_tognoli}. We just present a sketch of the proof, a detailed one can be found in the proof of mentioned \cite[Theorem\,3.9]{GSa}.

\vspace{0.5em}

\begin{lemma}\label{lem:Q-Tognoli}
Let $M$ be a compact $\mscr{C}^\infty$ submanifold of $\R^{n}$ of dimension $d$. Let $s\in\N$ with $s\geq n$. Let $Y\subset\R^s=\R^{n}\times\R^{s-n}$ be a projectively $\Q$-closed $\Q$-nonsingular $\Q$-algebraic set  of dimension $d$ and $S\subset\R^{s+1}$ be a compact $\mscr{C}^\infty$ manifold  of dimension $d+1$ such that $M\cap Y=\varnothing$ and $S\cap(\R^s\times(-1,1))=(M\sqcup Y)\times(-1,1)$. Let $W$ be a $\Q$-nonsingular $\Q$-algebraic set and $\varphi:S\to W$ be a $\mscr{C}^\infty$ map such that $\varphi|_{Y\times\{0\}}$ is a $\Q$-regular map. Then, there exist $k\in\N$, a $\Q$-nonsingular $\Q$-algebraic subset $Z$ of $\R^{s+1+k}=\R^{s+1} \times \R^k$, a projectively $\Q$-closed $\Q$-nonsingular $\Q$-algebraic set $M'\subset\R^{s+1+k}$, a $\mscr{C}^\infty$ diffeomorphism $\phi:M\to M'$ and a $\Q$-regular map $\eta:Z\to W$ with the following properties:
\begin{enumerate}[label=\emph{(\roman*)}, ref=(\roman*)]
%\item\label{thm:Q_tognoli-1} $L\times\{0\}\subset V$.
\item\label{lem:Q_tognoli-1} $\phi\circ\jmath_{M'}$ is arbitrarily $\mscr{C}^\infty_\w$-close to the inclusion map $\jmath_{M}:M\hookrightarrow\R^{s+1+k}$, where $\jmath_{M'}:M'\hookrightarrow\R^{s+1+k}$ denotes the inclusion map.
\item\label{lem:Q_tognoli-2} $\eta|_{M'}\circ\phi$ is arbitrarily $\mscr{C}^\infty_\w$-close to $\varphi$.
\end{enumerate}
\end{lemma}

\begin{proof}
Let $G:=\G_{s-d,d+1}$ and let $\beta:S \to G$ be the normal bundle map of $S$ in $\R^{s+1}$. Observe that $\beta|_{Y\times\{0\}}$ is a $\Q$-regular map, indeed $\beta$ coincides with the normal bundle map of $Y\times\R$ in $\R^{s+1}$ in a neighborhood of $Y\times\{0\}$ in $\R^{s+1}$.

Let $E:=\E_{s-d,d+1}=\{(A,b) \in G \times \R^{s+1} \, | \, Ab=b\}$ be the universal vector bundle over the Grassmannian manifold $G$. Let $\beta^{\ast}(E):=\{(x,y) \in S \times \R^{s+1} \, | \,\beta(x) y=y\}$ be the pullback bundle and let $\theta:\beta^{\ast}(E) \to \R^{s+1}$ defined by $\theta(x,y):=x+y$. By the Implicit Function Theorem, there exists an open neighborhood $U_0$ in $\beta^{\ast}(E)$ of the zero section $S \times 0$ of $\beta^{\ast}(E)$ and an open neighborhood $U$ of $S$ in $\R^{s+1}$ such that $\theta|_{U_0}:U_0 \to U$ is a diffeomorphism. Define a $\mscr{C}^\infty$ map $\widetilde{\beta}:U \to E$ and a smooth map $\varrho:U \to S$ as follows: for every $x \in U$, let $(z_x,y_x):=(\theta|_{U_0})^{-1}(x)$ and let $N_x:=\beta(z_x)$, then define $\widetilde{\beta}(x):=(N_x,y_x)$ and ${\varrho}(x):=z_x$. Hence, ${\widetilde{\beta}}^{-1}(G\times \{0\})=S$, $\widetilde{\beta}|_{Y}$ is $\Q$-regular and $\widetilde{\beta}$ is transverse to $G \times \{0\}$ in $E$.

Let $\widetilde{\varphi}: U \to W$ be the smooth map defined by $\widetilde{\varphi}:=\varphi\circ {\varrho}$. An application of \cite[Lemma\,3.7]{GSa} with the following substitutions: %``$V$''$:=\R^{n+1}$,
``$W$''$:=E\times W$, ``$L$''$:=Y$, ``$f$''$:=\widetilde{\beta} \times \widetilde{\varphi}$ and ``$U$'' equal to some open neighborhood $U'$ of $S$ in $\R^{s+1}$ relatively compact in $U$ gives a $\Q$-nonsingular $\Q$-algebraic subset $Z$ of $\R^{s+1} \times \R^k$, for some integer $k$, an open subset $Z_0$ of $Z$ and a $\Q$-regular map $\eta:Z \to E\times W$ satisfying:
\begin{itemize}
\item[$(\mr{iii})$]  $Y\times\{0\}\times\{0\}\subset Z_0$, $\pi(Z_0)=U'$, the restriction $\pi|_{Z_0}:Z_0 \to U'$ is a $\mscr{C}^\infty$ diffeomorphism, and the $\mscr{C}^\infty$ map $\sigma:U'\to\R^{s+1+k}$, defined by $\sigma(x,x_{s+1}):=(\pi|_{Z_0})^{-1}(x,x_{s+1})$ for all $(x,x_{s+1})\in U'$, is arbitrarily $\mscr{C}^\infty_\w$ close to $\iota$.
\item[$(\mr{iv})$]  $\eta(x,x_{s+1},0)=(\widetilde{\beta} \times \widetilde{\varphi})(x,x_{s+1})$ for all $(x,x_{s+1})\in Y\times\{0\}$.
\item[$(\mr{v})$]  The $\mscr{C}^\infty$ map $\widehat{\eta}:U'\to E\times W$, defined by $\widehat{\eta}(x,x_{s+1}):=\eta(\sigma(x,x_{s+1}))$, is arbitrarily $\mscr{C}^\infty_\w$ close to $(\widetilde{\beta} \times \widetilde{\varphi})|_{U'}$,
\end{itemize}
where $\pi:\R^{s+1} \times \R^k \to \R^{s+1}$ denotes the natural projection and $\iota:U'\hookrightarrow \R^{s+1} \times \R^k$ denotes the inclusion map.

Choose an open neighborhood $U''$ of $S$ in $\R^{s+1}$ such that $\overline{U''}\subset U'$. Set $Z_1:=(\pi|_{Z_0})^{-1}(U'')$. By $(\mr{iii})$, $(\mr{iv})$, $(\mr{v})$ and \cite[Theorem 14.1.1]{BCR98}, we deduce that $S':=\widehat{\eta}^{\,-1}((G\times\{0\}) \times W)$ is a compact $\mscr{C}^\infty$ submanifold of $U''$ containing $Y\times\{0\}\times\{0\}$ and there exists a $\mscr{C}^\infty$ diffeomorphism $\psi_1:U''\to U''$ arbitrarily $\mscr{C}^\infty_\w$ close to $\mr{id}_{U''}$ such that $\psi_1(S)=S'$ and $\psi=\mr{id}_{U''}$ on $Y\times\{0\}\times\{0\}$. Moreover, \cite[Lemma\,2.19]{GSa} ensures that $S'':=\eta^{-1}((G\times\{0\}) \times W)\subset\R^{s+1+k}$ is a $\Q$-nonsingular $\Q$-algebraic set of dimension $d+1$. % such that $S''_1:=S''\cap Z_1=(\pi|_{Z_1})^{-1}(S')\subset\Reg^*(S'')$.
In addition, the $\mscr{C}^\infty$ embedding $\psi_2:S\to\R^{s+1+k}$ defined by $\psi_2(x,x_{s+1}):=(\pi|_{Z_1})^{-1}(\psi_1(x,x_{s+1}))$, is arbitrarily $\mscr{C}^\infty_\w$ close to the inclusion map $j_S:S\hookrightarrow\R^{s+1+k}$, $\psi_2=j_S$ on $Y\times\{0\}$ and $\psi_2(S)=S''_1$. Note that $S''_1$ is the union of some connected components of $S''$. Let $S''_2:=S''\setminus S''_1$. The coordinate hyperplane $\{x_{s+1}=0\}$ of $\R^{s+1+k}$ is transverse to $S''_1$ in $\R^{s+1+k}$, thus $S''_1\cap\{x_{s+1}=0\}=M' \sqcup (Y\times\{0\}\times\{0\})$ for some compact $\mscr{C}^\infty$ submanifold $M'$ of $\R^{s+1+k}$ and there exists a $\mscr{C}^\infty$ embedding $\psi_3:M\to\R^{s+1+k}$ arbitrarily $\mscr{C}^\infty_\w$ close to the inclusion map $j_M:M\hookrightarrow\R^{s+1+k}$ such that $M'=\psi_3(M)$.

Let $K$ be a compact neighborhood of $S''_1$ in $\R^{s+1+k}$ such that $K\cap S''_2=\varnothing$ and let $\pi_{s+1}:\R^{s+1+k}=\R^{s}\times\R\times\R^{k}\to\R$ be the projection $\pi_{s+1}(x,x_{s+1},y):=x_{s+1}$. Let $q\in\Q[x_1,\dots,x_{s+1+k}]$ be an overt polynomial such that $\mathcal{Z}_\R(q)=Y\times\{0\}\times\{0\}$, $q\geq0$ on $\R^{s+1+k}$ and $q\geq 2$ on $\R^{s+1+k}\setminus K$. Let $K'$ be a compact neighborhood of $S''_1$ in $\mr{int}_{\R^{s+1+k}}(K)$. Using a $\mscr{C}^\infty$ partition of unity subordinated to $\{\mr{int}_{\R^{s+1+k}}(K),\R^{s+1+k} \setminus K'\}$, we can define a $\mscr{C}^\infty$ function $h:\R^{s+1+k}\to\R$ such that $h=\pi_{s+1}$ on $K'$ and $h=q$ on $\R^{s+1+k}\setminus K$. An application of \cite[Lemma\,3.6]{GSa} to $h-q$ gives a $\Q$-regular function $u':\R^{s+1+k}\to\R$ with the following properties:
\begin{itemize}
\item[$(\mr{vi})$] There exist $e\in\N$ and a polynomial $p\in\Q[x_1,\ldots,x_{s+1+k}]$ of degree $\leq 2e$ such that $u'(x)=p(x)(1+|x|_{s+t}^2)^{-e}$ for all $x\in\R^{s+1+k}$. 
\item[$(\mr{vii})$] $Y\times\{0\}\times\{0\}\subset\mathcal{Z}_\R(u')$.
\item[$(\mr{viii})$] $\sup_{x\in\R^{s+1+k}}|h(x)-q(x)-u'(x)|<1$.
\item[$(\mr{ix})$] $u'$ is arbitrarily $\mscr{C}^\infty_\w$ close to $\pi_{s+1}-q$ on $\mr{int}_{\R^{s+1+k}}(K')$.
\end{itemize}

Let $u:\R^{s+1+k}\to\R$ be the $\Q$-regular map given by $u:=u'+q$, and let $v\in\Q[x]$, with $x=(x_1,\ldots,x_{s+1+k})$, be the polynomial $v(x):=q(x)(1+|x|_{s+1+k}^2)^e+p(x)$. Observe that $u(x)=(1+|x|_{s+1+k}^2)^{-e}v(x)$ and $v\in\Q[x]$ is overt. In addition, by $(\mr{ix})$, $(\mr{x})$ \& $(\mr{xi})$, we have that $Y\times\{0\}\times\{0\}\subset\mathcal{Z}_\R(u)$, $u>1$ on $\R^{s+1+k}\setminus K$ and $u$ is arbitrarily $\mscr{C}^\infty_\w$ close to $\pi_{s+1}$ on $\mr{int}_{\R^{s+1+k}}(K')$. Hence, $S''_1\cap \mathcal{Z}_\R(u)=M'' \sqcup X$ for some compact $\mscr{C}^\infty$ submanifold $M''$ of $\R^{s+1+k}$ and there exists a $\mscr{C}^\infty$ embedding $\psi_4:M'\to\R^{s+1+k}$ arbitrarily $\mscr{C}^\infty_\w$ close to the inclusion map $j_{M'}:M'\hookrightarrow\R^{s+1+k}$ such that $M''=\psi_4(M')$. An application of \cite[Lemmas\,2.19\,\&\,2.24(ii)]{GSa} ensure that $M''\sqcup (Y\times\{0\}\times\{0\})\subset\R^{s+1+k}$ is a is projectively $\Q$-closed $\Q$-nonsingular $\Q$-algebraic set. As a consequence, Proposition \ref{cor:Q_setminus} gives that $M''\subset\R^{s+1+k}$ is a projectively $\Q$-closed $\Q$-nonsingular $\Q$-algebraic set. Consider the embedding $\psi:M\to \R^{s+1+k}$ defined as $\psi:=\psi_4\circ\psi_3$. Then, $\psi$ is arbitrarily $\mscr{C}^\infty_\w$ close to $j_M$ and $\psi(M)=M''$. To conclude the proof it suffices to fix ``$M'$''$:=M''$ and ``$\phi$'' as the function $\phi:M\to M''$ defined by $\phi(x)=\psi(x)$.
\end{proof}

Now we are in position to state and prove our relative algebraic approximation theorem `over $\Q$' in the compact case.

\vspace{0.5em}

\begin{theorem}[Relative Nash-Tognoli theorem `over $\Q$']\label{thm:Q_tico_tognoli}
Let $M$ be a compact $\mscr{C}^\infty$ submanifold of $\R^{n}$ of dimension $d$ and let $M_i$ for $i=1,\dots,\ell$, be $\mscr{C}^\infty$ submanifolds of $M$ in general position. Set $m:=\max\{n,2 d+1\}$. Then, for every neighborhood $\mathcal{U}$ of the inclusion map $\iota:M\hookrightarrow\R^{m}$ in $\mscr{C}^\infty_\w(M,\R^{m})$ and for every neighborhood $\mathcal{U}_i$ of the inclusion map $\iota|_{M_i}:M_i\hookrightarrow\R^m$ in $\mscr{C}^\infty_\w(M_i,\R^m)$, for every $i\in\{1,\dots,\ell\}$, there exist a projectively $\Q$-closed $\Q$-nonsingular $\Q$-algebraic set $M'\subset\R^m$, $\Q$-nonsingular $\Q$-algebraic subsets $M'_i$ of $M'$ for $i=1,\dots,\ell$, in general position and a $\mscr{C}^\infty$ diffeomorphism $h:M\to M'$ which simultaneously takes each $M_i$ to $M'_i$ such that, if $\jmath:M'\hookrightarrow\R^m$ denotes the inclusion map, then $\jmath\circ h\in\mathcal{U}$ and $\jmath\circ h|_{M_i}\in\mathcal{U}_i$, for every $i\in\{1,\dots,\ell\}$.

If in addition $M$ and each $M_i$ are compact Nash manifolds, then we can assume $h:M\rightarrow M'$ to be a Nash diffeomorphism extending to a semialgebraic homeomorphism from $\R^m$ to $\R^m$.
\end{theorem}

\begin{proof}
Let $c_i$ be the codimension of $M_i$ in $M$ for $i\in\{1,\dots,\ell\}$. An application of Theorem \ref{thm:Q-spine-cobordism} gives $s\in\N$, a projectively $\Q$-closed $\Q$-nonsingular $\Q$-algebraic set $Y\subset\R^{s}:=\R^{n}\times\R^{s-n}$ of dimension $d$, $\Q$-nonsingular $\Q$-algebraic subsets $Y_i$, for $i\in\{1,\dots,\ell\}$, of $Y$ in general position, a compact $\mscr{C}^\infty$ submanifold $S$ of $\R^{s+1}=\R^{s}\times\R$ of dimension $m+1$ and compact $\mscr{C}^\infty$ submanifolds $S_i$ of $S$ of codimension $c_i$, for $i\in\{1,\dots,\ell\}$, in general position satisfying Theorem \ref{thm:Q-spine-cobordism}\ref{thm:Q-spine-cobordism-1}-\ref{thm:Q-spine-cobordism-4}.

\begin{figure}[h!]
\centering

\tikzset{every picture/.style={line width=0.75pt}} %set default line width to 0.75pt        

\begin{tikzpicture}[x=0.75pt,y=0.75pt,yscale=-1,xscale=1]
%uncomment if require: \path (0,390); %set diagram left start at 0, and has height of 390

%Shape: Ellipse [id:dp8106675034997846] 
\draw  [color={rgb, 255:red, 208; green, 2; blue, 27 }  ,draw opacity=1 ][line width=0.75]  (310,169) .. controls (310,162.92) and (320.3,158) .. (333,158) .. controls (345.7,158) and (356,162.92) .. (356,169) .. controls (356,175.08) and (345.7,180) .. (333,180) .. controls (320.3,180) and (310,175.08) .. (310,169) -- cycle ;
%Shape: Ellipse [id:dp5089481920638588] 
\draw  [color={rgb, 255:red, 208; green, 2; blue, 27 }  ,draw opacity=1 ][line width=0.75]  (449,170) .. controls (449,164.48) and (458.63,160) .. (470.5,160) .. controls (482.37,160) and (492,164.48) .. (492,170) .. controls (492,175.52) and (482.37,180) .. (470.5,180) .. controls (458.63,180) and (449,175.52) .. (449,170) -- cycle ;
%Curve Lines [id:da7491794051962015] 
\draw [line width=0.75]    (213,170) .. controls (212,110) and (309,110) .. (310,169) ;
%Curve Lines [id:da4324190909221365] 
\draw [color={rgb, 255:red, 0; green, 0; blue, 0 }  ,draw opacity=1 ][line width=0.75]    (195,177) .. controls (194,77) and (336,87) .. (325,178) ;
%Straight Lines [id:da7610653475666803] 
\draw    (93,239) -- (493,240) ;
%Straight Lines [id:da8333095787036428] 
\draw    (493,240) -- (549,149) ;
%Straight Lines [id:da5328885917818029] 
\draw    (93,239) -- (149,148) ;
%Curve Lines [id:da32540428609902294] 
\draw [line width=0.75]  [dash pattern={on 4.5pt off 4.5pt}]  (213,170) .. controls (210,231) and (312,237) .. (310,169) ;
%Curve Lines [id:da019317282548600212] 
\draw [color={rgb, 255:red, 0; green, 0; blue, 0 }  ,draw opacity=1 ][line width=0.75]  [dash pattern={on 4.5pt off 4.5pt}]  (195,177) .. controls (200,260) and (328,252) .. (325,178) ;
%Curve Lines [id:da9812552787673235] 
\draw [color={rgb, 255:red, 0; green, 0; blue, 0 }  ,draw opacity=1 ]   (166,170) .. controls (166,-22) and (488,-16.5) .. (492,170) ;
%Curve Lines [id:da3953902743322475] 
\draw [color={rgb, 255:red, 0; green, 0; blue, 0 }  ,draw opacity=1 ]   (269,66) .. controls (274,73) and (286,73) .. (292,66) ;
%Curve Lines [id:da542496338968459] 
\draw [color={rgb, 255:red, 0; green, 0; blue, 0 }  ,draw opacity=1 ]   (274,69) .. controls (277,64) and (285,66) .. (287,69) ;
%Curve Lines [id:da06548350500733602] 
\draw [color={rgb, 255:red, 0; green, 0; blue, 0 }  ,draw opacity=1 ]   (370,67) .. controls (374,72) and (382,72) .. (390,66) ;
%Curve Lines [id:da7045156596816785] 
\draw [color={rgb, 255:red, 0; green, 0; blue, 0 }  ,draw opacity=1 ]   (375,70) .. controls (377,66) and (383,67) .. (385,69) ;
%Curve Lines [id:da7661267417377097] 
\draw [color={rgb, 255:red, 0; green, 0; blue, 0 }  ,draw opacity=1 ]   (166,170) .. controls (169,144) and (215,153) .. (213,170) ;
%Curve Lines [id:da2881154651852632] 
\draw [color={rgb, 255:red, 0; green, 0; blue, 0 }  ,draw opacity=1 ]   (166,170) .. controls (168.5,189) and (182,170) .. (195,177) ;
%Curve Lines [id:da46439090514165116] 
\draw [color={rgb, 255:red, 0; green, 0; blue, 0 }  ,draw opacity=1 ][line width=0.75]    (195,177) .. controls (203,181) and (212.5,182) .. (213,170) ;
%Curve Lines [id:da008228408188737713] 
\draw [color={rgb, 255:red, 0; green, 0; blue, 0 }  ,draw opacity=1 ]   (275,276.75) .. controls (278,272) and (285,272.17) .. (289,274) ;
%Curve Lines [id:da38092854194997483] 
\draw [color={rgb, 255:red, 0; green, 0; blue, 0 }  ,draw opacity=1 ]   (370,270) .. controls (377,279.1) and (384,276.82) .. (390,270) ;
%Curve Lines [id:da003088853534576552] 
\draw [color={rgb, 255:red, 0; green, 0; blue, 0 }  ,draw opacity=1 ]   (271,273.38) .. controls (276,282.13) and (288,280.67) .. (293,269) ;
%Curve Lines [id:da9477768952078941] 
\draw [color={rgb, 255:red, 0; green, 0; blue, 0 }  ,draw opacity=1 ]   (374,274.55) .. controls (377,271) and (382,272) .. (385,273.41) ;
%Curve Lines [id:da6398105316697813] 
\draw    (181,154) .. controls (184,102) and (267,1) .. (332,65) ;
%Curve Lines [id:da5803652051439607] 
\draw    (470.5,160) .. controls (467,102) and (394,137) .. (332,65) ;
%Curve Lines [id:da17206080706068605] 
\draw    (356,169) .. controls (357,127) and (452,126) .. (449,170) ;
%Curve Lines [id:da06888773985212848] 
\draw  [dash pattern={on 4.5pt off 4.5pt}]  (356,169) .. controls (357,220) and (454,211) .. (449,170) ;
%Curve Lines [id:da9250367530756871] 
\draw  [dash pattern={on 4.5pt off 4.5pt}]  (379,239) .. controls (425,221) and (475,221) .. (470.5,160) ;
%Curve Lines [id:da7219428320616984] 
\draw [color={rgb, 255:red, 0; green, 0; blue, 0 }  ,draw opacity=1 ] [dash pattern={on 4.5pt off 4.5pt}]  (181,154) .. controls (175,191) and (188,223) .. (199,241) ;
%Curve Lines [id:da976895759421796] 
\draw    (199,241) .. controls (264,338) and (326,303) .. (337,279) ;
%Curve Lines [id:da9691342160520644] 
\draw    (337,279) .. controls (345,259) and (358,248) .. (379,239) ;
%Curve Lines [id:da2970316931245669] 
\draw  [dash pattern={on 4.5pt off 4.5pt}]  (166,170) .. controls (166,202) and (171,221) .. (180,239) ;
%Curve Lines [id:da7761908966776035] 
\draw  [dash pattern={on 4.5pt off 4.5pt}]  (492,170) .. controls (493,198) and (485,220) .. (476,239) ;
%Curve Lines [id:da18837280129001266] 
\draw    (180,239) .. controls (240,352) and (421,348) .. (476,239) ;

% Text Node
\draw (141,158) node [anchor=north west][inner sep=0.75pt]    {$M$};
% Text Node
\draw (197,182) node [anchor=north west][inner sep=0.75pt]    {$M_{i}$};
% Text Node
\draw (449,41.4) node [anchor=north west][inner sep=0.75pt]    {$S\subset \mathbb{R}^{s+1}$};
% Text Node
\draw (227,89.4) node [anchor=north west][inner sep=0.75pt]    {$S_{i}$};
% Text Node
\draw (390,89.4) node [anchor=north west][inner sep=0.75pt]    {$S_{j}$};
% Text Node
\draw (542,162) node [anchor=north west][inner sep=0.75pt]    {$Y\subset \mathbb{R}^{s}$};
% Text Node
\draw (327,182) node [anchor=north west][inner sep=0.75pt]    {$Y_{\alpha }$};
% Text Node
\draw (361,158) node [anchor=north west][inner sep=0.75pt]    {$Y_{i}^{\alpha }$};
% Text Node
\draw (472,182) node [anchor=north west][inner sep=0.75pt]    {$Y_{\beta }$};
% Text Node
\draw (498,158) node [anchor=north west][inner sep=0.75pt]    {$Y_{j}^{\beta }$};
\
\end{tikzpicture}
\caption{Starting situation after the application of Theorem \ref{thm:Q-spine-cobordism}.}\label{fig:relative-tognoli-1}
\end{figure}
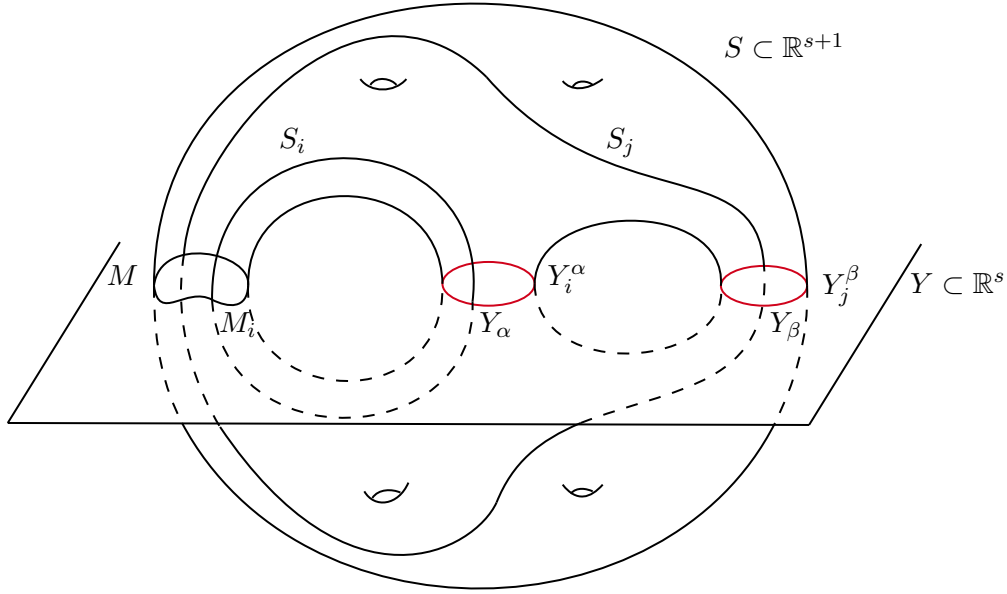

Consider the map $\beta_i:S_i\to\G_{c_i,s+1-c_i}$ classifying the normal bundle of $S_i$ in $S$. By Theorem \ref{thm:Q-spine-cobordism}\ref{thm:Q-spine-cobordism-4} we have that $\beta_i|_{Y_i}$ is a $\Q$-regular map extending the codomain from $\G_{c_i,n-c_i}$ to $\G_{c_i,s+1-c_i}$. An application of \cite[Theorem\,3.9]{GSa}, with $``L":=Y_i\times\{0\}$, $``M":=S_i$, $``W":=\G_{c_i,s+1-c_i}$ and $``f":=\beta_i$ gives $t\in\N$, a projectively $\Q$-closed $\Q$-nonsingular $\Q$-algebraic subset $X_i $ of $\R^{s+1+t}$, a diffeomorphism $\rho_i:S_i\to X_i$ and a $\Q$-regular map $\gamma_i:X_i\rightarrow \G_{c_i,s+1-c_i}$ satisfying \cite[Theorem\,3.9(i)-(iii)]{GSa}. In particular, $(Y_i\times\{0\})\times\{0\}\subset X_i$, $\rho_i(x)=(x,0)$ and $\gamma_i(x,0)=\beta_i|_{Y_i\times\{0\}}(x)$ for every $x\in Y_i\times\{0\}$. 

Consider the pullback bundle $Z_i:=\gamma_i^*(\E_{c_i,s+1-c_i}^*)$. By Lemma \ref{lem:Q_sphere_pullback}, $Z_i$ is a projectively $\Q$-closed $\Q$-nonsingular $\Q$-algebraic subset of $\R^{s+1+t}\times\R^{s+1+t+1}$ and it contains those subsets $Y^\alpha$ of $Y$ such that $i\in\alpha$, by  Theorem \ref{thm:Q-spine-cobordism}\ref{thm:Q-spine-cobordism-3}\ref{thm:Q-spine-cobordism-4} and Lemma \ref{lem:Q_tico_bordism}\ref{Q_tico_bordism_d}. More precisely, following the notations of Theorem \ref{thm:Q-spine-cobordism}, we have that
\[
Y'^{\alpha}:=(\gamma_i|_{(Y_i^{\alpha}+v_\alpha)\times\{0\}\times\{0\}})^*(\E_{c_i,s+1-c_i}^*)
\]
is contained in $Z_i$, for every $\alpha\in A_i$, and is $\Q$-biregularly isomorphic to $Y^{\alpha}$ fixing each $x\in Y_\alpha$. % via the map $Y^\alpha\to Y'^{\alpha}$ as $(x,y)\mapsto(x,0,0,0,y)\in\R^{s}\times\R\times\R^t\times\R\times\R^{s+1+t+1}$.
Let $Y'_i\subset\R^s\times\R\times\R^t\times\R^{s+1+t+1}$ be defined as
\[
Y'_i:=\Big(\bigsqcup_{\alpha\in A_i} Y'^{\alpha}\Big)\sqcup \Big(\bigsqcup_{\alpha\notin A_i} (Y^\alpha+v_\alpha)\times\{0\}\times\{0\}\times\{0\}\Big).
\]
%\bigsqcup_{\alpha\in A_i}((Y^\alpha_i+v_\alpha)\times\{0\}\times\{0\})\subset Z_i.
Since $\gamma_i$ can be chosen such that $\gamma_i\circ\rho$ is arbitrarily $\mscr{C}^\infty_{\w}$ close to $\beta_i$, those maps are homotopic, thus the normal bundle of $S_i$ in $S$ and the normal bundle of $X_i\times\{0\}\times\{0\}$ in $Z_i$ are equivalent. Hence, the germ $(S,S_i\cup (Y\times\{0\}))$ of $S$ at $S_i\cup (Y\times\{0\})$ is diffeomorphic to the germ 
\[
\Big(Z_i\cup (\bigsqcup_{\alpha\notin A_i}(Y^\alpha+v_\alpha)\times\R\times\{0\}\times\{0\}),(X_i\times\{0\})\cup Y'_i\Big)
\]
of the $\Q$-algebraic set $Z_i\cup (\bigsqcup_{\alpha\notin A_i}(Y^\alpha+v_\alpha)\times\R\times\{0\}\times\{0\})$ $\Q$-nonsingular locally at $(X_i\times\{0\})\cup Y'_i$.

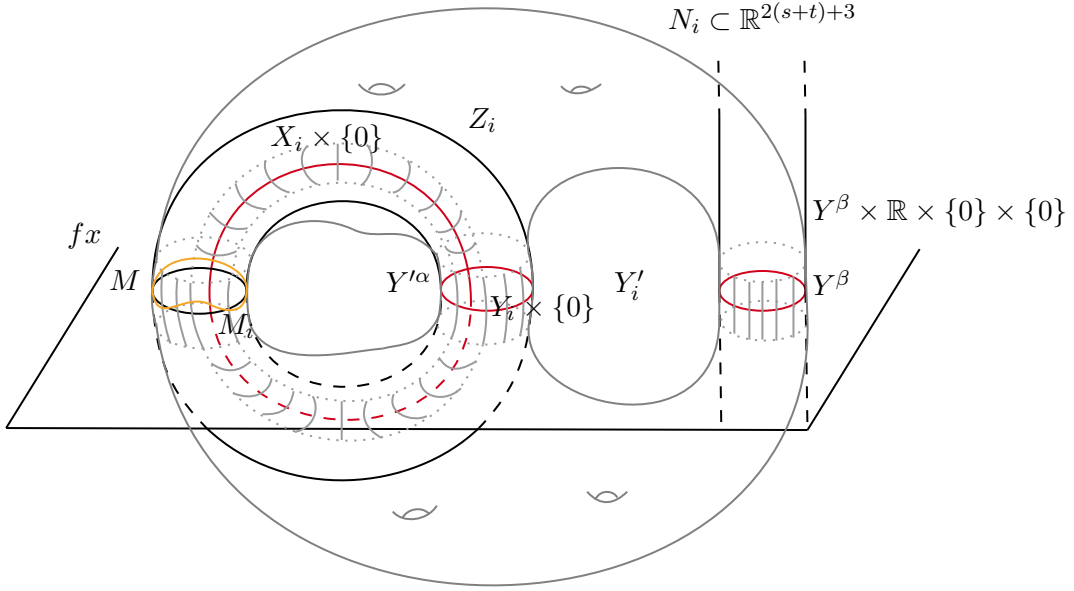
\begin{figure}[h!]
\centering
\tikzset{every picture/.style={line width=0.75pt}} %set default line width to 0.75pt        

\begin{tikzpicture}[x=0.75pt,y=0.75pt,yscale=-1,xscale=1]
%uncomment if require: \path (0,350); %set diagram left start at 0, and has height of 350

%Shape: Ellipse [id:dp8106675034997846] 
\draw  [color={rgb, 255:red, 208; green, 2; blue, 27 }  ,draw opacity=1 ][line width=0.75]  (310,169) .. controls (310,162.92) and (320.3,158) .. (333,158) .. controls (345.7,158) and (356,162.92) .. (356,169) .. controls (356,175.08) and (345.7,180) .. (333,180) .. controls (320.3,180) and (310,175.08) .. (310,169) -- cycle ;
%Shape: Ellipse [id:dp5089481920638588] 
\draw  [color={rgb, 255:red, 208; green, 2; blue, 27 }  ,draw opacity=1 ][line width=0.75]  (449,170) .. controls (449,164.48) and (458.63,160) .. (470.5,160) .. controls (482.37,160) and (492,164.48) .. (492,170) .. controls (492,175.52) and (482.37,180) .. (470.5,180) .. controls (458.63,180) and (449,175.52) .. (449,170) -- cycle ;
%Curve Lines [id:da7491794051962015] 
\draw [line width=0.75]    (213,170) .. controls (212,110) and (309,110) .. (310,169) ;
%Curve Lines [id:da4324190909221365] 
\draw [color={rgb, 255:red, 208; green, 2; blue, 27 }  ,draw opacity=1 ][line width=0.75]    (195,180) .. controls (185,85) and (327,79) .. (325,178) ;
%Straight Lines [id:da7610653475666803] 
\draw [line width=0.75]    (93,239) -- (493,240) ;
%Straight Lines [id:da8333095787036428] 
\draw [line width=0.75]    (493,240) -- (549,149) ;
%Straight Lines [id:da5328885917818029] 
\draw [line width=0.75]    (93,239) -- (149,148) ;
%Curve Lines [id:da32540428609902294] 
\draw [line width=0.75]  [dash pattern={on 4.5pt off 4.5pt}]  (213,170) .. controls (211,236) and (315,233) .. (310,169) ;
%Curve Lines [id:da019317282548600212] 
\draw [color={rgb, 255:red, 208; green, 2; blue, 27 }  ,draw opacity=1 ][line width=0.75]  [dash pattern={on 4.5pt off 4.5pt}]  (195,180) .. controls (204,251) and (323,256) .. (325,178) ;
%Shape: Ellipse [id:dp6724734276468394] 
\draw  [line width=0.75]  (166,170) .. controls (166,163.65) and (176.52,158.5) .. (189.5,158.5) .. controls (202.48,158.5) and (213,163.65) .. (213,170) .. controls (213,176.35) and (202.48,181.5) .. (189.5,181.5) .. controls (176.52,181.5) and (166,176.35) .. (166,170) -- cycle ;
%Curve Lines [id:da8545354249932194] 
\draw [line width=0.75]    (166,170) .. controls (165,49) and (355,49) .. (356,169) ;
%Curve Lines [id:da4271247483886591] 
\draw [line width=0.75]  [dash pattern={on 4.5pt off 4.5pt}]  (166,170) .. controls (167,211) and (180,225) .. (193,239) ;
%Curve Lines [id:da400567453749579] 
\draw [line width=0.75]    (193,239) .. controls (233,277) and (296,271) .. (329,239) ;
%Curve Lines [id:da7799034520460064] 
\draw [line width=0.75]  [dash pattern={on 4.5pt off 4.5pt}]  (356,169) .. controls (354,191) and (354,213) .. (329,239) ;
%Straight Lines [id:da41114738098507586] 
\draw [line width=0.75]    (449,82) -- (449,170) ;
%Straight Lines [id:da1313418795408623] 
\draw [line width=0.75]  [dash pattern={on 4.5pt off 4.5pt}]  (492,170) -- (493,240) ;
%Straight Lines [id:da17593857059182838] 
\draw [line width=0.75]    (492,82) -- (492,170) ;
%Straight Lines [id:da2896859279750269] 
\draw [line width=0.75]  [dash pattern={on 4.5pt off 4.5pt}]  (449,170) -- (450,240) ;
%Curve Lines [id:da8406520845344945] 
\draw [color={rgb, 255:red, 155; green, 155; blue, 155 }  ,draw opacity=1 ][line width=0.75]  [dash pattern={on 0.84pt off 2.51pt}]  (206,162) .. controls (210,107) and (300,94) .. (317,161) ;
%Shape: Ellipse [id:dp34562514191614224] 
\draw  [color={rgb, 255:red, 155; green, 155; blue, 155 }  ,draw opacity=1 ][dash pattern={on 0.84pt off 2.51pt}][line width=0.75]  (168,154.5) .. controls (168,148.15) and (178.52,143) .. (191.5,143) .. controls (204.48,143) and (215,148.15) .. (215,154.5) .. controls (215,160.85) and (204.48,166) .. (191.5,166) .. controls (178.52,166) and (168,160.85) .. (168,154.5) -- cycle ;
%Shape: Ellipse [id:dp8217137854653292] 
\draw  [color={rgb, 255:red, 155; green, 155; blue, 155 }  ,draw opacity=1 ][dash pattern={on 0.84pt off 2.51pt}][line width=0.75]  (168,187.5) .. controls (168,181.15) and (178.52,176) .. (191.5,176) .. controls (204.48,176) and (215,181.15) .. (215,187.5) .. controls (215,193.85) and (204.48,199) .. (191.5,199) .. controls (178.52,199) and (168,193.85) .. (168,187.5) -- cycle ;
%Shape: Ellipse [id:dp33473918811684966] 
\draw  [color={rgb, 255:red, 155; green, 155; blue, 155 }  ,draw opacity=1 ][dash pattern={on 0.84pt off 2.51pt}][line width=0.75]  (308,186) .. controls (308,179.92) and (318.3,175) .. (331,175) .. controls (343.7,175) and (354,179.92) .. (354,186) .. controls (354,192.08) and (343.7,197) .. (331,197) .. controls (318.3,197) and (308,192.08) .. (308,186) -- cycle ;
%Shape: Ellipse [id:dp47601699215440174] 
\draw  [color={rgb, 255:red, 155; green, 155; blue, 155 }  ,draw opacity=1 ][dash pattern={on 0.84pt off 2.51pt}][line width=0.75]  (450,185) .. controls (450,179.48) and (459.63,175) .. (471.5,175) .. controls (483.37,175) and (493,179.48) .. (493,185) .. controls (493,190.52) and (483.37,195) .. (471.5,195) .. controls (459.63,195) and (450,190.52) .. (450,185) -- cycle ;
%Shape: Ellipse [id:dp5042763306011632] 
\draw  [color={rgb, 255:red, 155; green, 155; blue, 155 }  ,draw opacity=1 ][dash pattern={on 0.84pt off 2.51pt}][line width=0.75]  (308,152) .. controls (308,145.92) and (318.3,141) .. (331,141) .. controls (343.7,141) and (354,145.92) .. (354,152) .. controls (354,158.08) and (343.7,163) .. (331,163) .. controls (318.3,163) and (308,158.08) .. (308,152) -- cycle ;
%Shape: Ellipse [id:dp8852432002341147] 
\draw  [color={rgb, 255:red, 155; green, 155; blue, 155 }  ,draw opacity=1 ][dash pattern={on 0.84pt off 2.51pt}][line width=0.75]  (449,155.5) .. controls (449,149.98) and (458.63,145.5) .. (470.5,145.5) .. controls (482.37,145.5) and (492,149.98) .. (492,155.5) .. controls (492,161.02) and (482.37,165.5) .. (470.5,165.5) .. controls (458.63,165.5) and (449,161.02) .. (449,155.5) -- cycle ;
%Curve Lines [id:da49712641224371734] 
\draw [color={rgb, 255:red, 155; green, 155; blue, 155 }  ,draw opacity=1 ][line width=0.75]  [dash pattern={on 0.84pt off 2.51pt}]  (191.5,199) .. controls (214,254) and (306,268) .. (331,197) ;
%Curve Lines [id:da2071502476578262] 
\draw [color={rgb, 255:red, 155; green, 155; blue, 155 }  ,draw opacity=1 ][line width=0.75]  [dash pattern={on 0.84pt off 2.51pt}]  (209,196) .. controls (228,231) and (295,240) .. (315,194) ;
%Curve Lines [id:da9562289985526179] 
\draw [color={rgb, 255:red, 155; green, 155; blue, 155 }  ,draw opacity=1 ][line width=0.75]  [dash pattern={on 0.84pt off 2.51pt}]  (185.5,163) .. controls (191,79) and (319,68) .. (331,163) ;
%Curve Lines [id:da07460479666374342] 
\draw [color={rgb, 255:red, 155; green, 155; blue, 155 }  ,draw opacity=1 ][line width=0.75]    (171,160) .. controls (169.5,173) and (171.5,186.5) .. (175,196.5) ;
%Curve Lines [id:da36538048894744013] 
\draw [color={rgb, 255:red, 155; green, 155; blue, 155 }  ,draw opacity=1 ][line width=0.75]    (178.5,164) .. controls (177,176) and (179,186) .. (183,198) ;
%Curve Lines [id:da7322944156813784] 
\draw [color={rgb, 255:red, 155; green, 155; blue, 155 }  ,draw opacity=1 ][line width=0.75]    (185,166) .. controls (184.5,180) and (187,190) .. (191.5,199) ;
%Curve Lines [id:da4981561536400487] 
\draw [color={rgb, 255:red, 155; green, 155; blue, 155 }  ,draw opacity=1 ][line width=0.75]    (204.5,165) .. controls (202.5,175) and (204.5,186.5) .. (209,196) ;
%Curve Lines [id:da3428149869962023] 
\draw [color={rgb, 255:red, 128; green, 128; blue, 128 }  ,draw opacity=1 ][line width=0.75]    (215,154.5) .. controls (212.5,165.5) and (213,181) .. (215,187.5) ;
%Curve Lines [id:da03123586206750195] 
\draw [color={rgb, 255:red, 155; green, 155; blue, 155 }  ,draw opacity=1 ][line width=0.75]    (188,147) .. controls (192,153) and (201.5,153.5) .. (208,150) ;
%Curve Lines [id:da3642522357797848] 
\draw [color={rgb, 255:red, 155; green, 155; blue, 155 }  ,draw opacity=1 ][line width=0.75]    (194.5,130.5) .. controls (198.5,134.5) and (202.5,140.5) .. (213.5,138.5) ;
%Curve Lines [id:da334118316444021] 
\draw [color={rgb, 255:red, 155; green, 155; blue, 155 }  ,draw opacity=1 ][line width=0.75]    (207.5,116) .. controls (209.5,124) and (213,128) .. (223,129) ;
%Curve Lines [id:da14675051286216478] 
\draw [color={rgb, 255:red, 155; green, 155; blue, 155 }  ,draw opacity=1 ][line width=0.75]    (222.5,104.5) .. controls (220.5,110.5) and (227,119) .. (234.5,120.5) ;
%Curve Lines [id:da032523682882883764] 
\draw [color={rgb, 255:red, 155; green, 155; blue, 155 }  ,draw opacity=1 ][line width=0.75]    (243.5,97) .. controls (237.5,102.5) and (241.5,113.5) .. (248.5,116.5) ;
%Curve Lines [id:da21869191565981516] 
\draw [color={rgb, 255:red, 155; green, 155; blue, 155 }  ,draw opacity=1 ][line width=0.75]    (310.5,117.5) .. controls (310.5,122) and (306,129.5) .. (299,129.5) ;
%Curve Lines [id:da9912775093707827] 
\draw [color={rgb, 255:red, 155; green, 155; blue, 155 }  ,draw opacity=1 ][line width=0.75]    (293.5,104.5) .. controls (297,110) and (291,120.5) .. (286,120.5) ;
%Curve Lines [id:da44716263437833237] 
\draw [color={rgb, 255:red, 155; green, 155; blue, 155 }  ,draw opacity=1 ][line width=0.75]    (312.5,148.5) .. controls (316.5,153.5) and (326.5,150) .. (327,145.5) ;
%Curve Lines [id:da793544057298582] 
\draw [color={rgb, 255:red, 155; green, 155; blue, 155 }  ,draw opacity=1 ][line width=0.75]    (273,96.5) .. controls (277.5,101.5) and (276,113) .. (270,117) ;
%Straight Lines [id:da12699395275100345] 
\draw [color={rgb, 255:red, 155; green, 155; blue, 155 }  ,draw opacity=1 ][line width=0.75]    (258.5,96) -- (258.5,114.5) ;
%Curve Lines [id:da1537647644999035] 
\draw [color={rgb, 255:red, 155; green, 155; blue, 155 }  ,draw opacity=1 ][line width=0.75]    (308.5,140.5) .. controls (315.5,141) and (320.5,135.5) .. (320,129.5) ;
%Curve Lines [id:da6201271766707364] 
\draw [color={rgb, 255:red, 155; green, 155; blue, 155 }  ,draw opacity=1 ][line width=0.75]    (273,225.5) .. controls (270,232.5) and (271.5,239) .. (280.5,244) ;
%Curve Lines [id:da04843508929988194] 
\draw [color={rgb, 255:red, 155; green, 155; blue, 155 }  ,draw opacity=1 ][line width=0.75]    (195.5,206.5) .. controls (202.5,211.5) and (209,211) .. (215,205.5) ;
%Curve Lines [id:da402238175952082] 
\draw [color={rgb, 255:red, 155; green, 155; blue, 155 }  ,draw opacity=1 ][line width=0.75]    (221,232) .. controls (231.5,232.5) and (235,228.5) .. (236.5,220.5) ;
%Curve Lines [id:da04234900724733126] 
\draw [color={rgb, 255:red, 155; green, 155; blue, 155 }  ,draw opacity=1 ][line width=0.75]    (237,241) .. controls (247,239) and (249.5,229) .. (246,223) ;
%Straight Lines [id:da17528354409623226] 
\draw [color={rgb, 255:red, 155; green, 155; blue, 155 }  ,draw opacity=1 ][line width=0.75]    (260.5,225.5) -- (260.5,245) ;
%Curve Lines [id:da3014671586976465] 
\draw [color={rgb, 255:red, 155; green, 155; blue, 155 }  ,draw opacity=1 ][line width=0.75]    (286.5,222) .. controls (282.5,231.5) and (295.5,236.5) .. (303,233) ;
%Curve Lines [id:da9060072607832064] 
\draw [color={rgb, 255:red, 155; green, 155; blue, 155 }  ,draw opacity=1 ][line width=0.75]    (299.5,214) .. controls (299.5,224) and (312.5,226.5) .. (320,219) ;
%Curve Lines [id:da21787050470134384] 
\draw [color={rgb, 255:red, 155; green, 155; blue, 155 }  ,draw opacity=1 ][line width=0.75]    (309,203.5) .. controls (312.5,209) and (317,212) .. (327.5,205.5) ;
%Curve Lines [id:da11110062106164387] 
\draw [color={rgb, 255:red, 155; green, 155; blue, 155 }  ,draw opacity=1 ][line width=0.75]    (206.5,221.5) .. controls (216,225.5) and (224.5,221.5) .. (225.5,214) ;
%Curve Lines [id:da18226850721713694] 
\draw [color={rgb, 255:red, 155; green, 155; blue, 155 }  ,draw opacity=1 ][line width=0.75]    (317,161) .. controls (319,169.5) and (319,187.5) .. (315,194) ;
%Curve Lines [id:da29171535051183606] 
\draw [color={rgb, 255:red, 155; green, 155; blue, 155 }  ,draw opacity=1 ][line width=0.75]    (331,163) .. controls (333,173.5) and (333.5,186.5) .. (331,197) ;
%Curve Lines [id:da7689325185781456] 
\draw [color={rgb, 255:red, 155; green, 155; blue, 155 }  ,draw opacity=1 ][line width=0.75]    (349.5,158.5) .. controls (351.5,167.5) and (352,183.5) .. (349.5,192.5) ;
%Curve Lines [id:da6382616492533212] 
\draw [color={rgb, 255:red, 155; green, 155; blue, 155 }  ,draw opacity=1 ][line width=0.75]    (341.5,162) .. controls (343.5,172.5) and (343.5,185.5) .. (341.5,196) ;
%Straight Lines [id:da0076319519170964245] 
\draw [color={rgb, 255:red, 155; green, 155; blue, 155 }  ,draw opacity=1 ][line width=0.75]    (477.5,164) -- (477.5,195) ;
%Straight Lines [id:da7100176306580517] 
\draw [color={rgb, 255:red, 155; green, 155; blue, 155 }  ,draw opacity=1 ][line width=0.75]    (456.5,162.5) -- (456.5,191.5) ;
%Straight Lines [id:da7953192180584282] 
\draw [color={rgb, 255:red, 155; green, 155; blue, 155 }  ,draw opacity=1 ][line width=0.75]    (464,165.5) -- (464,194.5) ;
%Straight Lines [id:da7352676023887846] 
\draw [color={rgb, 255:red, 155; green, 155; blue, 155 }  ,draw opacity=1 ][line width=0.75]    (470.5,165.5) -- (470.5,194.5) ;
%Straight Lines [id:da9325794992814571] 
\draw [color={rgb, 255:red, 155; green, 155; blue, 155 }  ,draw opacity=1 ][line width=0.75]    (486,162.5) -- (486,193.5) ;
%Straight Lines [id:da15880805288857014] 
\draw [color={rgb, 255:red, 128; green, 128; blue, 128 }  ,draw opacity=1 ][line width=0.75]    (449,155.5) -- (449,184.5) ;
%Straight Lines [id:da4052568809846453] 
\draw [color={rgb, 255:red, 128; green, 128; blue, 128 }  ,draw opacity=1 ][line width=0.75]    (492,155.5) -- (493,185) ;
%Curve Lines [id:da9524513195662432] 
\draw [color={rgb, 255:red, 128; green, 128; blue, 128 }  ,draw opacity=1 ][line width=0.75]    (354,152) .. controls (357,161.5) and (356,178.5) .. (354,186) ;
%Curve Lines [id:da7676117780474452] 
\draw [color={rgb, 255:red, 128; green, 128; blue, 128 }  ,draw opacity=1 ][line width=0.75]    (308,152) .. controls (310.5,160) and (311,177.5) .. (308,186) ;
%Curve Lines [id:da15450925499063106] 
\draw [color={rgb, 255:red, 128; green, 128; blue, 128 }  ,draw opacity=1 ][line width=0.75]    (168,154.5) .. controls (165.5,164) and (164.5,178.5) .. (168,187.5) ;
%Curve Lines [id:da9812552787673235] 
\draw [color={rgb, 255:red, 128; green, 128; blue, 128 }  ,draw opacity=1 ][line width=0.75]    (168,154.5) .. controls (183,-17) and (490,-12) .. (492,155.5) ;
%Curve Lines [id:da22974569743479378] 
\draw [color={rgb, 255:red, 128; green, 128; blue, 128 }  ,draw opacity=1 ][line width=0.75]    (354,152) .. controls (344,96) and (450,90) .. (449,155.5) ;
%Curve Lines [id:da3953902743322475] 
\draw [color={rgb, 255:red, 128; green, 128; blue, 128 }  ,draw opacity=1 ][line width=0.75]    (269,66) .. controls (274,73) and (286,73) .. (292,66) ;
%Curve Lines [id:da542496338968459] 
\draw [color={rgb, 255:red, 128; green, 128; blue, 128 }  ,draw opacity=1 ][line width=0.75]    (274,69) .. controls (277,64) and (285,66) .. (287,69) ;
%Curve Lines [id:da06548350500733602] 
\draw [color={rgb, 255:red, 128; green, 128; blue, 128 }  ,draw opacity=1 ][line width=0.75]    (370,67) .. controls (374,72) and (382,72) .. (390,66) ;
%Curve Lines [id:da7045156596816785] 
\draw [color={rgb, 255:red, 128; green, 128; blue, 128 }  ,draw opacity=1 ][line width=0.75]    (375,70) .. controls (377,66) and (383,67) .. (385,69) ;
%Curve Lines [id:da624200971478155] 
\draw [color={rgb, 255:red, 128; green, 128; blue, 128 }  ,draw opacity=1 ][line width=0.75]    (215,154.5) .. controls (221,136) and (246,130) .. (266,139) ;
%Curve Lines [id:da17801599714159333] 
\draw [color={rgb, 255:red, 128; green, 128; blue, 128 }  ,draw opacity=1 ][line width=0.75]    (266,139) .. controls (280,146) and (292,134) .. (308,152) ;
%Curve Lines [id:da7306159358952146] 
\draw [color={rgb, 255:red, 128; green, 128; blue, 128 }  ,draw opacity=1 ][line width=0.75]    (168,187.5) .. controls (192,369) and (498,355) .. (493,185) ;
%Curve Lines [id:da30626640658699145] 
\draw [color={rgb, 255:red, 128; green, 128; blue, 128 }  ,draw opacity=1 ][line width=0.75]    (354,186) .. controls (348,230) and (449,251.5) .. (449,184.5) ;
%Curve Lines [id:da5765862396431445] 
\draw [color={rgb, 255:red, 128; green, 128; blue, 128 }  ,draw opacity=1 ][line width=0.75]    (215,187.5) .. controls (225,211) and (267,200) .. (281,199) ;
%Curve Lines [id:da7870586993478003] 
\draw [color={rgb, 255:red, 128; green, 128; blue, 128 }  ,draw opacity=1 ][line width=0.75]    (281,199) .. controls (288,197) and (304,198) .. (310,180) ;
%Curve Lines [id:da7661267417377097] 
\draw [color={rgb, 255:red, 245; green, 166; blue, 35 }  ,draw opacity=1 ][line width=0.75]    (166,170) .. controls (169,144) and (215,153) .. (213,170) ;
%Curve Lines [id:da2881154651852632] 
\draw [color={rgb, 255:red, 245; green, 166; blue, 35 }  ,draw opacity=1 ][line width=0.75]    (166,170) .. controls (168.5,189) and (182,170) .. (195,177) ;
%Curve Lines [id:da46439090514165116] 
\draw [color={rgb, 255:red, 245; green, 166; blue, 35 }  ,draw opacity=1 ][line width=0.75]    (195,177) .. controls (203,181) and (212.5,182) .. (213,170) ;
%Curve Lines [id:da008228408188737713] 
\draw [color={rgb, 255:red, 128; green, 128; blue, 128 }  ,draw opacity=1 ][line width=0.75]    (291,284) .. controls (293,280) and (300,278) .. (304,282) ;
%Curve Lines [id:da38092854194997483] 
\draw [color={rgb, 255:red, 128; green, 128; blue, 128 }  ,draw opacity=1 ][line width=0.75]    (383,271) .. controls (390,279) and (397,277) .. (403,271) ;
%Curve Lines [id:da003088853534576552] 
\draw [color={rgb, 255:red, 128; green, 128; blue, 128 }  ,draw opacity=1 ][line width=0.75]    (286,281) .. controls (291,287) and (303,286) .. (308,278) ;
%Curve Lines [id:da9477768952078941] 
\draw [color={rgb, 255:red, 128; green, 128; blue, 128 }  ,draw opacity=1 ][line width=0.75]    (387,275) .. controls (390,269) and (397,271) .. (398,274) ;
%Straight Lines [id:da5406558484676577] 
\draw [line width=0.75]  [dash pattern={on 4.5pt off 4.5pt}]  (448.5,54.25) -- (449,82) ;
%Straight Lines [id:da5403748335327383] 
\draw [line width=0.75]  [dash pattern={on 4.5pt off 4.5pt}]  (491.5,54.25) -- (492,82) ;

% Text Node
\draw (143,158) node [anchor=north west][inner sep=0.75pt]    {$M$};
% Text Node
\draw (197,180) node [anchor=north west][inner sep=0.75pt]    {$M_{i}$};
% Text Node
\draw (281,158) node [anchor=north west][inner sep=0.75pt]    {$Y^{\prime \alpha }$};
% Text Node
\draw (333,170) node [anchor=north west][inner sep=0.75pt]    {$Y_{i} \times \{0\}$};
% Text Node
\draw (494,158) node [anchor=north west][inner sep=0.75pt]    {$Y^{\beta }$};
% Text Node
\draw (494,120) node [anchor=north west][inner sep=0.75pt]    {$Y^{\beta }\times\R\times\{0\}\times\{0\}$};
% Text Node
\draw (322,76.4) node [anchor=north west][inner sep=0.75pt]    {$Z_{i}$};
% Text Node
\draw (223.5,84.9) node [anchor=north west][inner sep=0.75pt]    {$X_{i} \times \{0\}$};
% Text Node
\draw (422,23.4) node [anchor=north west][inner sep=0.75pt]    {$N_{i} \subset \mathbb{R}^{2( s+t) +3}$};
% Text Node
\draw (122,134.4) node [anchor=north west][inner sep=0.75pt]    {$fx$};
% Text Node
\draw (395,158) node [anchor=north west][inner sep=0.75pt]    {$Y'_{i}$};
\end{tikzpicture}
\caption{Topological construction of $N_i$ with $i\in\alpha$ and $i\notin\beta$.}\label{fig:relative-tognoli-2}
\end{figure}

Let $\phi_i:U_i\to V_i$ be the above $\mscr{C}^\infty$ diffeomorphism between a neighborhood $U_i$ of $S_i\cup(Y\times\{0\})$ in $S$ and a neighborhood $V_i$ of $(X_i\times\{0\})\cup Y'_i$ in $Z_i\cup (\bigsqcup_{\alpha\notin A_i}(Y^\alpha+v_\alpha)\times\R\times\{0\}\times\{0\})$ such that $\phi_i|_{S_i}=\rho_i\times\{0\}$ and $\phi_i|_{Y^{\alpha}}$ is the above $\Q$-biregular isomorphism for every $\alpha\in A_i$, and $\phi_i|_{Y^{\alpha}}$ is the inclusion map for every $\alpha\notin A_i$. Let $V'_i\subset V_i$ be a neighborhood of $(X_i\times\{0\})\cup Y'_i$ in $Z_i\cup (\bigsqcup_{\alpha\notin A_i}(Y^\alpha+v_\alpha)\times\R\times\{0\}\times\{0\})$  such that $\overline{V'_i}\subsetneq V_i$. Set $A_i:=\phi_i^{-1}(\overline{V'_i})\subset U_i$ closed neighborhood of $S_i\cup (Y\times\{0\})$ in $S$ and consider the map $\phi_i|_{A_i}:A_i\to \R^{s+1+t}\times\R^{s+1+t+1}$. Since $2(s+1+t)+1 \geq 2(d+1)+1$, Tietze's theorem ensures the existence of a continuous extension of $\phi_i$ from the whole $S$ to $\R^{s+1+t}\times\R^{s+1+t+1}$, we can apply to $\phi_i|_{A_i}$ the extension theorem \cite[Theorem 5$(\mr{f})$]{Whi36} of Whitney, obtaining a $\mscr{C}^\infty$ embedding $\phi'_i:S\to \R^{s+1+t}\times\R^{s+1+t+1}$ extending $\phi_i|_{A_i}$. Thus, there exists a $\mscr{C}^\infty$ manifold $N_i\subset\R^{s+1+t}\times\R^{s+1+t+1}$ which is $\mscr{C}^\infty$ diffeomorphic to $S$ via $\phi'_i$ and, by construction, the following properties are satisfied:
\begin{itemize}
\item[$(\mr{i})$] $(X_i\times\{0\})\cup Y'_i\subset N_i$;
\item[$(\mr{ii})$] the germ of $N_i$ at $(X_i\times\{0\})\cup Y'_i$ is the germ of a $\Q$-nonsingular $\Q$-algebraic set.
\end{itemize}
Since $X_i\times\{0\}$ is a $\Q$-nonsingular $\Q$-algebraic subset of $N_i$ and $Y'_i$ is a $\Q$-nonsingular $\Q$-algebraic hypersurface of $N_i$ satisfying above property (ii), Lemma \ref{lem:Q-nice-hyp} ensures that $(X_i\times\{0\})\cup Y'_i\subset\R^{s+1+t}\times\R^{s+1+t+1}$ is a $\Q$-stable $\Q$-algebraic set. An application of \cite[Theorem\,3.9]{GSa} with $``L":=(X_i\times\{0\})\cup Y'_i$, $``M":=N_i$ and $``W":=\{0\}$ gives $t_i\in\N$, with $t_i\geq 2(s+1+t)+1 $, a projectively $\Q$-closed $\Q$-nonsingular $\Q$-algebraic set $X^i\subset\R^{t_i}=\R^{s+1+t}\times\R^{s+1+t+1}\times\R^{t_i-2(s+1+t)+1 }$ such that $((X_i\times\{0\})\cup Y'_i)\times\{0\}\subset X^i$ and a $\mscr{C}^\infty$ diffeomorphism $\tau_i:N_i\to X^i$ such that $\tau_i(x,0)=x$ for every $x\in (X_i\times\{0\})\cup Y'_i$. Define the diffeomorphism $\varphi_i:S\to X^i$ as $\varphi_i:=\tau_i\circ \phi'_i$, for every $i\in\{1,\dots,\ell\}$. Let $t:=\max\{t_i\}_{i\in\{1,\dots,\ell\}}$ and consider $X^i\subset\R^t$, for every $i\in\{1,\dots,\ell\}$.

Apply Lemma \ref{lem:Q-Tognoli} with the following substitutions: ``$M$''$:=M$, ``$Y$''$:=Y$, ``$S$''$:=S$, ``$W$''$:=X^1\times\dots\times X^\ell$ and ``$\varphi$''$:=\varphi_1 \times\dots\times \varphi_\ell$ obtaining $k\in\N$, a $\Q$-nonsingular $\Q$-algebraic subset $Z$ of $\R^{s+1+k}=\R^{s+1} \times \R^k$, a projectively $\Q$-closed $\Q$-nonsingular $\Q$-algebraic set $M'\subset\R^{s+1+k}$, a $\mscr{C}^\infty$ diffeomorphism $\phi:M\to M'$ and a $\Q$-regular map $\eta:Z\to W$ satisfying properties Lemma \ref{lem:Q-Tognoli}\ref{lem:Q_tognoli-1}\,\&\,\ref{lem:Q_tognoli-2}. Thus, $\psi(M_i)$, for every $i\in\{1,\dots,\ell\}$, are submanifolds of $M''$ in general position.

Let $G:=\G_{s-d,d+1}$, $E:=\E_{s-d,d+1}$ and $\beta:S \to G$ be the normal bundle map of $S$ in $\R^{s+1}$. Define $\widetilde{\beta}:U \to E$, $\widetilde{\varphi}: U \to X^1\times\dots\times X^\ell$ and $\widehat{\eta}:U'\to E\times X^1\times\dots\times X^\ell$ be defined as in the proof of Lemma \ref{lem:Q-Tognoli}. Let $\pi_i:E\times X^1\times\dots\times X^\ell\rightarrow X^i$ be the projection on the $i$-th component of $X^1\times\dots\times X^\ell$, thus $\pi_i\circ(\widetilde{\beta}\times\widetilde{\varphi})=\varphi_i\circ\widetilde{\rho}$. Let $X'_i:=X_i\times\{0\}\times\{0\}\subset X^i$, for every $i\in\{1,\dots,\ell\}$. By $(\mr{vii})$ in the proof of Lemma \ref{lem:Q-Tognoli}, we know that $\pi_i\circ\widehat{\eta}$ is arbitrarily $\mscr{C}^\infty_\w$ close to $\varphi_i\circ\widetilde{\rho}$, thus $\pi_i\circ\widehat{\eta}$ is transverse to $X'_i$ in $X^i$ for every $i\in\{1,\dots,\ell\}$. By $(\mr{v})$,$(\mr{vi})$\,\&\,$(\mr{vii})$ in the proof of Lemma \ref{lem:Q-Tognoli} and \cite[Theorem 14.1.1]{BCR98}, we have that $S'_i:=\widehat{\eta}^{\,-1}((G\times\{0\}) \times X^1\times\dots\times X'_i\times\dots\times X^\ell)=S'\cap(\pi_i\circ\widehat{\eta})^{-1}(X'_i)$ is a compact $\mscr{C}^\infty$ submanifold of $S\subset U''$ containing $Y_i\times\{0\}\times\{0\}$ and there exists a $\mscr{C}^\infty$ diffeomorphism $\psi^i_1:U''\to U''$ arbitrarily $\mscr{C}^\infty_\w$ close to $\mr{id}_{U''}$ such that $\psi^i_1(S_i)=S'_i$ and $\psi^i_1=\mr{id}_{U''}$ on $Y_i\times\{0\}\times\{0\}$. Moreover, by $(\mr{v})$ in the proof of Lemma \ref{lem:Q-Tognoli}, \cite[Lemma\,2.19]{GSa} ensures that $S''_i:=\eta^{-1}((G\times\{0\}) \times X^1\times\dots\times X'_i\times\dots\times X^\ell)=S''\cap(\pi_i\circ\eta)^{-1}(X'_i)\subset\R^{s+1+k}$ is a  $\Q$-algebraic set such that  $S''_1:=S''\cap Z_1=(\pi|_{Z_1})^{-1}(S')\subset\Reg^{*}(S'')$
In addition, the $\mscr{C}^\infty$ embedding $\psi^i_2:S_i\to\R^{s+1+k}$ defined by $\psi_2^i(x,x_{s+1}):=(\pi|_{Z_1})^{-1}(\psi^i_1(x,x_{s+1}))$, is arbitrarily $\mscr{C}^\infty_\w$ close to the inclusion map $j_{S_i}:S_i\hookrightarrow\R^{s+1+k}$, $\psi^i_2=j_{S_i}$ on $Y_i\times\{0\}$ and $\psi^i_2(S_i)=S''_{i1}$. Note that the set $S''_{i1}$ is both compact and open in $S''_i$; thus, $S''_{i1}$ is the union of some connected components of $S''_i$ and $S''_{i2}:=S''_i\setminus S''_{i1}$ is a closed subset of $\R^{s+1+k}$. Since $\psi^i_2$ is arbitrarily $\mscr{C}^\infty_\w$ close to $j_{S_i}$, the coordinate hyperplane $\{x_{s+1}=0\}$ of $\R^{s+1+k}$ is transverse to $S''_{i1}$ in $\R^{s+1+k}$, $S''_{i1}\cap\{x_{s+1}=0\}=M'_i \sqcup (Y_i\times\{0\}\times\{0\})$ for some compact $\mscr{C}^\infty$ submanifold $M'_i$ of $\R^{s+1+k}$ and there exists a $\mscr{C}^\infty$ embedding $\psi^i_3:M_i\to\R^{s+1+k}$ arbitrarily $\mscr{C}^\infty_\w$ close to the inclusion map $j_{M_i}:M_i\hookrightarrow\R^{s+1+k}$ such that $M'_i=\psi^i_3(M_i)$. Observe that, by construction $M'_i\subset M'$, for every $i\in\{1,\dots,\ell\}$, are in general position. Define $M''_i:=M''\cap S''_{i}$, for every $i\in\{1,\dots,\ell\}$. By $(\mr{ix})$,$(\mr{x})$\,\&\,$(\mr{xi})$, we deduce that the $M''_i$'s are $\Q$-nonsingular $\Q$-algebraic subsets of $M''$ in general position and there exists a $\mscr{C}^\infty$ embedding $\psi^i_4:M_i\to\R^{s+1+k}$ arbitrarily $\mscr{C}^\infty_\w$ close to the inclusion map $j_{M'_i}:M'_i\hookrightarrow\R^{s+1+k}$ such that $M''_i=\psi^i_4(M'_i)$, for every $i\in\{1,\dots,\ell\}$. Consider the embedding $\psi_i:M_i\to \R^{s+1+k}$ defined as $\psi_i:=\psi^i_3\circ\psi^i_4$. Then, $\psi_i$ is $\mscr{C}^\infty_\w$ close to $j_{M_i}$ and $\psi(M_i)=M''_i$, for every $i\in\{1,\dots,\ell\}$. As a consequence, $\psi_i\circ(\psi|)^{-1}|:\psi(M_i)\to M''_i\subset M''$ is a $\mscr{C}^\infty$ diffeomorphism $\mscr{C}^\infty_\w$ close to $j_{\psi(M_i)}:\psi(M_i)\hookrightarrow \R^{s+1+k}$, for every $i\in\{1,\dots,\ell\}$. Thus, \cite[Lemma 2.9]{AK81a} ensures the existence of a $\mscr{C}^\infty$ diffeomorphism $\psi_5:M''\to M''$ such that $\psi_5(\psi(M_i))=M''_i$ and $\psi_5$ is $\mscr{C}^\infty_\w$ close to $j_{M'}:M'\hookrightarrow\R^{s+1+k}$. 

Let $m:=\max\{n,2d+1\}$ and $t:=s+1+k-m\leq 0$. Denote by $(x,y)$ the coordinates of $\R^{m+t}=\R^m\times\R^{t}$ and consider the set $\mc{M}_{m,t}(\Q)$ of rational $m\times t$-matrices endowed with the Euclidean topology induced by the one of $\mc{M}_{m,t}(\R)=\R^{mt}$. By denseness of $\Q$ in $\R$, the $\R|\Q$-generic projection theorem \cite[Theorem\,2.25]{GSa} ensures the existence of $A\in\mc{M}_{n,t}(\Q)$ arbitrarily close to the zero matrix $O$ such that the corresponding projection $\pi_A:\R^{m+t}\to\R^m$, $(x,y)\mapsto x-Ay$ (here $x$ and $y$ are interpreted as column vectors) satisfies $M''':=\pi_A(M'')\subset\R^m$ is a projectively $\Q$-closed $\Q$-nonsingular $\Q$-algebraic set and the restriction $\pi'_A|:M''\to M'''$ is a $\Q$-biregular isomorphism. %Let $\psi_6:M\to\R^m$ be the $\mscr{C}^\infty$ embedding $\psi_6:=\pi_A|\circ\phi_5$. Therefore, by choosing $A$ sufficiently close to $O$ in $\mc{M}_{n,t}(\Q)$, we may assume that $\psi_6$ is arbitrarily $\mscr{C}^\infty$ close to the inclusion map $\psi:M\hookrightarrow\R^n$. 
Therefore, by setting $``M' ":=M'''$ and $``M'_i ":=\pi_A(M''_i)$, for every $i\in\{1,\dots,\ell\}$, and the $\mscr{C}^\infty$ diffeomorphism $h:M\to M'$ as $``h ":=\pi_A\circ\psi_5\circ\psi_4\circ\psi_3$ we get the wondered projectively $\Q$-closed $\Q$-nonsingular $\Q$-algebraic model $M'\subset\R^{m}$ of $M$ with $\Q$-nonsingular $\Q$-algebraic subsets $\{M'_i\}_{i=1}^\ell$ in general position such that the $\mscr{C}^\infty$ diffeomorphism $h:M\to M'$ satisfies $\jmath\circ h\in\mathcal{U}$, $h(M_i)=M'_i$ and $\jmath\circ h|_{M_i}\in\mathcal{U}_i$ for every $i\in\{1,\dots,\ell\}$, where $\jmath:M'\hookrightarrow\R^{m}$ denotes the inclusion map.

Assume in addition that $M$ and each $M_i$ are Nash manifolds, for every $i\in\{1,\dots,\ell\}$. By \cite[Theorem\,4.2\,\&\,Corollary\,4.3]{GSa} obtained as an adaptation of the Nash approximation techniques developed in  \cite{BFR14}, we may assume that $h:M\to M'$ is a Nash diffeomorphism such that $h(M_i)=M'_i$, for every $i\in\{1,\dots,\ell\}$. Moreover, an application of \cite{Jel08} provides a semialgebraic homeomorphism $\R^m\to\R^m$ extending $h$, as desired.
\end{proof}

\begin{remark}\label{rem:Q-tognoli}
In \cite{BT92} Ballico and Tognoli proved a similar statement of Theorem \ref{thm:Q_tico_tognoli} in the simplified setting with $\ell=0$. In their work the authors refer to the following notion for a real algebraic set to be \emph{defined over $\Q$}, which was introduced by Tognoli in \cite[Definition\,3,\,p.30]{Tog78}, that is: a real algebraic set $V\subset\R^n$ is \emph{defined over $\Q$} if $\Ii(V)=\Ii_\Q(V)\R[x]$. Clearly, if a real algebraic set $V\subset\R^n$ is defined over $\Q$, then it is $\Q$-determined and if in addition $V$ is nonsingular, then it is $\Q$-nonsingular. Hence, \cite[Theorem\,0.1]{BT92} is even stronger than our Theorem \ref{thm:Q_tico_tognoli} in the simplified setting with $\ell=0$. However, the classical strategy of Nash-Tognoli theorem, which is proposed in \cite{BT92}, requires to approximate smooth functions by polynomial (and regular) functions with the additional property of having coefficients over $\Q$. Unfortunately, it is easy to construct polynomials $p\in\Q[x]$ for which $0$ is a regular value, thus they can appear in Weierstrass approximation arguments, but $\mathcal{Z}_{\R^n}(p)$ is not defined over $\Q$, that is, $\Ii(\mathcal{Z}_{\R^n}(p))\neq\Ii_\Q(\mathcal{Z}_{\R^n}(p))\R[x]$. Consider for example the polynomial $p(x)=x^3-2\in\Q[x]$ and observe that $\mathcal{Z}_{\R}(p)=\{\sqrt[3]{2}\}$, $\nabla p(\sqrt[3]{2})\neq 0$ but 
\[
\Ii_{\R}(\mathcal{Z}_\R(p))=(x-\sqrt[3]{2})\R[x]\supsetneq (x^3-2)\R[x]=\Ii_{\Q}(\mathcal{Z}_\R(p))\R[x].
\]
This simple remark proves that the argument in \cite{BT92} is not sufficient to deduce that the resulting algebraic models of compact manifolds are defined over $\Q$ in the sense of \cite{Tog78}, so the proof can not be considered valid. Moreover, when replacing the property for a nonsingular algebraic model of a compact manifold to be defined over $\Q$ with being $\Q$-nonsingular, other arguments in the classical proof of Nash-Tognoli theorem are not easily guaranteed. The clearest example is the classical lemma to separate the irreducible components of a nonsingular algebraic set whose generalization over $\Q$ is carefully proved using the notion of $\R|\Q$-regular points, see Proposition \ref{cor:Q_setminus} and its complete proof in \cite[Proposition\,2.14]{GSa}.
\end{remark}

\vspace{0.5em}

\begin{remark}\label{rmk:nonsing-rat-pts}
In the statement of Theorem \ref{thm:Q_tico_tognoli} we can add the following requirement: `` $M'\subset\R^{m}$ contains a hypersurface of rational points, that is, $\dim(\Zcl_{\R^m}(M'(\Q)))\geq d-1$".

Indeed, up to perform a small translation and rotation we may suppose that there is $a\in (M\setminus \bigcup_{i=1}^\ell M_i)\cap\Q^n$ and the tangent space $T_a M$ of $M$ at $a$ has equation over the rationals. Then, consider a sphere $\sph^{n-1}(a,r)$ centered at $a$ of radius $r\in\Q$ such that $\sph^{n-1}(a,r)\cap\bigcup_{i=1}^\ell M_i=\varnothing$. Observe that $\sph^{n-1}(a,r)\cap T_a M$ is a $\Q$-nonsingular $\Q$-algebraic set of dimension $d-1$ having Zariski dense (actually Euclidean dense) rational points. Choose neighborhoods $U'_a$ and $U_a$ of $a$ in $M$ such that $\sph^{n-1}(a,r)\cap T_a M\subset U'_a$ and $\overline{U'_a}\subset U_a$ and neighborhoods $V$ and $V'$ of $\bigcup_{i=1}^\ell M_i$ in $M$ such that $\overline{V'}\subset V$. By a partition of unity argument, we obtain a $\mscr{C}^\infty$ manifold $\widetilde{M}\subset \R^n$ such that:
\begin{enumerate}[label={(\roman*)}, ref=(\roman*)]
\item $M_{\ell+1}:=\sph^{n-1}(a,r)\cap T_a M\subset \widetilde{M}$ and $\{M_i\}_{i=1}^{\ell+1}$ are $\mscr{C}^\infty$ submanifolds of $\widetilde{M}$ in general position.
\item Since $(\sph^{n-1}(a,r)\cap T_a M)(\Q)\subset\widetilde{M}(\Q)$, then $\dim(\Zcl_{\R^m}(\widetilde{M}(\Q)))\geq d-1$.
\item We may choose $\widetilde{M}$ to be diffeomorphic to $M$, in addition, by \cite[Lemma 2.9]{AK81a}, there exists a diffeomorphism $\widetilde{\phi}:M\to \widetilde{M}$ such that $\widetilde{\phi}|_{M_i}=\id_{M_i}$, for every $i\in\{1,\dots,\ell\}$, and $\jmath_{\widetilde{M}}\circ\widetilde{\phi}$ is arbitrarily $\mscr{C}^\infty_\w$ close to $\jmath_M$, where $\jmath_{M}:M\hookrightarrow\R^n$ and $\jmath_{\widetilde{M}}:\widetilde{M}\hookrightarrow\R^n$ denote the inclusion maps.
\item Suppose that in addition $M,M_1,\dots,M_\ell\subset \R^n$ are Nash manifolds. By \cite[Corollary\,4.3]{GSa}, we may suppose that above diffeomorphism $\widetilde{\phi}:M\to \widetilde{M}$ is actually a Nash diffeomorphism such that $\widetilde{\phi}|_{M_i}=\id_{M_i}$ and $\jmath_{\widetilde{M}}\circ\widetilde{\phi}$ is arbitrarily $\mathcal{N}_\w$ close to $\jmath_M$.
\end{enumerate}
Then, it suffices to substitute $``M":=\widetilde{M}$, $``\ell":=\ell+1$ and fix $``M_{\ell+1}":=\sph^{n-1}(a,r)\cap T_a M$ in the proof of Theorem \ref{thm:Q_tico_tognoli} and observe that, being $M_{\ell+1}$ a $\Q$-nonsingular $\Q$-algebraic subset of $\R^n$ contained in $M$ such that $M_{\ell+1}\cap\bigcup_{i=1}^\ell M_i=\varnothing$ and the normal bundle map of $M$ restricted to $M_{\ell+1}$ is $\Q$-regular, we can keep $M_{\ell+1}$ fixed during the approximation steps. This ensures that $M_{\ell+1}=\sph^{n-1}(a,r)\cap T_a M\subset M'(\Q)$, thus $\dim(\Zcl_{\R^m}(M'(\Q)))\geq d-1$, as desired. %\bs
\end{remark}

\vspace{0.5em}

Observe that, if $M$ and the submanifolds $M_i$, for every $i\in\{1,\dots,\ell\}$, are compact nonsingular algebraic sets, Theorem \ref{thm:Q_tico_tognoli} provides a positive answer to the \textsc{Relative $\Q$-algebraicity problem} in the compact case.

Above result can be extended to the case $M$ and the submanifolds $M_i$, for every $i\in\{1,\dots,\ell\}$, are nonsingular algebraic sets not assumed to be compact. The strategy is to apply algebraic compactification getting a compact algebraic set with only an isolated singularity and then apply a relative variant of the strategy proposed in the proof of \cite[Theorem\,1.10]{GSa}. Next theorem provides a complete proof of our Main Theorem, hence it gives a complete positive answer to the \textsc{Relative $\Q$-algebraicity problem}.

\vspace{0.5em}

\begin{theorem}\label{thm:non-compact}
Let $V$ be a nonsingular algebraic subset of $\R^{n}$ of dimension $d$ and let $\{V_i\}_{i=1}^\ell$ be a finite family of nonsingular algebraic subsets of $V$ in general position. Set $m:=n+2d+3$. Then, for every neighborhood $\mathcal{U}$ of the inclusion map $\iota:V\hookrightarrow\R^{m}$ in $\Nn_{\w}(V,\R^{m})$ and for every neighborhood $\mathcal{U}_i$ of the inclusion map $\iota|_{V_i}:V_i\hookrightarrow\R^m$ in $\Nn_{\w}(V_i,\R^m)$ for every $i\in\{1,\dots,\ell\}$, there exist a $\Q$-nonsingular $\Q$-algebraic set $V'\subset\R^m$, a family $\{V'_i\}_{i=1}^\ell$ of $\Q$-nonsingular $\Q$-algebraic subsets of $V'$ in general position and a Nash diffeomorphism $h:V\to V'$ which simultaneously takes each $V_i$ to $V'_i$ such that, if $\jmath:V'\hookrightarrow\R^m$ denotes the inclusion map, then $\jmath\circ h\in\mathcal{U}$ and $\jmath\circ h|_{V_i}\in\mathcal{U}_i$ for every $i\in\{1,\dots,\ell\}$. Moreover, $h$ extends to a semialgebraic homeomorphism from $\R^m$ to $\R^m$.
\end{theorem}

\begin{proof}
Let $c_i$ be the codimension of $M_i$ in $M$ for every $i\in\{1,\dots,\ell\}$. We can assume $V$ is noncompact. If $V=\R^n$, then it suffices to identify $V$ with the algebraic set $V\times\{0\}\subset\R^{n+1}=\R^n\times\R$ and next proof continues to work with the same estimate $m=n+2d+3$. Up to translate $V$ and each $V_i$ with $i\in\{1,\dots,\ell\}$, of a very small vector we may suppose that the origin $0$ of $\R^{n}$ is not contained in $V$. Let $s,s_1,\dots,s_\ell\in\R[x]$ such that $\mathcal{Z}_{\R}(s)=V$ and $\mathcal{Z}_{\R}(s_i)=V_i$, for every $i\in\{1,\dots,\ell\}$. Let $\sph^{n-1}$ be the standard unit sphere of $\R^{n}$ and let $\theta:\R^{n}\setminus\{0\}\to\R^{n}\setminus\{0\}$ as $\theta(x)=\frac{x}{| x |_{n}^2}$ be the inversion with respect to $\sph^{n-1}$. Recall that $\theta\circ\theta=\id_{\R^n\setminus\{0\}}$. Let $e\geq \max\{\deg(s),\deg(s_1),\dots,\deg(s_\ell)\}$. Define the polynomials $t:=| x  |^{2e}_{n}\cdot (s\circ\theta)(x)\in\R[x]$, $t_i:=| x  |^{2e}_{n}(s_i\circ\theta(x))\in\R[x]$, the compact algebraic sets $\widetilde{V}:=\mathcal{Z}_\R(t)$ and $\widetilde{V}_i:=\mathcal{Z}_\R(t_i)$, for every $i\in\{1,\dots,\ell\}$. By construction, $\widetilde{V}=\theta(V)\sqcup\{0\}$, $\widetilde{V_i}=\theta(V)_i\sqcup\{0\}$, for every $i\in\{1,\dots,\ell\}$, and $\theta:V\to \widetilde{V}\setminus\{0\}$ is a $\Q$-biregular map between the algebraic set $V$ and the Zariski open subset $\widetilde{V}\setminus\{0\}$ of $\widetilde{V}$. In general, $0$ may be a singular point of $\widetilde{V}$ and $\widetilde{V}_i$ for $i\in\{1,\dots,\ell\}$.

By a relative version of Hironaka's desingularization theorem (see \cite[Lemma 6.2.3]{AK92a}) there are a finite set $J\subset\N\setminus\{1,\dots,\ell\}$, nonsingular algebraic sets $X$, $X_i$ and $E_j$, for every $i\in\{1,\dots,\ell\}$ and $j\in J$, and a regular map $p:X\to \widetilde{V}$ satisfying the following properties:
\begin{enumerate}[label=(\roman*), ref=(\roman*)]
\item $E_j$ is an algebraic hypersurface of $X$ for every $j\in J$ and $\bigcup_{j\in J} E_j=p^{-1}(0)$;
\item the nonsingular algebraic sets $\{X_i\}_{i=1}^\ell\sqcup \{E_j\}_{j\in J}$ are in general position;
\item $p|_{X\setminus \bigcup_{j\in J} E_j}:X\setminus \bigcup_{j\in J} E_j\to \widetilde{V}$ is biregular.
\item $p(X_i)=V_i$ for every $i\in\{1,\dots,\ell\}$.
\end{enumerate} 
An application of Theorem \ref{thm:Q_tico_tognoli} with the following substitutions: 
``$M$''$:=X$, ``$\ell$''$:=\ell+|J|$,``$M_i$''$:=X_i$ for every $i\in\{1,\dots,\ell\}$, ``$M_j$''$:=E_j$ for every $j\in J$, gives a projectively $\Q$-closed $\Q$-nonsingular $\Q$-algebraic set $X'\subset\R^{2d+1}$ of dimension $d$, $\Q$-nonsingular $\Q$-algebraic subsets $X'_i$ for $i\in\{1,\dots,\ell\}$, and $\Q$-nonsingular $\Q$-algebraic hypersurfaces $E'_j$, for $j\in J$, of $X'$ in general position and a Nash diffeomorphism $\phi:X\to X'$ such that $\phi(X_i)=X'_i$ for every $i\in\{1,\dots,\ell\}$, and $\phi(E_j)=E'_j$, for every $j\in J$.

Consider the Nash map $p':=p\circ\phi^{-1} : X'\to \widetilde{V}$ such that $(p')^{-1}(0)=\bigcup_{j\in J}^\ell E'_j$. Recall that $\bigcup_{j\in J}^\ell E'_j\subset\R^{2d+1}$ is $\Q$-stable since each $E'_j$ is a $\Q$-nonsingular $\Q$-algebraic set for every $j\in J$ and the $E'_j$'s are in general position, thus we can apply \cite[Lemma\,3.2]{GSa} with $``L"=``P":=\bigcup_{j\in J} E'_j$ to each entry of any smooth extension $\R^{2d+1}\to\R^n$ of $p':X'\to \R^n$ getting a polynomial map $q:=(q_1,\dots,q_n):\R^{2d+1}\to\R^n$ such that $q|_{X'}$ is arbitrarily $\Nn_\w$ close to $p'$ and $q_i\in\Ii_\Q(\bigcup_{j\in J} E'_j)$.

Finally, an application of \cite[Lemma\,3.14]{GSa} with the following substitutions: ``$X$''$:=X'$, ``$Y$''$:=\{0\}$, ``$A$''$:=\bigcup_{j\in J} E'_j$, ``$p$''$:=q|_{\bigcup_{j\in J} E'_j}$ and ``$P$''$:=q$ gives a $\Q$-determined $\Q$-algebraic set $\widetilde{V}'\subset\R^{2d+1}\times\R^{n}\times\R$ of dimension $d$ with (eventually) only an isolated singularity at the origin $0$ of $\R^{2d+1}\times\R^{n}\times\R$, a homeomorphism $h:\widetilde{V}'\to X'\cup_p \{0\}$, where $X'\cup_p \{0\}$ denotes the adjunction topological space of $X'$ and $\{0\}$ along $p$, $\Q$-regular maps $f:X'\to \widetilde{V'}$ and $g:\{0\}\to \widetilde{V'}$ such that the following diagram commutes
\[
\centering
\begin{tikzcd}
X' \arrow[dr, hook] \arrow[drrr, bend left=20, "f"] \\
& X'\sqcup \{0\} \arrow[r, "\pi"] & X'\cup_p \{0\} &  \arrow[l, "h"'] \widetilde{V'}\\
\{0\} \arrow[ur, hook] \arrow[urrr, bend right=20, "g"']
\end{tikzcd}
\]
and $\widetilde{V}'_i:=f(X'_i)\cup\{0\}$, for every $i\in\{1,\dots,\ell\}$, is a $\Q$-determined $\Q$-algebraic subset of $\widetilde{V}'$ of codimension $c_i$ with (eventually) only an isolated singularity at the origin $0$ of $\R^{2d+1}\times\R^{n}\times\R$.

Define the semialgebraic homeomorphism $\widetilde{h}:\widetilde{V}\to \widetilde{V}'$ as:

\[
\widetilde{h}(x)=
\begin{cases}
0\quad\quad\quad&\text{if $x=0\in\R^n$,}\\
f\circ\phi\circ p^{-1}(x)&\textnormal{otherwise.}
\end{cases}
\]

Let $m':=2(d+1)+n$. Observe that  $\widetilde{h}|_{\widetilde{V}\setminus\{\overline{0}\}}:\widetilde{V}\setminus\{\overline{0}\}\to \widetilde{V}'\setminus\{\overline{0}\}$ is a Nash diffeomorphism and $\widetilde{h}|_{\widetilde{V}_i}:\widetilde{V}_i \to \widetilde{V}'_i$ is a semialgebraic homeomorphism satisfying the following approximation properties:
\begin{itemize}%[label=\emph{(\roman*)}, ref=(\roman*)]
\item[$(\mr{iv})$] $\iota_{\widetilde{V}'}\circ\widetilde{h}$ is arbitrarily $\mscr{C}^0_\w$ close to $\iota_{\widetilde{V}}$ and $\iota_{\widetilde{V}'}|_{\widetilde{V}'\setminus\{0\}}\circ\widetilde{h}|_{V\setminus\{0\}}$ is arbitrarily $\mscr{C}^\infty_\w$ close to $\iota_{\widetilde{V}}|_{\widetilde{V}\setminus\{\overline{0}\}}$, 
\item[$(\mr{v})$] $\iota_{\widetilde{V}'}|_{\widetilde{V}'_i}\circ\widetilde{h}|_{\widetilde{V}_i}$ is arbitrarily $\mscr{C}^0_\w$ close to $\iota_{\widetilde{V}}|_{\widetilde{V}_i}$ and $\iota_{\widetilde{V}'}|_{\widetilde{V}'_i\setminus\{0\}}\circ\widetilde{h}|_{\widetilde{V}_i\setminus\{0\}}$ is arbitrarily $\mscr{C}^\infty_\w$ close to $\iota_{\widetilde{V}}|_{\widetilde{V}_i\setminus\{\overline{0}\}}$,
\end{itemize}
where $\iota_{\widetilde{V}}:\widetilde{V}\hookrightarrow\R^{m'}$ and $\iota_{\widetilde{V}'}:\widetilde{V}'\hookrightarrow \R^{m'}$ denote the inclusion maps.
%$\widetilde{h}|_{\widetilde{V}\setminus\{\overline{0}\}}:\widetilde{V}\setminus\{\overline{0}\}\to \widetilde{V}'\setminus\{\overline{0}\}$ is a Nash diffeomorphism, $\widetilde{h}|_{\widetilde{V}_i}:\widetilde{V}_i \to \widetilde{V}'_i$ is a semialgebraic homeomorphism, $\widetilde{h}$ is arbitrarily $\mscr{C}^0_\w$ close to the inclusion $\iota_{\widetilde{V}}:\widetilde{V}\hookrightarrow\R^{m'}$ and $\widetilde{h}|_{V\setminus\{0\}}$ is arbitrarily $\mscr{C}^\infty_\w$ close to $\iota_{\widetilde{V}}|_{\widetilde{V}\setminus\{\overline{0}\}}$, and $\widetilde{h}_{\widetilde{V}_i}$ is arbitrarily $\mscr{C}^0_\w$ close to $\iota_{\widetilde{V}}|_{\widetilde{V}_i}$ and $\widetilde{h}_{\widetilde{V}_i\setminus\{0\}}$ is arbitrarily $\mscr{C}^\infty_\w$ close to $\iota_{\widetilde{V}}|_{\widetilde{V}_i\setminus\{\overline{0}\}}$.

Let $t',t'_1,\dots,t'_\ell\in\Q[x_1,\dots,x_{m'}]$ such that $\mathcal{Z}_{\R}(t')=\widetilde{V'}$ and $\mathcal{Z}_{\R}(t'_i)=\widetilde{V'_i}$ for every $i\in\{1,\dots,\ell\}$. Let $\sph^{m-1}$ be the standard unit sphere of $\R^{m'}$ and let $\theta':\R^{m'}\setminus\{0\}\to\R^{m'}\setminus\{0\}$ as $\theta'(x)=\frac{x}{| x |_{m'}^2}$ be the inversion with respect to $\sph^{m'-1}$. Recall that $\theta'\circ\theta'=\id_{\R^{m'}\setminus\{0\}}$. Let $e'> \max\{\deg(t'),\deg(t'_1),\dots,\deg(t'_\ell)\}$. Define the polynomials $s':=| x  |^{2e'}_{m'}\cdot (t'\circ\theta')(x)\in\Q[x]$, $s'_i:=| x  |^{2e'}_{\R^{m'}} (t'_i\circ\theta')(x)\in\Q[x]$, the algebraic sets $V':=\mathcal{Z}_\R(s')$ and $V_i:=\mathcal{Z}_\R(s'_i)$, for every $i\in\{1,\dots,\ell\}$. By construction,
\begin{align*}
V'=\theta'(\widetilde{V}'\setminus\{0\})\cup\{0\}\quad\text{and}\quad V'_i=\theta'(\widetilde{V}'_i\setminus\{0\})\cup\{0\},
\end{align*}
for every $i\in\{1,\dots,\ell\}$, and $\theta':\widetilde{V}'\setminus\{0\}\to V'\cup\{0\}$ is a $\Q$-biregular map between Zariski open subsets of $\Q$-algebraic sets. Moreover, $\theta'(\widetilde{V}'_i\setminus\{0\})=V'_i\setminus\{0\}$ for every $i\in\{1,\dots,\ell\}$. Observe that, by construction, the $\Q$-nonsingular $\Q$-algebraic sets $\{V'_i\}_{i=1}^\ell$ are in general position. Let $C\in\Q\setminus\{0\}$ and define the $\Q$-algebraic sets
\begin{align*}
V''&:=\Big\{(x,y)\in\R^{m'}\times\R\,\Big|\,y\sum_{k=1}^{m'} x_k^2=C,\, s'(x)=0\Big\},\\
V''_i&:=\Big\{(x,y)\in\R^{m'}\times\R\,\Big|\,y\sum_{k=1}^{m'} x_k^2=C,\, s'_i(x)=0\Big\}
\end{align*}
for every $i\in\{1,\dots,\ell\}$. By construction, $V''$ and $V''_i$ are $\Q$-nonsingular $\Q$-algebraic sets, for every $i\in\{1,\dots,\ell\}$, $V'\setminus\{0\}$ and $V''$ are $\Q$-biregularly isomorphic via projection $\pi:\R^{m'}\times\R\to\R^{m'}$, $\pi|_{V''_i}:V''_i\to V'_i\setminus\{0\}$, for every $i\in\{1,\dots,\ell\}$, and the $\Q$-nonsingular $\Q$-algebraic sets $\{V''_i\}_{i=1}^\ell$ are in general position.

Define the Nash diffeomorphism $h:V\to V''$ as
\[
h:=(\pi|_{V''})^{-1}\circ\theta'|_{\widetilde{V}'\setminus\{\overline{0}\}}\circ\widetilde{h}|_{\widetilde{V}\setminus\{\overline{0}\}}\circ\theta|_{V}.
\]
Let $m:=m'+1=2d+n+3$ and fix ``$V'$'':= $V''$ and ``$V'_i$'':= $V''_i$, for every $i\in\{1,\dots,\ell\}$. If we fix $C\in\Q\setminus\{0\}$ be sufficiently small, by $(\mr{iv})$, $(\mr{v})$ and the choice of $\widetilde{h}$ as above, we deduce that $h|_{V_i}:V_i \to V''_i$ is a Nash diffeomorphism, $\jmath\circ h$ is $\mscr{C}^\infty_\w$ close to the inclusion $\iota:V\hookrightarrow \R^{m}$ and $\jmath|_{V'_i}\circ h|_{V_i}$ is $\mscr{C}^\infty_\w$ close to the inclusion map $\iota|_{V_i}:V_i\hookrightarrow \R^{m}$, where $\iota:V\hookrightarrow \R^m$ and $\jmath:V'\hookrightarrow\R^{m}$ denote the inclusion maps. Moreover, an application of \cite{Jel08} provides a semialgebraic homeomorphism $\R^m\to\R^m$ extending $h$, as desired.
\end{proof}

\begin{remark}
In the statement of Theorem \ref{thm:non-compact} we can add the following requirement: `` $V'\subset\R^{m}$ contains a hypersurface of rational points, that is, $\dim(\Zcl_{\R^m}(V'(\Q)))\geq d-1$".

By Remark \ref{rmk:nonsing-rat-pts} we may suppose that $X'\subset\R^{2d+1}$ in the proof of Theorem \ref{thm:non-compact} is such that $\dim(\Zcl_{\R^{2d+1}}(X'(\Q)))\geq n-1$. In addition, since $\Q$-biregular maps send rational points to rational points, as $f|_{X'\setminus(\bigcup_{j\in J} E'_j)}:X'\setminus(\bigcup_{j\in J} E'_j)\to \widetilde{V'}\setminus\{0\}$ and $\theta':\R^m\setminus\{0\}\to \R^m\setminus\{0\}$ are, we get that $\dim(\Zcl_{\R^{2d+1}}(X'(\Q)\setminus(\bigcup_{j\in J} E'_j)))\geq d-1$ and $(\theta'\circ f) (X'(\Q)\setminus(\bigcup_{j\in J} E'_j))=V'(\Q)$, hence, being both $f$ and $\theta'$ biregular, $\dim(\Zcl_{\R^m}(V'(\Q)))\geq d-1$, as desired. %\bs
\end{remark}

\vspace{0.5em}

\begin{remark}
If we are willing to loose the approximation properties in the statement of Theorem \ref{thm:non-compact}, we can find a $\Q$-nonsingular $\Q$-algebraic model $V'$ of $V$ with $\Q$-nonsingular $\Q$-algebraic subsets $V'_i$ with $i=1,\dots,\ell$ with an improvement on the estimate of $m$, namely, we can choose $m=2d+4$. Indeed, it suffices to substitute ``$Y$'':= $\{0\}\subset\R^n$ with ``$Y$'':= $\{0\}\subset\R$ in the application of \cite[Lemma\,3.14]{GSa}.  %\bs
\end{remark}

\section*{Acknowledgments}

The author would like to thank Riccardo Ghiloni and Adam Parusi\'{n}ski for valuable discussions during the drafting process and for their comments, especially concerning the \textsc{$\Q$-algebraicity problem} and the resolution of singularities in characteristic $0$.

\printbibliography

\end{document}